\documentclass[11pt,a4paper]{article}

\usepackage[utf8]{inputenc}
\usepackage[T1]{fontenc}
\usepackage{lmodern}
\usepackage{extpfeil} 

\usepackage[margin=1in]{geometry}

\usepackage{amsmath}
\usepackage{amssymb}
\usepackage{amsthm}
\usepackage{mathtools}
\usepackage{enumitem}
\usepackage{dsfont}

\usepackage{tikz-cd}
\usetikzlibrary{cd}

\usepackage{sectsty}
\sectionfont{\centering}

\theoremstyle{plain}
\newtheorem{theorem}{Theorem}[section]
\newtheorem{lemma}[theorem]{Lemma}
\newtheorem{proposition}[theorem]{Proposition}
\newtheorem{corollary}[theorem]{Corollary}

\theoremstyle{definition}
\newtheorem{definition}[theorem]{Definition}
\newtheorem{example}[theorem]{Example}
\newtheorem{remark}[theorem]{Remark}

\usepackage{graphicx}
\usepackage{caption}
\usepackage{subcaption}
\usepackage{booktabs,tabularx,array}
\newcolumntype{P}[1]{p{#1}} 
\usepackage[colorlinks=true, linkcolor=blue, citecolor=blue, urlcolor=blue]{hyperref}
\usepackage[nameinlink,capitalize]{cleveref}
\newcommand{\R}{\mathbb{R}}

\newcommand{\Longrightleftarrows}[1][1.5em]{%
  \mathrel{\substack{%
      \xrightarrow{\hspace{#1}}\\[-.6ex]%
      \xleftarrow{\hspace{#1}}%
  }}}

\newcommand{\email}[1]{\href{mailto:#1}{#1}}

\title{Log-Euclidean Lie Groups}

\author{Olivier Bisson\thanks{Université Nice Côte d'Azur, INRIA, EPIONE, France
  (\email{olivier.bisson@inria.fr}).}
\and
Xavier Pennec\thanks{Université Nice Côte d'Azur, INRIA, EPIONE, France
  (\email{xavier.pennec@inria.fr}).}
}

\date{}

\begin{document}

\maketitle

\begin{abstract}
We develop a self-contained theory of log-Euclidean Lie groups: smooth manifolds diffeomorphic to finite-dimensional vector spaces, equipped with the pullback of a constant Euclidean metric. This framework encompasses symmetric positive-definite (SPD) matrices $\mathcal{S}^+(n)$ and full-rank correlation matrices $\mathrm{Cor}^+(n)$, and explains why many seemingly different log-Euclidean constructions yield the same Riemannian geometry. We provide explicit Riemannian isometries (and Lie group isomorphisms) linking several log-Euclidean metrics on SPD and correlation matrix manifolds, and we characterize quotient log-Euclidean metrics in a principal-bundle setting. Finally, using the diagonal correction map underlying the off-log parametrization, we construct an explicit log-Euclidean metric on $\mathcal{S}^+(n)$ for which the standard inclusion $i:\mathrm{Cor}^+(n)\hookrightarrow \mathcal{S}^+(n)$ becomes an isometric (indeed, totally geodesic) embedding, yielding closed-form formulas for geodesics and orthogonal decompositions in adapted coordinates. The nested isometric embeddings constructed here also provide a simple solution to the comparison of matrices of different dimensions in the log-Euclidean setting: SPD or correlation matrices may be transported to a common dimension via explicit maps while preserving all intrinsic distances.
\end{abstract}

\tableofcontents

%\tableofcontents

\section{Introduction}

Let $\phi: M\to V$ be a diffeomorphism onto a finite-dimensional real vector space $V$ endowed with an inner product $\langle\cdot,\cdot\rangle$. Via the canonical identification $T_XV\simeq V$, $X \in V$, $\langle\cdot,\cdot\rangle$ defines a constant Riemannian metric on $V$, and its pullback $g:=\phi^{*}\langle\cdot,\cdot\rangle$ is a flat Riemannian metric on $M$ which we call a \emph{log-Euclidean metric}. Similarly, the pullback via $\phi$ of the additive structure on $V$ defines a commutative Lie group operation on $M$, i.e. for $u,v \in M$, 
\begin{align*}
    u \star v &= \phi^{-1}\left( \phi(u) + \phi(v) \right),\\
    u_\star^{-1} &= \phi^{-1}(-\phi(u)).
\end{align*}

\medskip

Then $(M, \star)$ is called a \emph{log-Euclidean Lie group}. We show in Section \ref{LELG_sec} that log-Euclidean metrics on $M$ are the bi-invariant metrics of $M$, and vice versa. Moreover, for each $X\in V$ the curve $\gamma_X(t):=\phi^{-1}(tX)$ is a one-parameter subgroup of $(M,\star)$; thus $\exp_M(X)=\phi^{-1}(X)$ is the Lie exponential map, and we may identify $V$ with the Lie algebra $\mathfrak g=\mathcal{T}_{e_M}M$ of $M$ (where $e_M=\phi^{-1}(0)$). It follows that $(M,g)$ is a complete, simply connected, flat Riemannian manifold in which any two points are joined by a unique geodesic, given by the Riemannian exponential map, which in this case coincides with the Lie exponential map.

\medskip

On the applied side, symmetric positive-definite (covariance, SPD, $\mathcal{S}^+$) matrices, diffusion tensors and full-rank correlation matrices ($\mathrm{Cor}^+$), all live on smooth matrix manifolds that lack a globally defined intrinsic multiplication. Yet in many applications (for instance signal processing, diffusion MRI, cf:\cite{arsigny2006, shahbazi2021, you2021}) the use of geometric statistics involves statistical operations such as averaging, interpolation, or principal component analysis which each require algebraic operations (\emph{multiplication}, \emph{inversion}) that are not intrinsically available on these manifolds \cite{xu2012}. A necessary condition to have a Lie group structure is parallelizability: every Lie group is parallelizable because left translations carry local triviality at the identity to the entire group. Parallelizability alone, however, is not sufficient for the existence of a Lie group structure (e.g. the $7$-sphere). For SPD matrices, several group operations have been derived (see for instance \cite{lin2019, arsigny2007}), and, to our knowledge, Lie group structures on full-rank correlation matrices are developed only in dimension $2$ in \cite{david2022}, even if several diffeomorphisms to Euclidean spaces were derived, as in \cite{archakov2021, thanwerdas2024}.

\medskip

We follow \cite{arsigny2007} and investigate the widely used case whenever a matrix manifold $M$ is diffeomorphic to some finite-dimensional vector space $V$. In fact, any Euclidean parametrization of a smooth manifold leads to a natural log-Euclidean Lie group structure on it. While the theory of connected, simply connected commutative Lie groups is classical (see for instance, textbooks such as \cite{doCarmo1992, berger2007, warner2013, gallier2020}), the
log-Euclidean viewpoint has numerous advantages for geometric
statistics, among them:\\
\begin{enumerate}
    \item it furnishes explicit closed-form formulae for geodesics, means,
    parallel transports;
    \item it extends seamlessly to matrix manifolds of practical interest
    (e.g.\ $\mathcal S^{+}(n)$ or $\mathrm{Cor}^{+}(n)$) where classical
    matrix multiplication and inversion may not be compatible;
    \item it can greatly reduce computational costs of statistical          operations, once transported in the model Euclidean space (see \cite{arsigny2006} where the authors compare the log-Euclidean and affine-invariant metrics (\cite{pennec2006}) on SPD matrices).
\end{enumerate}

\medskip

Therefore, there is a need to study in depth log-Euclidean Riemannian structures on smooth manifolds, and to understand how different such structures compare. Indeed, several log-Euclidean metrics have been introduced for covariance and full-rank correlation matrices, but it is not obvious how they are connected, especially when no closed form is available. In \cite{bisson2025} the authors observe astonishingly similar results for different log-Euclidean metrics on SPD matrices and full-rank correlation matrices. This raises the following question:
\\
\begin{quote}
  Given two log-Euclidean Riemannian metrics on a smooth manifold, how are they related?\\
\end{quote}

In Section \ref{isom_iso_sec}, we answer that question using the log-Euclidean Lie group viewpoint. We give a \emph{constructive} proof that any two log-Euclidean metrics on finite-dimensional manifolds of the same dimension are \emph{isomorphically isometric}: a linear isometry (hence, an isomorphism) of their Lie algebras integrates to a global Lie group isometry. Because log-Euclidean Lie groups are modelled over finite dimensional inner product vector spaces of same dimension, such an isometry always exists. This explains that Fréchet mean, geodesic interpolation and extrapolation lead to similar results despite different Euclidean parametrizations \cite{bisson2025}. In particular, we provide explicit \emph{isomorphic isometries} between recent log-Euclidean metrics on full-rank correlation matrices. Even more striking, our method allows us to derive an isometry between $\mathcal{S}^+(n-1)$ equipped with the standard log-Euclidean metric and $\mathrm{Cor}^+(n)$ equipped with recently derived log-Euclidean metrics (\cite{thanwerdas2024}) given by the following flat diffeomorphic parametrizations:\\

\begin{enumerate}
\item the \emph{log} map  
      $\mathrm{log}\colon\mathcal{S}^+(n-1)\!\longrightarrow\!\mathcal{S}(n-1)$, obtained by taking the matrix logarithm;
\item the \emph{off-log} map  
      $\mathrm{Log}\colon\mathrm{Cor}^+(n)\!\longrightarrow\!\mathrm{Hol}(n)$, obtained by zeroing the diagonal after taking the matrix logarithm;
\item the \emph{log-scaling} map  
      $\mathrm{Log}^{\bullet}\colon\mathrm{Cor}^+(n)\!\longrightarrow\!\mathrm{Row}_0(n)$, obtained by optimal diagonal scaling prior to the logarithm.
\end{enumerate}

\medskip

In other words, the three Riemannian manifolds have identical Riemannian geometry. Consequently, any construction that depends only on the metric (and the associated Levi-Civita connection) is \emph{equivariant} under these isometries. In practice, this means that geometric procedures such as the computation of Fréchet means, geodesic interpolation and extrapolation, or critical points of Riemannian gradient flows produce results that correspond exactly under the explicit isometries we construct. By contrast, algorithms that use additional structure beyond the metric---for instance, a specific coordinate chart, an extrinsic embedding, or a sparsity constraint in a chosen parametrization---need not agree across isometric models. A typical example is a $k$-means procedure applied in a specific basis of a Euclidean space. It is therefore important in applications to distinguish carefully between genuinely intrinsic, geometry-based operations and methods that depend on extra modelling choices tied to a particular parametrization.

\medskip

In Section \ref{quotient_sec} we then investigate log-Euclidean Lie groups in the classical principal-bundle framework (for references on this subject, see for instance \cite{kobayashi1963, steenrod1999}). Let $G$ be a log-Euclidean Lie group and let $H\subset G$ be a connected closed subgroup (necessarily normal, since $G$ is abelian). Then $G/H$ is itself a log-Euclidean Lie group and the quotient map
\begin{align*}
    \pi : G \longrightarrow G/H
\end{align*}
is a \emph{trivial} principal $H$-bundle, so that $G \cong G/H \times H$.
Endowing $G/H$ with the quotient metric turns $\pi$ into a Riemannian submersion \cite{oneill1966}. Because a global section of $\pi$ provides a parametrization of $G/H$, we answer: \\

\begin{enumerate}
  \item which sections themselves carry a log-Euclidean
        Lie group structure;
  \item which sections are \emph{isometric embeddings} in $G$;
  \item which sections enjoy \emph{both} properties ?
\end{enumerate}

\medskip

It turns out that there is a \emph{unique} section satisfying both $1$ and $2$; we call it the \emph{canonical section} of $\pi$, which provides a Euclidean parametrization of $G/H$. Any section that is a \emph{group morphism} (hence a subgroup of $G$) corresponds to a choice of supplementary subspace $\mathfrak m$ to the Lie algebra of $H$, $\mathfrak h$ in the Lie algebra of $G$, $\mathfrak g$. These types of sections are transverse sections and identify $G$ with the internal direct product $G = K \times H$, where $K$ is the image of the section. Moreover, we prove that any section that is an isometric embedding must also be horizontal and is therefore unique up to vertical translation by an element of $H$. Consequently, only one section enjoys both properties, and this canonical section not only realizes the splitting $G = K \times H$ but also does so metrically: the log-Euclidean metric on $G$ decomposes orthogonally as the product metric of $K$ and $H$. Even though the proofs may appear technical, the theoretical results are simple to interpret and we apply them to SPD and full-rank correlation matrices.

\medskip

In the last section of this paper, we recast the geometry of fixed finite $n$-dimensional manifolds $\mathcal{S}^+(n)$ and $\mathrm{Cor}^+(n)$ through a principal bundle point of view. Starting from $\mathcal S^{+}(n)$ and the normal subgroup of positive diagonal matrices $\mathrm{Diag}^{+}(n)$, the quotient map $\pi:\mathcal S^{+}(n) \rightarrow \mathcal S^{+}(n)/\mathrm{Diag}^{+}(n)$ is a Riemannian submersion and we derive its canonical section, being simultaneously log-Euclidean and an isometric embedding; on this section, the quotient and induced metrics coincide. We then show that, when the quotient $\mathcal S^{+}(n)/\operatorname{Diag}^{+}(n)$ is parametrized by $\operatorname{Cor}^{+}(n)$, the pullback of the quotient metric along this parametrization coincides with the off-log metric on $\operatorname{Cor}^+(n)$; as a consequence, the log-scaling metric on $\mathrm{Cor}^{+}(n)$ is obtained by pulling back this quotient metric via the isometry constructed in Section \ref{isom_iso_sec}. By contrast, the usual \emph{correlation section} is not horizontal: the induced log-Euclidean metric along it introduces a non-zero vertical correction term given by the diagonal part of the metric, so the induced metric on $\mathrm{Cor}^{+}(n)$ differs from the quotient one. This shows that, for the standard ambient log-Euclidean metric on $\mathcal{S}^+(n)$, the induced metric on $\operatorname{Cor}^+(n)$ differs from the intrinsic off-log metric by an explicit vertical term. Nevertheless, using the diagonal correction map underlying the off-log parametrization, we also construct an explicit log-Euclidean metric on $\mathcal{S}^+(n)$ for which the standard inclusion $i : \operatorname{Cor}^+(n)\hookrightarrow \mathcal{S}^+(n)$ becomes a totally geodesic isometric embedding. Table \ref{table:notation} sums up the spaces, their model vector-spaces and associated log-Euclidean metrics used throughout this article. 

\begin{table}[ht]
\small
\centering
\resizebox{\textwidth}{!}{%
\begin{tabular}{@{}l c l P{6.8cm} P{6.8cm}@{}}
\toprule
\textbf{Space} & \textbf{Dim} & \textbf{Model space} & \textbf{Log–Euclidean map } & \textbf{Log-Euclidean Metric} \\
\midrule
$\mathcal{S}^+(n)$ (log-Euclidean)
& $\tfrac{n(n+1)}{2}$
& $\mathcal{S}(n)$
& $\log$ (matrix logarithm)
& $g^{\mathrm{LE}}=\log^{*}\langle\cdot,\cdot\rangle$ \\[4pt]

$\mathrm{Cor}^+(n)$ (off-log)
& $\tfrac{n(n-1)}{2}$
& $\mathrm{Hol}(n)$
& $\mathrm{Log}=\mathrm{off}\circ\log$
& $g^{\mathrm{OL}}=\mathrm{Log}^{*}\langle\cdot,\cdot\rangle$ \\[4pt]

$\mathrm{Cor}^+(n)$ (log-scaling)
& $\tfrac{n(n-1)}{2}$
& $\mathrm{Row}_0(n)$
& $\mathrm{Log}^{\bullet}=\log(\mathcal{D}^*(C)C\mathcal{D}^*(C))$
& $g^{\mathrm{LS}}=(\mathrm{Log}^{\bullet})^{*}\langle\cdot,\cdot\rangle$ \\[4pt]

$\mathcal{S}^+(n)/\mathrm{Diag}^+(n)$
& $\tfrac{n(n-1)}{2}$
& $\mathcal{S}(n)/\mathrm{Diag}(n)\cong\mathrm{Hol}(n)$
& $\mathrm{off} \circ \; \overline{\log}:[\Sigma]\mapsto \mathrm{off} \left(\log\Sigma+\mathrm{Diag}(n) \right)$
& $g^{Q}=g^{\overline{\mathrm{OL}}}=(\mathrm{off}\circ \overline{\log})^{*}\langle\cdot,\cdot\rangle$ \\[4pt]

$\mathrm{Diag}^+(n)$
& $n$
& $\mathrm{Diag}(n)$
& $\log$ (entrywise on the diagonal)
& $g^\mathrm{DL} = \log^{*}\langle\cdot,\cdot\rangle$ \\
\midrule
\multicolumn{5}{@{}l}{\textbf{Model vector space}}\\
\midrule
$\mathcal{S}(n)$ & $\tfrac{n(n+1)}{2}$ &  & symmetric matrices & model for $\mathcal{S}^+(n)$ \\
$\mathrm{Hol}(n)$ & $\tfrac{n(n-1)}{2}$ &  & symmetric matrices with zero diagonal & model for $\mathrm{Cor}^+(n)$ and $\mathcal{S}^+(n)/\mathrm{Diag}^+(n)$ \\
$\mathrm{Row}_0(n)$ & $\tfrac{n(n-1)}{2}$ &  & symmetric matrices with zero row-sum & model for $\mathrm{Cor}^+(n)$ \\
$\mathrm{Diag}(n)$ & $n$ &  & diagonal matrices & model for $\mathrm{Diag}^+(n)$ \\
\bottomrule
\end{tabular}}
\caption{Manifolds, model vector spaces, and associated log-Euclidean metrics.}
\label{table:notation}
\end{table}

A byproduct of our explicit isometries is a family of \emph{dimension-raising} totally geodesic isometric embeddings that relates log-Euclidean geometries across matrix sizes. On SPD matrices with the usual log-Euclidean metric (Theorem~\ref{thm:SPD_hierarchy}) we obtain a nested sequence of isometric embeddings
\[
\cdots \ \overset{F_{n-1}}{\hookrightarrow}\ \mathcal{S}^+(n-1)\ \overset{F_n}{\hookrightarrow}\ \mathcal{S}^+(n)\ \overset{F_{n+1}}{\hookrightarrow}\ \mathcal{S}^+(n+1)\ \overset{F_{n+2}}{\hookrightarrow}\ \cdots ,
\]
and, in parallel, on full-rank correlation matrices with the off-log metric (Theorem~\ref{thm:Cor_hierarchy}) we obtain
\[
\cdots \ \overset{E_{n-1}}{\hookrightarrow}\ \mathrm{Cor}^+(n-1)\ \overset{E_n}{\hookrightarrow}\ \mathrm{Cor}^+(n)\ \overset{E_{n+1}}{\hookrightarrow}\ \mathrm{Cor}^+(n+1)\ \overset{E_{n+2}}{\hookrightarrow}\ \cdots .
\]
Besides their intrinsic geometric interest, these hierarchies provide a principled way to \emph{compare} or \emph{transport} data represented by SPD or full-rank correlation matrices across different ambient dimensions, while preserving the relevant log-Euclidean (resp.\ off-log) distances, geodesics and first order statistics. While a method to compare covariances matrices of different dimensions, equipped with the affine-invariant metric \cite{pennec2006}, was given in \cite{LimSepulchreYe2019}, the nested isometric embeddings constructed here provide a simple solution to this problem in the log-Euclidean setting, and applies to full-rank correlation matrices of different dimensions.

\subsubsection*{Acknowledgements} This work was supported by ERC grant \#786854 \textit{G-Statistics} from the European Research Council under the European Union’s Horizon 2020 research and innovation program, and by the French government through the \textit{3IA Côte d’Azur} Investments ANR-23-IACL-0001 managed by the National Research Agency (ANR).

\section{Log-Euclidean Lie groups}
\label{LELG_sec}

\subsection{Preliminaries}
In this section, we review the foundations for the study of finite-dimensional log-Euclidean Lie groups, following the general introduction to Lie groups given in \cite[Ch.~3]{warner2013}, \cite{gallier2020} and \cite{procesi2007}. Our goal is both to recall the basic definitions and to fix the notation used in the rest of the paper.

\begin{definition}[Lie groups]
A \emph{Lie group} $G$ is a differentiable manifold endowed with a group structure such that the map $G \times G \rightarrow G$ defined by $(\sigma, \tau) \mapsto \sigma \tau^{-1}$ is $C^{\infty}$.
\end{definition}

\begin{definition}[Lie algebra]
A \emph{Lie algebra} $\mathfrak{g}$ over $\mathbb{R}$ is a real vector space $\mathfrak{g}$ together with a bilinear operator (called the \emph{bracket}) $\left[\cdot, \cdot \right] : \mathfrak{g} \times \mathfrak{g} \rightarrow \mathfrak{g}$ such that for all $x, y , z \in \mathfrak{g}:$
\begin{enumerate}[label=(\alph*)]
    \item $\left[x, y \right] = -\left[y, x \right]$ \hfill (anti-commutativity)
    \item $\left[\left[x, y \right], z\right] + \left[\left[y, z \right], x\right] + \left[\left[z, x \right], y\right] = 0$ \hfill (Jacobi identity)
\end{enumerate}
\end{definition}
The vector space of all smooth vector fields on the manifold $G$ forms
a Lie algebra under the Lie bracket operation on vector fields, and we define the \emph{Lie algebra of the Lie group $G$} to be the Lie algebra $\mathrm{Lie}(G) = \mathfrak{g}$ of left-invariant vector fields on $G$. Because each left-invariant vector field is uniquely determined by its value at the identity, the Lie algebra $\mathfrak{g}$ of $G$ can be defined as the tangent space at the identity $e_G$ of $G$, $\mathfrak{g} \simeq \mathcal{T}_{e_G}G$.

\begin{theorem}
Let $G$ and $H$ be Lie groups with Lie algebras $\mathfrak g$ and $\mathfrak h$ respectively, and let $\phi\colon G \to H$ be a Lie group homomorphism.  Then:
\begin{enumerate}[label=\emph{(\alph*)}]
  \item $X$ and $\mathrm d\phi(X)$ are $\phi$-related for each $X\in\mathfrak g$.
  \item $\mathrm d_{e_G}\phi:\mathfrak g \to \mathfrak h$ is a Lie algebra homomorphism.
\end{enumerate}
\end{theorem}

Connected, simply connected Lie groups are completely determined up to isomorphism by their Lie algebras.

\begin{theorem}
\label{thm:Lie_second}
Let $G$ and $H$ be Lie groups with Lie algebras $\mathfrak g$ and $\mathfrak h$, respectively, and assume $G$ is simply connected.  Let $\Phi : \mathfrak g\to\mathfrak h$ be a Lie‐algebra homomorphism.  Then there exists a unique group homomorphism $\phi : G\to H$ such that $\mathrm d\phi=\Phi$.
\end{theorem}

\begin{corollary}
If simply connected Lie groups $G$ and $H$ have isomorphic Lie algebras $\mathfrak{g}$ and $\mathfrak{h}$, then $G$ and $H$ are isomorphic.
\end{corollary}

\begin{theorem}
There is a one-to-one correspondence between isomorphism classes of Lie algebras and isomorphism classes of simply connected Lie groups.
\end{theorem}

\begin{definition}[$1$-parameter subgroup]
    A homomorphism $\varphi : \mathbb{R} \rightarrow G$ is called a \emph{$1$-parameter subgroup} of $G$.
\end{definition}

\begin{definition}[Lie exponential map]
The \emph{exponential} of $X\in\mathfrak{g}$ is given by $\exp_G(X)\;=\;\gamma(1)$,
where $\gamma : \mathbb{R}\to G$ is the unique $1$‐parameter subgroup of $G$ whose tangent vector at the identity equals $X$.
\end{definition}

\begin{theorem}
\label{comm_diag}
Let $\phi : G \rightarrow H$ be a homomorphism of Lie groups with $\mathfrak g$ and $\mathfrak h$ their Lie algebras, respectively. Then, the following diagram is commutative
\begin{center}
    \begin{tikzcd}
    G \arrow[rr, "\phi"]                                         &  & H                                 \\
                                                                 &  &                                   \\
    \mathfrak g \arrow[uu, "\exp_G"] \arrow[rr, "\mathrm d\phi"] &  & \mathfrak h \arrow[uu, "\exp_H"']
    \end{tikzcd}
\end{center}
\end{theorem}

The \emph{adjoint representation} $\mathrm{Ad} : G \rightarrow \mathrm{GL}(\mathfrak{g})$ of the Lie group $G$ is the map defined such that $\mathrm{Ad}_a : \mathfrak{g} \rightarrow \mathfrak{g}$ yields the linear isomorphism
\begin{align*}
    \mathrm{Ad}_a = \mathrm d_{e_G}(\mathrm{ad}_a) = \mathrm d_{e_G}(\mathrm R_{a^{-1}} \circ \mathrm L_a) = \mathrm d_a(\mathrm R_{a^{-1}}) \circ \mathrm d_{e_G}(\mathrm L_a), \quad \forall a \in G
\end{align*}
where $\mathrm R_{a^{-1}}$ and $\mathrm L_a$ represents right and left translations by $a^{-1}$ and $a$, respectively.

\begin{definition}[$\mathrm{Ad}$-invariance]
Given a Lie group $G$ with Lie algebra $\mathfrak{g}$, we say that an inner product $\langle \cdot, \cdot \rangle$ is \emph{$\mathrm{Ad}$-invariant} if
\begin{align*}
    \langle \mathrm{Ad}_a x, \mathrm{Ad}_a y \rangle = \langle x, y \rangle
\end{align*}
for all $a \in G$ and all $x,y \in \mathfrak{g}$.
\end{definition}

It is well known (see for instance for proofs of this result \cite{warner2013, gallier2020, procesi2007}) that there is a bijective correspondence between bi-invariant metric on a Lie group $G$ and $\mathrm{Ad}$-invariant inner products on its Lie algebra $\mathfrak{g}$. Furthermore, the next lemma classifies all connected Lie groups with bi-invariant metric.

\begin{lemma}[\cite{milnor1976}]
The connected Lie group $G$ admits a bi-invariant metric if and only if it is isomorphic to the cartesian product of a compact group and an additive vector group.
\end{lemma}

\begin{theorem}
A connected commutative Lie group $G$ is isomorphic to a product $\mathbb{R}^m \times \left( \mathbb{S}^1 \right)^h$.
\end{theorem}

If we further require that the connected Lie group $G$ is commutative and simply connected, then we can deduce from the above theorem that $G$ must be isomorphic to $\mathbb{R}^m$. The isomorphism is precisely realized by the Lie exponential map. Indeed, since $G$ is commutative the Lie bracket $\left[ \cdot, \cdot \right]$ in $\mathfrak{g}$ is trivial and the Baker–Campbell–Hausdorff formula shows that $\mathrm{exp}_G : \mathfrak{g} \rightarrow G$ is a group homomorphism. Moreover one can show that $\mathrm{exp}_G : \mathfrak{g} \rightarrow G/ \mathrm{ker}(\mathrm{exp}_G)$ is a covering map, and since $G$ is simply connected it must be its universal cover and thus, in particular a homeomorphism which implies bijectivity. In that specific context, we call the \emph{Lie logarithm} the bijective inverse of the Lie exponential, and denote it $\mathrm{log}_G : G \rightarrow \mathfrak{g}$.

\subsection{Definitions and Elementary Properties}

Let us now define the central object of this paper, log-Euclidean Lie groups, which arise precisely when a smooth manifold $M$ is diffeomorphic to some finite-dimensional vector space $V$.

\begin{definition}[log-Euclidean Lie groups]
Let $M$ be a smooth manifold diffeomorphic through a map $\phi$ to some $n$-dimensional vector space $V$. Define a group operation "$\star$" on $M$ by pullback via $\phi$ the additive structure on $V$, i.e. for $u,v \in M$, 
\begin{align*}
    u \star v := \phi^{-1}\left( \phi(u) + \phi(v) \right).
\end{align*}
Then $(M, \star)$ is called a \emph{log-Euclidean Lie group}, often denoted $\phi : M \rightarrow V$.
\end{definition}

\begin{proposition}
Log-Euclidean Lie groups are commutative Lie groups.
\end{proposition}

\begin{proof}
The proof follows directly from the fact that $(V, +)$ is a commutative and additive Lie group. Nevertheless, we make explicit the computations. The identity element of $M$ is given by $e_M = \phi^{-1}(0)$ where $0$ is the identity element of $V$. The inverse of any element $u \in M$ is given by $u_{\star}^{-1} := \phi^{-1}(- \phi(u))$, indeed:
\begin{align*}
    u \star u_{\star}^{-1} &= \phi^{-1}(\phi(u) + \phi(\phi^{-1}(-\phi(u)))) \\
    &= \phi^{-1}(\phi(u) - \phi(u)) = \phi^{-1}(0) = e_M.
\end{align*}
Additivity is also seen easily, for $u, v, w \in M$:
\begin{align*}
    (u \star v) \star w &= \phi^{-1}(\phi(u \star v) + \phi(w)) \\
    &= \phi^{-1}(\phi(u) + \phi(v) + \phi(w)) = \phi^{-1}(\phi(u) + \phi(v \star w)) = u \star (v \star w).
\end{align*}
Clearly, this group action is commutative, and smoothness follows by composition.
\end{proof}

Since every log-Euclidean Lie groups $\phi : M \rightarrow V$ are connected, simply connected, the Lie exponential is a global diffeomorphism to their Lie algebras. Next proposition shows that we can naturally identify their Lie algebras with $V$.

\begin{proposition}
Let $\phi : M \rightarrow V$ be a log-Euclidean Lie group, then the Lie exponential of $(M, \star)$ is the inverse diffeomorphism $\phi^{-1}$.
\end{proposition}

\begin{proof}
Let $S \in V$, $s, t \in \mathbb{R}$ and define $\gamma_S(t) := \phi^{-1}(tS)$ for all $t$. Then $\gamma_S(0) = \phi^{-1}(0) = e_M \in (M, \star)$ and
\begin{align*}
    \gamma_S(t+s) &= \phi^{-1}\left( (t+s)S \right) \\
    &= \phi^{-1}\left( \phi(\phi^{-1}(tS)) + \phi(\phi^{-1}(sS))\right) \\
    &= (\phi^{-1}(tS)) \star (\phi^{-1}(sS)) = \gamma_S(t) \star \gamma_S(s).
\end{align*}
Thus, $\gamma_S$ is a $1$-parameter subgroup of $G$.
\end{proof}

Let us now equip $M$ with a Riemannian metric $g$ defined as the pullback via $\phi$ of an inner product $\langle \cdot, \cdot \rangle$ in $V \simeq \mathrm{Lie}(M)$, i.e., for all $p \in M$ and all $\delta_p, \xi_p \in \mathcal{T}_pM$:
\begin{align*}
    g_p\left( \delta_p, \xi_p \right) := \langle \mathrm d_p\phi(\delta_p), \mathrm d_p\phi(\xi_p)\rangle. 
\end{align*}
Then, $g$ is a bi-invariant metric on $M$ and conversely, every such bi-invariant metric on $M$ arises this way.

\begin{proposition}
A Riemannian metric on $\phi : M \rightarrow V$ is bi-invariant if and only if it is the pullback by $\phi$ of some inner product on $V$.
\end{proposition}

\begin{proof}
There is a bijective correspondence between bi-invariant metrics on a Lie group and $\mathrm{Ad}$-invariant inner products on its Lie algebra. Since the group is commutative, its adjoint action is the identity, and since $V = \mathrm{Lie}(M)$, every inner products on $V$ pulls-back to a bi-invariant metric on $M$, and conversely.
\end{proof}

\begin{remark}[Pulling back a metric through a diffeomorphism]
Let $\phi: M\rightarrow V$ be a diffeomorphism onto a finite-dimensional real vector space $V$. Choose an inner product $\langle\cdot,\cdot\rangle_V$ on $V$. Via the canonical identification $\mathcal{T}_vV\simeq V$ for all $v \in V$, this determines a constant Riemannian metric $g^V$ on $V$ by $g^V_{v}(u,w)=\langle u,w\rangle_V$ for all $u, w \in \mathcal{T}_vV$. The pullback
\begin{align*}
    g^M:=\phi^*g^V,\qquad 
    g^M_x(\xi,\eta)=\big\langle \mathrm d_x\phi(\xi),\; \mathrm d_x\phi(\eta) \big\rangle_V, \quad \forall \;  \xi,\eta \in \mathcal{T}_xM, \; x \in M
\end{align*}
is a Riemannian metric on $M$, and $\phi:(M,g^M)\to (V,g^V)$ is an isometry.
Note that what is being pulled back is the metric on tangent spaces $\mathcal{T}_vV$ (identified with $V$), not an inner product on the set $V$ itself. Hence, in all that follows, we identify $V$ with its tangent spaces.
\end{remark}

It is a classical fact (see for instance \cite{doCarmo1992}) that given a bi-invariant Riemannian metric, the Lie exponential map and the Riemannian exponential map coincide, i.e., geodesics of the Levi-Civita connection are integral curves of left-invariant vector fields. Hence the following properties.

\begin{proposition}\label{prop:LE_geodesics}
Let $\phi : M \rightarrow V$ be a log-Euclidean Lie group and let $g=\phi^{*}\langle\cdot,\cdot\rangle$ be the associated log-Euclidean metric (equivalently, a bi-invariant metric for $\star$). Let $p_1,p_2\in M$. Then
\begin{enumerate}[label=\emph{(\alph*)}]
    \item the geodesics of $(M,g)$ are exactly the curves $\gamma(t)=\phi^{-1}(a+t v)$ with $a,v\in V$; in particular the geodesics through $e_M=\phi^{-1}(0)$ are the one-parameter subgroups $\gamma(t)=\phi^{-1}(t v)$;
    \item the unique geodesic connecting $p_1$ to $p_2$ is $\gamma(t)=\phi^{-1}\!\big((1-t)\phi(p_1)+t\phi(p_2)\big)$;
    \item the geodesic distance between $p_1$ and $p_2$ is $\|\phi(p_1)-\phi(p_2)\|_V$;
    \item the Levi-Civita connection of $(M,g)$ is flat;
    \item the Riemannian curvature of $(M,g)$ is identically zero.
\end{enumerate}
\end{proposition}

\begin{proof}
By definition, $\phi:(M,g)\rightarrow (V,\langle\cdot,\cdot\rangle)$ is a global Riemannian isometry. Therefore $\phi$ sends geodesics to straight lines and preserves their lengths and distances. Items \emph{(a)}--\emph{(c)} follow immediately by pulling back affine lines in $V$. Since $(V,\langle\cdot,\cdot\rangle)$ is Euclidean, its Levi-Civita connection and curvature are flat/zero, and the same holds on $(M,g)$ by invariance under isometry, yielding \emph{(d)}--\emph{(e)}.
\end{proof}

The following lemma will be particularly useful for many of the computations carried out later in the paper.

\begin{lemma}
\label{lem:LE-dlog-dexp}
Let $G$ be a log-Euclidean Lie group with Lie algebra $\mathfrak g$, Lie exponential $\exp_G$ and Lie logarithm $\log_G$. Then for every $p \in G$ and $X \in \mathfrak g$ with $p = \exp_G(X)$ one has
\begin{align*}
\mathrm d_p \log_G 
&= (L_{p^{-1}})_* : \mathcal{T}_p G \longrightarrow \mathcal{T}_{e_G} G \simeq \mathfrak g,
&
\mathrm d_{e_G} \log_G &= \operatorname{Id}_{\mathcal{T}_{e_G}G},\\
\mathrm d_X \exp_G 
&= (L_p)_* : \mathfrak g \longrightarrow \mathcal{T}_p G \simeq (L_p)_* \mathfrak g,
&
\mathrm d_0 \exp_G &= \operatorname{Id}_{\mathfrak g},
\end{align*}
where $L_p : G \to G$ denotes left translation by $p$.
\end{lemma}

\begin{proof}
The proof follows directly from the classical formula for the differential of the Lie exponential in the connected, simply connected case (see, e.g., \cite[Ch.~II]{helgason1979}, \cite{taniguchi1984}, \cite{rossmann2002}):
\begin{align*}
\mathrm d_X \exp_G &= (L_p)_{*} \circ \frac{1 - e^{-\operatorname{ad}_X}}{\operatorname{ad}_X},\\
\frac{1 - e^{-\operatorname{ad}_X}}{\operatorname{ad}_X}
&= \sum_{n=0}^\infty \frac{(-1)^n}{(n+1)!}\,\operatorname{ad}_X^n.
\end{align*}
Since $G$ is log-Euclidean, it is abelian, so $\operatorname{ad}_X = 0$ for all $X \in \mathfrak g$, hence $\frac{1 - e^{-\operatorname{ad}_X}}{\operatorname{ad}_X} = \operatorname{Id}_{\mathfrak g}$ and therefore
\begin{align*}
\mathrm d_X \exp_G = (L_{\exp_G(X)})_* = (L_p)_*.
\end{align*}
Inverting gives $\mathrm d_p \log_G = (L_{p^{-1}})_*$, and evaluating at $X = 0$ and $p = e_G$ one concludes.
\end{proof}

Log-Euclidean Lie groups are designed to be a tool for geometric statistics, and as such, we provide some basic statistical definitions and properties they share.

\begin{definition}[Log-Euclidean mean, \cite{arsigny2007}]
Let $\left\{ g_i \right\}_{i=1}^N \subset M$ be samples of the log-Euclidean Lie group $\phi : M \rightarrow V$. Define its \emph{log-Euclidean mean} as
\begin{align*}
    \Bar{g} := \phi^{-1}\left(\frac{1}{N}\sum_{i=1}^N \phi(g_i)\right).
\end{align*}
\end{definition}
In a similar fashion, we define the notion of log-Euclidean covariance.
\begin{definition}[Log-Euclidean variance]
Let $\left\{ g_i \right\}_{i=1}^N \subset M$ be samples of the log-Euclidean Lie group $\phi : M \rightarrow V$. Define its \emph{log-Euclidean variance} as
\begin{align*}
    \mathrm{Var}\left( g_1, \ldots, g_N \right) = \frac{1}{N}\sum_{i=1}^N \Vert \phi(g_i) - \phi(\Bar{g}) \Vert_V^2.
\end{align*}
\end{definition}

\begin{definition}[Geometric covariance and correlation matrices]
\label{geom_cov_corr}
Let $\left\{ (g_i, h_i) \right\}_{i=1}^N \subset M \times M$ be paired samples of two random variables $g$ and $h$ of the log-Euclidean Lie group $\phi : M \rightarrow V$. Let $\Bar{g}$ and $\Bar{h}$ denote their log-Euclidean means, respectively. Define their \emph{empirical geometric covariance matrices} as
\begin{align*}
    \Sigma_{gg} &= \frac{1}{N}\sum_{i=1}^N \left( \phi(g_i) - \phi(\Bar{g}) \right) \left( \phi(g_i) - \phi(\Bar{g}) \right)^\top,\\
    \Sigma_{hh} &= \frac{1}{N}\sum_{i=1}^N \left( \phi(h_i) - \phi(\Bar{h}) \right) \left( \phi(h_i) - \phi(\Bar{h}) \right)^\top,\\
    \Sigma_{gh} &= \frac{1}{N}\sum_{i=1}^N \left( \phi(g_i) - \phi(\Bar{g}) \right) \left( \phi(h_i) - \phi(\Bar{h}) \right)^\top, \quad \Sigma_{gh} = \Sigma_{hg}^\top.
\end{align*}
When they are full-rank, the \emph{empirical geometric correlation matrix} of $g$ and $h$ is then
\begin{align*}
    \rho = D_{gg}^{-1/2} \Sigma_{gh} D_{hh}^{-1/2}
\end{align*}
where $D_{gg} := \mathrm{Diag}(\Sigma_{gg})$ and $D_{hh} := \mathrm{Diag}(\Sigma_{hh})$.
\end{definition}

Naturally, these definitions share similar properties with their Euclidean analogues.

\begin{proposition}
In the settings of Definition \ref{geom_cov_corr}, the following hold:
\begin{enumerate}[label=\emph{(\alph*)}]
  \item \emph{Translation‐invariance:}  
    $\Sigma_{(g\star a)(h\star b)} = \Sigma_{gh}$, for all $a,b\in M$.
  \item \emph{Symmetry and positive semi‐definiteness:}  
    $\Sigma_{gg}^T=\Sigma_{gg}$, $\Sigma_{gg}\succeq0$,  
    $\Sigma_{hh}^T=\Sigma_{hh}$, $\Sigma_{hh}\succeq0$.
  \item \emph{Cauchy–Schwarz bound:} for all $v,w\in V$,
    $$\bigl|v^T\Sigma_{gh}\,w\bigr|
      \le 
      \sqrt{v^T\Sigma_{gg}v}\;\sqrt{w^T\Sigma_{hh}w}.$$
\end{enumerate}
\end{proposition}

\begin{proof}
For any $a,b \in (M, \star)$, left translation by $a,b$ in $M$ corresponds to addition in the log-charts, which cancels in differences:
\begin{align*}
    \phi(g_i \star a) - \phi(\overline{g \star a}) = \phi(g_i)+\phi(a) - \phi(\Bar{g}) - \phi(a) = \phi(g_i) - \phi(\Bar{g})
\end{align*}
The roles of $a$ and $b$ are symmetric so the equality of $\mathrm{(a)}$ follows. Since the sum of symmetric matrices remains symmetric, $\Sigma_{gg}$ and $ \Sigma_{hh}$ are symmetric. Let $v \in V$, we have
\begin{align*}
    v^\top \Sigma_{gg} v = \frac{1}{N}\sum_{i=1}^{N}\left( v^\top(\phi(g_i) - \phi(\Bar{g})) \right)^2 \geq 0,
\end{align*}
thus $\Sigma_{gg} \succeq 0$ (likewise, $\Sigma_{hh} \succeq 0$), proving $\mathrm{(b)}$. Item $\mathrm{(c)}$ is simply applying the Cauchy-Schwarz inequality in $\mathbb{R}^m$,
\begin{align*}
\bigl|v^\top \Sigma_{gh}\, w\bigr|
=
\Bigl|\frac{1}{N} \sum_{i=1}^N
    v^\top\bigl(\phi(g_i)-\phi(\bar g)\bigr)\;
    w^\top\bigl(\phi(h_i)-\phi(\bar h)\bigr)
\Bigr|
\le
\sqrt{\,v^\top \Sigma_{gg}\, v\,}
\;\sqrt{\,w^\top \Sigma_{hh}\, w\,},
\end{align*}
for all $v, w \in V$. 
\end{proof}

\subsection{Riemannian Manifolds: SPD and Full-Rank Correlation Matrices}

Throughout all of this article, we are going to apply the theory of log-Euclidean Lie groups to the smooth manifolds of symmetric positive-definite (SPD) matrices $\mathcal{S}^+(n)$ and full-rank correlation matrices $\mathrm{Cor}^+(n)$. We recall that a (full-rank) correlation matrix is just a matrix $\Sigma \in \mathcal{S}^+(n)$ satisfying the unit-diagonal constraint, i.e., $\mathrm{diag}\left( \Sigma \right) =1 \in \mathbb{R}^n$. Hence, it is an embedded submanifold $i:\mathrm{Cor}^+(n) \hookrightarrow \mathcal{S}^+(n)$ for the canonical inclusion $i$.
\par
The first log-Euclidean Lie group structure on $\mathcal{S}^+(n)$ was introduced in \cite{arsigny2007} to endow SPD matrices with an associative, commutative and stable operation. More recently, log-Euclidean metrics have been developed for full-rank correlation matrices in \cite{thanwerdas2022, thanwerdas2024}. In this work we highlight the parallel between log-Euclidean metrics and group structures: in particular, we exhibit two log-Euclidean group operations on $\operatorname{Cor}^+(n)$, together with their corresponding group inverses.

\begin{definition}[Log-Euclidean Riemannian metric]
Given a log-Euclidean Lie group $\phi : (M, \diamond) \rightarrow V$, a \emph{log-Euclidean Riemannian metric} on $M$ is a bi-invariant metric for the log-Euclidean group action "$\diamond$". Equivalently, it is the pullback metric by $\phi$ of some inner product on its Lie algebra, $V$.
\end{definition}

 We generalize the notion of inverse-consistency given in \cite{arsigny2007} and \cite{thanwerdas2024} to any log-Euclidean Lie group.

\begin{definition}[Inverse-Consistency]
Let $\phi : (M, \diamond) \rightarrow V$ be a log-Euclidean Lie group. We say that a map $f : M \rightarrow M$ is \textit{inverse-consistent with respect to the group action} if for all $X \in M$, 
\begin{align*}
    f(X_\diamond^{-1}) = \left(f(X)\right)_\diamond^{-1}
\end{align*}
in other words, $f$ commutes with group-inversion.
\end{definition}

\begin{proposition}
Let $\phi : (M, \diamond) \rightarrow V$ be a log-Euclidean Lie group, a smooth function $f$ of $M$ is inverse-consistent if and only if the induced map $\tilde{f} : V \rightarrow V$, defined by:
\begin{align*}
    \tilde{f}(v) := \phi \left( f \left( \phi^{-1}(v) \right) \right)
\end{align*}
is odd, i.e., satisfies $\tilde{f}(-v) = -\tilde{f}(v)$ for all $v \in V$.
\end{proposition}

\begin{proof}
Let $f : M \rightarrow M$ be inverse-consistent and define $\tilde{f}$ as above. Let $v \in V$ and write $X = \phi^{-1}(v) \in M$. Then 
\begin{align*}
    \tilde{f}(v) &= \phi(f(X)) \\
    \Longleftrightarrow \phi^{-1}(-\tilde{f}(v)) &= \phi^{-1}(-\phi(f(X))) \\
    &= f(X)_\diamond^{-1} = f(X_\diamond^{-1}) = f(\phi^{-1}(-\phi(X))) = f(\phi^{-1}(-v))
\end{align*}
since $\phi(X)=v$. Thus,
\[
-\tilde{f}(v) = \phi(f(\phi^{-1}(-v))) = \tilde{f}(-v).
\]
Conversely, assume that $\tilde{f}$ is odd. Then for all $X = \phi^{-1}(v) \in M$, with $v \in V$, we have
\begin{align*}
    \phi(f(X_\diamond^{-1})) = \tilde{f}(-v) = -\tilde{f}(v) = -\phi(f(X)).
\end{align*}
Applying $\phi^{-1}$ to both sides yields
\[
f(X_\diamond^{-1}) = f(X)_\diamond^{-1}.
\]
\end{proof}

The most obvious examples of inverse-consistent functions are given by the inverse maps in a log-Euclidean Lie group. But the proposition above shows in particular that any linear map $V \rightarrow V$ yields an inverse-consistent function.

\subsubsection{The Exponential Diffeomorphism}
\label{exp_sec}

The matrix exponential is a smooth diffeomorphism 
\begin{align*}
    \mathrm{exp} : \mathcal{S}(n) \rightarrow \mathcal{S}^+(n)
\end{align*}
(for a proof of this see for instance \cite{tumpach2024}) and thus induces a log-Euclidean Lie group structure on $\mathcal{S}^+(n)$.

\begin{definition}[Log-Euclidean structure on $\mathcal{S}^+(n)$, \cite{arsigny2007}]
Define $\phi : \mathcal{S}^+(n) \rightarrow \mathcal{S}(n)$ to be the \emph{log-Euclidean Lie group on $\mathcal{S}^+(n)$} given via the diffeomorphism
\begin{align*}
    \phi : \Sigma \in \mathcal{S}^+(n) \longmapsto \mathrm{log}\left( \Sigma \right) \in \mathcal{S}(n)
\end{align*}
where $\mathrm{log} = \mathrm{exp}^{-1}$ is the usual matrix logarithm, and $\mathcal{S}(n)$ denote the space of symmetric matrices.
\end{definition}

\begin{definition}[Log-Euclidean metric on $\mathcal{S}^+(n)$, \cite{arsigny2007}]
The classical \emph{log-Euclidean} metric $g^{\mathrm{LE}}$ on $\mathcal{S}^+(n)$ is the pullback metric of the Frobenius inner product $\langle \cdot, \cdot \rangle$ on $\mathcal{S}(n)$ by the usual matrix logarithm:
\begin{align*}
    g^{\mathrm{LE}} := \mathrm{log}^* \langle \cdot, \cdot \rangle.  
\end{align*}
\end{definition}

Following these definitions, we can explicit the group operations, given $X, Y \in (\mathcal{S}^+(n), \diamond)$: 
\begin{align*}
    X \diamond Y &= \mathrm{exp}\left( \mathrm{log}(X) + \mathrm{log}(Y) \right)\\
    X_\diamond^{-1} &= \mathrm{exp}\left( - \mathrm{log}(X) \right)\\
    &= X^{-1}
\end{align*}
and given two tangent vectors at $X$, $\delta_X, \xi_X \in \mathcal{T}_X \mathcal{S}^+(n) \simeq \mathcal{S}(n)$, the metric writes
\begin{align*}
    g^{\mathrm{LE}}_X\left( \delta_X, \xi_X \right) = \mathrm{log}^* \langle \delta_X, \xi_X \rangle = \mathrm{tr}\Bigl(\left( \mathrm d_X \mathrm{log}(\delta_X) \right) \left( \mathrm d_X \mathrm{log}(\xi_X) \right) \Bigr).
\end{align*}
Observe that the log-Euclidean Lie group inverse coincides here with the usual matrix inverse.

\subsubsection{The Off-Log Diffeomorphism}
\label{off-log_sec}

It was proven in~\cite{thanwerdas2024} that the off-log bijection defined in~\cite{archakov2021} 
\begin{align*}
    \mathrm{Log} : \text{Cor}^+(n) \rightarrow \text{Hol}(n)
\end{align*}
between full-rank correlation matrices and the vector space of hollow matrices (symmetric matrices with null-diagonal) is a smooth diffeomorphism. Given a symmetric matrix $S$, there exists a unique diagonal matrix $\mathcal{D}(S)$ such that $\exp(\mathcal{D}(S)+S) \in \mathrm{Cor}^+(n)$, thereby defining a smooth parametrization of full-rank correlation matrices.

\begin{theorem}[\cite{archakov2021}]
For all symmetric matrices $S \in \mathcal{S}(n) $, there exists a unique diagonal matrix $ \mathcal{D}(S) \in \mathrm{Diag}(n) $ such that $ \exp(\mathcal{D}(S)+S) \in \mathrm{Cor}^+(n) $. We denote this mapping as
\begin{align*}
    \mathcal{D} : S \in \mathcal{S}(n) \longmapsto \mathcal{D}(S) \in \mathrm{Diag}(n).
\end{align*}
\end{theorem}

An inductive algorithm to compute $\mathcal{D}(S)$ that converges linearly was proposed in~\cite{archakov2021}: write $D_0 = 0$ and define $D_{k+1} = D_k - \log (\mathrm{Diag}(\exp (D_k + S))) \in \mathrm{Diag}(n)$. The algorithm stops at step $k$ whenever $| D_{k} - D_{k-1} |_E \leq \epsilon$, where $| \cdot |_E$ is the usual Euclidean norm, and $\epsilon > 0$ is a fixed threshold for convergence.\\
\par
Let us now specify the diffeomorphism, its inverse and their differential given in~\cite{thanwerdas2024}. For all $C \in \mathrm{Cor}^+(n)$ and $S, X, Y \in \mathrm{Hol}(n)$, define $H^0 \in \mathcal{S}^+(n)$ by 
\begin{align*}
   H^0_{il} = \sum_{j, k} P_{ij} P_{ik} P_{lj} P_{lk} \exp^{(1)}(\delta_j, \delta_k), 
\end{align*}
where $\exp^{(1)}$ denotes the first divided difference of the matrix exponential map, $P \in \mathcal{O}(n)$ and $\Delta = \mathrm{Diag}(\delta_1, \ldots, \delta_n) \in \mathrm{Diag}(n)$ are such that $\mathcal{D}(S) + S = P\Delta P^\top$.

\begin{itemize}
\item[]{\textbf{Diffeomorphism:}} \\$\mathrm{Log} : C \in \mathrm{Cor}^+(n) \longmapsto \mathrm{off} \circ \log (C) \in   \mathrm{Hol}(n) $.\hfill (definition of $\phi$)
\item[]{\textbf{Inverse diffeomorphism:}} \\$\mathrm{Exp} : S \in \mathrm{Hol}(n) \longmapsto \exp (\mathcal{D}(S) + S) \in \mathrm{Cor}^+(n)$. \hfill (definition of $\phi^{-1}$)
\item[]{\textbf{Tangent diffeomorphism:}} \\$\mathrm d_C \mathrm{Log} : X \in \mathrm{Hol}(n) \longmapsto \mathrm{off} \circ \mathrm d_C \log (X) \in \mathrm{Hol}(n)$.
\item[]{\textbf{Inverse tangent diffeomorphism:}} \\$\mathrm d_S \mathrm{Exp} : Y \in \mathrm{Hol}(n) \longmapsto \mathrm d_{\mathcal{D}(S) + S} \exp (Y + \mathrm d_S \mathcal{D}(Y)) \in \mathrm{Hol}(n)$,
\end{itemize}
where $\mathrm d_S \mathcal{D}(Y)$ is given by $ - \mathrm{Diag}\bigl((H^0)^{-1} \mathrm{Diag}(d_{\mathcal{D}(S) + S} \exp(Y)) \mathds{1}\bigr)$, and where the map $\mathrm{off}$ sets the diagonal of a matrix to $0$.

\begin{definition}[\textit{off-log} Log-Euclidean structure on $\mathrm{Cor}^+(n)$]
Define $\phi : \mathrm{Cor}^+(n) \rightarrow \mathrm{Hol}(n)$ to be the \emph{log-Euclidean Lie group on $\mathrm{Cor}^+(n)$} given via the diffeomorphism
\begin{align*}
    \mathrm{Log} : C \in \mathrm{Cor}^+(n) \longmapsto \mathrm{off} \circ \mathrm{log} \left( C \right) \in \mathrm{Hol}(n)
\end{align*}
\end{definition}

\begin{definition}[\textit{off-log} Log-Euclidean metric on $\mathrm{Cor}^+(n)$, \cite{thanwerdas2024}]
The (family of) \emph{log-Euclidean metric} $g^{\mathrm{OL}}$ associated to the log-Euclidean Lie group $\mathrm{Log} : C \in \mathrm{Cor}^+(n) \longmapsto \mathrm{off} \circ \mathrm{log} \left( C \right) \in \mathrm{Hol}(n)$ is the pullback metric of the Frobenius inner product $\langle \cdot, \cdot \rangle$ on $\mathrm{Hol}(n)$ by $\mathrm{Log}$:
\begin{align*}
    g^{\mathrm{OL}} := \mathrm{Log}^* \langle \cdot, \cdot \rangle.  
\end{align*}
\end{definition}

Although there is no closed‐form expression for $\mathcal{D}$, the following lemmas will enable the reader to carry out explicit computations.

\begin{lemma}[\cite{archakov2021}]
\label{equiv_diag}
The surjective map $S \in \mathcal{S}(n) \mapsto \exp(S + \mathcal{D}(S)) \in \mathrm{Cor}^+(n)$ is equivariant under the additive group action $\mathcal{S}(n) \times \mathrm{Diag}(n) \rightarrow \mathcal{S}(n)$.
\end{lemma}

\begin{lemma}
\label{diag_log}
Let $X \in \mathrm{Cor}^+(n)$, then $\mathcal{D}\left(\mathrm{Log} (X)\right) = \mathrm{Diag}\left( \log (X) \right)$.
\end{lemma}

\begin{proof}
Let $X \in \mathrm{Cor}^+(n)$, then
\begin{align*}
    & & \exp \left( \mathcal{D}\left(\mathrm{Log}(X)\right) + \mathrm{Log}(X) \right) & = X \\
    &\Longleftrightarrow & \mathcal{D}\left(\mathrm{Log}(X)\right) + \mathrm{Log}(X) & = \log X \\
    &\Longleftrightarrow & \mathcal{D}\left(\mathrm{Log}(X)\right) & = \log X - \mathrm{Log}(X) = \mathrm{Diag}\left( \log(X) \right).
\end{align*}
\end{proof}

\begin{lemma}
\label{inverse_Log}
Let $X$ belong to the Log-Euclidean Lie group $\mathrm{Log} : (\mathrm{Cor}^+(n), \star) \rightarrow \mathrm{Hol}(n)$, then the group inverse of $X$ is given by $X_\star^{-1} = \mathrm{Exp}\left( \log (X^{-1})\right)$.
\end{lemma}

\begin{proof}
Let $X \in \mathrm{Cor}^+(n)$, then,
\begin{align*}
    & & \mathrm{Exp}\bigl(-\mathrm{Log}(X)\bigr)
    & = \mathrm{Exp}\bigl(\mathrm{off}\circ \log(X^{-1})\bigr)
    &&\quad\text{(linearity of the off-map)}\\
    & &
    & = \exp\bigl(\mathcal{D}(\mathrm{off}\circ \log(X^{-1})) + \mathrm{off}\circ \log(X^{-1})\bigr)\\
    & &
    & = \exp\bigl(\mathcal{D}(\log(X^{-1})) + \log(X^{-1})\bigr)
    &&\quad\text{(Lemma \ref{equiv_diag})}\\
    & &
    & = \mathrm{Exp}(\log(X^{-1})) = \mathrm{Exp}(-\log(X)).
\end{align*}
\end{proof}

\begin{lemma}
Let $X \in \mathrm{Cor}^+(n)$, then $\mathcal{D}\left(-\mathrm{Log} (X)\right) = \mathcal{D}\left(-\log (X)\right) + \mathrm{Diag}\left( \log (X) \right)$.
\end{lemma}

\begin{proof}
Let $X \in \mathrm{Cor}^+(n)$, Lemma \ref{inverse_Log} writes
\begin{align*}
    & & \exp\bigl(\mathcal{D}(-\mathrm{Log}(X)) - \mathrm{Log}(X)\bigr)
    & = \exp\bigl(\mathcal{D}(-\log(X)) - \log(X)\bigr) \\
    &\Longleftrightarrow& \mathcal{D}(-\mathrm{Log}(X))
    & = \mathcal{D}(-\log(X)) - \log(X) + \mathrm{Log}(X) \\
    &\Longleftrightarrow& \mathcal{D}(-\mathrm{Log}(X))
    & = \mathcal{D}(-\log(X)) - \mathrm{Diag}(\log(X)).
\end{align*}
Summing $\mathcal{D}(-\mathrm{Log}(X))$ and $\mathcal{D}(\mathrm{Log}(X))$ gives $\mathcal{D}(-\log(X))$, using Lemma \ref{diag_log} the equality follows.
\end{proof}

Following these, we can explicit the group operations, given $X,Y \in \mathrm{Cor}^+(n)$
\begin{align*}
    & & X \star Y           & = \mathrm{Exp}\bigl(\mathrm{Log}(X) + \mathrm{Log}(Y)\bigr)  & \\
    & &                      & = \exp\bigl(\mathcal{D}(A) + A\bigr)                         & \quad\text{where }A=\mathrm{Log}(X)+\mathrm{Log}(Y),\\
    & & X_\star^{-1}         & = \mathrm{Exp}\bigl(\log(X^{-1})\bigr)                       & \quad\text{(Lemma \ref{inverse_Log})}
\end{align*}

and given two tangent vectors at $X$, $\delta_X, \xi_X \in \mathcal{T}_X \mathrm{Cor}^+(n) \simeq \mathrm{Hol}(n)$, the metric writes
\begin{align*}
    g^{\mathrm{OL}}_X\left( \delta_X, \xi_X \right) := \mathrm{Log}^* \langle \delta_X, \xi_X \rangle = \mathrm{tr}\Bigl( \left( \mathrm{off} \circ \mathrm d_X \mathrm{log}(\delta_X) \right) \left(\mathrm{off} \circ \mathrm d_X \mathrm{log}(\xi_X) \right) \Bigr)
\end{align*}
since the $\mathrm{off}$ map is linear, so its differential is itself. The off-log Riemannian metric on $\mathrm{Cor}^+(n)$ splits into the log-Euclidean Riemannian metric $g^{\mathrm{LE}}$ minus a diagonal correction term.

\begin{proposition}
Let $X \in \mathrm{Cor}^+(n)$ and let $\delta_X, \xi_X \in \mathcal{T}_X \mathrm{Cor}^+(n) \simeq \mathrm{Hol}(n)$ be two tangent vectors at $X$. Then 
\begin{align*}
    g^{\mathrm{OL}} = g^{\mathrm{LE}} - g^{\mathrm{DL}}
\end{align*}
where $g^{\mathrm{DL}}_X\left( \delta_X, \xi_X  \right) = \mathrm{tr}\Bigl( \left(\mathrm{Diag}(\mathrm d_X \mathrm{log}(\delta_X))\right) \left(\mathrm{Diag}(\mathrm d_X \mathrm{log}(\xi_X)) \right) \Bigr)$, that is, the Frobenius inner product on the diagonal part of $\mathrm d_X \mathrm{log}(\delta_X)$ and $\mathrm d_X \mathrm{log}(\xi_X)$.
\end{proposition}

\begin{proof}
The proof is only computational. Let $X \in \mathrm{Cor}^+(n)$ and let $\delta_X, \xi_X \in \mathcal{T}_X \mathrm{Cor}^+(n) \simeq \mathrm{Hol}(n)$ be two tangent vectors at $X$. Denote $A = \mathrm d_X \mathrm{log}(\delta_X)$ and $B = \mathrm d_X \mathrm{log}(\xi_X)$, by definition
\begin{align*}
    g^{\mathrm{OL}}_X\left(\delta_X, \xi_X \right) &= \mathrm{tr}\Bigl( \left( \mathrm{off} \circ \mathrm d_X \mathrm{log}(\delta_X) \right) \left(\mathrm{off} \circ \mathrm d_X \mathrm{log}(\xi_X) \right) \Bigr)\\
    &= \mathrm{tr}\Bigl( \left( \mathrm{off} \circ A \right) \left( \mathrm{off} \circ B \right) \Bigr)\\
    &= \mathrm{tr}\Bigl( \left( A - \mathrm{Diag}(A) \right) \left( B - \mathrm{Diag}(B) \right) \Bigr)\\
    &= \mathrm{tr}\Bigl(AB \Bigr) - \mathrm{tr} \Bigl(A \mathrm{Diag}(B)\Bigr) - \mathrm{tr}\Bigl(\mathrm{Diag}(A) B\Bigr) + \mathrm{tr} \Bigl(\mathrm{Diag}(A) \mathrm{Diag}(B)\Bigr)
\end{align*}
Using the fact that $\mathrm{tr}\left( UD \right) = \mathrm{tr}\Bigl( \mathrm{Diag}(U) D \Bigr)$ for all symmetric matrix $U$ and diagonal matrix $D$, by commutativity we get
\begin{align*}
    g^{\mathrm{OL}}_X\left(\delta_X, \xi_X \right) &= \mathrm{tr}\Bigl( AB \Bigr) - \mathrm{tr}\Bigl( \mathrm{Diag}(A) \mathrm{Diag}(B) \Bigr)\\
    &= \mathrm{tr}\Bigl( \mathrm d_X \mathrm{log}(\delta_X) \mathrm d_X \mathrm{log}(\xi_X) \Bigr) - \mathrm{tr}\Bigl( \mathrm{Diag}(\mathrm d_X \mathrm{log}(\delta_X)) \mathrm{Diag}(\mathrm d_X \mathrm{log}(\xi_X)) \Bigr)\\
    &= g^{\mathrm{LE}}_X\left(\delta_X, \xi_X \right) - g^{\mathrm{DL}}_X\left(\delta_X, \xi_X \right)
\end{align*}
\end{proof}

An immediate consequence of the above proposition is the following natural splitting of the log-Euclidean Riemannian metric on $\mathcal{S}^+(n)$.

\begin{theorem}
\label{metric_split_LE_thm}
The log-Euclidean Riemannian metric on $\mathcal{S}^+(n)$ satisfies the following orthogonal decomposition
\begin{align*}
    g^{\mathrm{LE}} = g^{\mathrm{OL}} \oplus g^{\mathrm{DL}}
\end{align*}
\end{theorem}

\begin{corollary}
Let $X \in \mathcal{S}^+(n)$, the map $g_X^{\mathrm{OL}}(\cdot, \cdot)$ defines a symmetric bilinear form on $\mathcal{T}_X\mathcal{S}^+(n)$ which is degenerate, with kernel $\mathrm{d}_X\log^{-1}(\mathrm{Diag}(n))$.
\end{corollary}

\begin{proof}
Let $\delta_X\in \mathcal{T}_X\mathcal{S}^+(n)$. Symmetry is immediate from the definition, and
\[
g_X^{\mathrm{OL}}(\delta_X,\delta_X)
=
\left\|\mathrm{off}\bigl(\mathrm{d}_X\log(\delta_X)\bigr)\right\|_F^2 \ge 0,
\]
so $g_X^{\mathrm{OL}}$ is positive semidefinite. Moreover,
\[
g_X^{\mathrm{OL}}(\delta_X,\cdot)\equiv 0
\quad\Longleftrightarrow\quad
\mathrm{off}\bigl(\mathrm{d}_X\log(\delta_X)\bigr)=0
\quad\Longleftrightarrow\quad
\mathrm{d}_X\log(\delta_X)\in \mathrm{Diag}(n),
\]
which proves the kernel formula. If $X$ is diagonal, then $\mathrm{d}_X\log$ preserves diagonal
matrices, so the kernel is exactly the diagonal subspace.
\end{proof}

We conclude this subsection by proving the following conjecture derived in \cite{thanwerdas2024}.

\begin{theorem}
\label{thm:canonical_congruence_normalization}
For all $\Sigma \in \mathcal{S}^+(n)$, there exists a unique $\Delta \in \mathrm{Diag}^+(n)$ such that
\[
\log (\Delta \Sigma \Delta) \in \mathrm{Hol}(n),
\]
where, clearly, $\mathcal{S}(n)=\mathrm{Hol}(n)\oplus \mathrm{Diag}(n)$.
\end{theorem}

\begin{proof}
Existence was proved in \cite{thanwerdas2024}. The uniqueness statement will follow from the fact, proved later in Proposition~\ref{prop:K_global_section}, that the subset
\[
K:=\exp(\mathrm{Hol}(n))\subset \mathcal{S}^+(n)
\]
is a global section for the congruence principal bundle associated with the action
\[
(\Delta,\Sigma)\longmapsto \Delta \Sigma \Delta,
\qquad
\Delta\in \mathrm{Diag}^+(n),\ \Sigma\in \mathcal{S}^+(n).
\]
Indeed, the condition $\log(\Delta\Sigma\Delta)\in \mathrm{Hol}(n)$ is equivalent to $\Delta\Sigma\Delta\in K$.
\end{proof}

\begin{remark}
It is shown in \cite{thanwerdas2024} that Theorem \ref{thm:canonical_congruence_normalization} further implies the uniqueness of Riemannian logarithm of the quotient-affine metric (see \cite{david2019}) in $\mathrm{Cor}^+(n)$.
\end{remark}

\subsubsection{The Log-Scaling Diffeomorphism}
\label{log-scaling_sec}

The log-scaling diffeomorphism 
\begin{align*}
    \mathrm{Log}^\bullet : \mathrm{Cor}^+(n) \to \mathrm{Row}_0(n),
\end{align*}
mapping full-rank correlation matrices to the space of symmetric matrices with null row-sum, was first made explicit in \cite{thanwerdas2024}, building on the earlier works \cite{marshall1968} and \cite{johnson009}. Given an SPD matrix $\Sigma \in \mathcal{S}^+(n)$, there exists a unique positive diagonal matrix $\mathcal{D}^*(\Sigma) \in \mathrm{Diag}^+(n)$ such that $\log(\mathcal{D}^*(\Sigma) \Sigma \mathcal{D}^*(\Sigma) )$ is a null row sum symmetric matrix, thereby defining another smooth parametrization of the space of full-rank correlation matrices via the inverse of $\mathrm{Log}^\bullet$.

\begin{theorem}[\cite{marshall1968, johnson009}]
For all $\Sigma \in \mathcal{S}^+(n)$, there exists a unique $\mathcal{D}^*(\Sigma)  \in \mathrm{Diag}^+(n)$ such that $\log(\mathcal{D}^*(\Sigma) \Sigma \mathcal{D}^*(\Sigma)) \in \mathrm{Row}_0(n)$. We denote this mapping as 
\begin{align*}
    \mathcal{D}^* : \Sigma \in \mathcal{S}^+(n) \longmapsto \mathcal{D}^*(\Sigma) \in \mathrm{Diag}^+(n).
\end{align*}
\end{theorem}

The numerical computation of $\mathcal{D}^*(\Sigma)$ for $\Sigma \in \mathcal{S}^+(n)$ is a well-known problem in the literature, commonly referred to as scaling a matrix to prescribed row sums. As observed in \cite{thanwerdas2024}, this problem can be reformulated as a strictly convex optimization problem. Specifically, for a given $\Sigma \in \mathcal{S}^+(n)$, $D = \mathcal{D}^*(\Sigma)$ if and only if it minimizes the strictly convex function $F: D \in \mathrm{Diag}^+(n) \longmapsto \frac{1}{2} \mathds{1}^\top D^\top \Sigma D \mathds{1} - \mathrm{tr}(\log (D)) \in \mathbb{R}$. The gradient at $D \in \mathrm{Diag}^+(n)$ is $\nabla_D F = \Sigma D \mathds{1} - D^{-1} \mathds{1}$, and the Hessian is $H_D F = \Sigma + D^{-2}$. Given the closed-form expression for the Hessian, Newton's method provides an efficient numerical approach to computing $\mathcal{D}^*$.\\
\par
Let us now specify the diffeomorphism, its inverse and their differential given in \cite{thanwerdas2024}. For all $C \in \mathrm{Cor}^+(n)$, $X \in \mathrm{Hol}(n)$ and $S,  Y \in \mathrm{Row}_0(n)$ such that $\Sigma = \mathcal{D}^*(C) C \mathcal{D}^*(C) = \exp (S) \in \mathcal{S}^+(n)$. Let 
\begin{align*}
    \pi_1 : \Sigma \in \mathcal{S}^+(n) \mapsto \mathrm{Diag}(\Sigma)^{-1/2} \Sigma \mathrm{Diag}(\Sigma)^{-1/2} \in \mathrm{Cor}^+(n)
\end{align*}
denote the usual projection which associates to an SPD matrix a full-rank correlation matrix, and let us denote $\Delta = \mathrm{Diag}(\Sigma)^{1/2}$ and $X^0= -2 \mathrm{Diag}\bigl((I_n + \Sigma)^{-1} \Delta X \Delta \mathds{1}\bigr)$. 

\begin{itemize}
\item[]{\textbf{Diffeomorphism:}} \\$\mathrm{Log}^\bullet : C \in \mathrm{Cor}^+(n) \longmapsto \log (\mathcal{D}^*(C) C \mathcal{D}^*(C)) \in   \mathrm{Row}_0(n) $. \hfill (definition of $\phi$)
\item[]{\textbf{Inverse diffeomorphism:}} \\$\mathrm{Exp}^\bullet : S \in \mathrm{Row}_0(n) \longmapsto \pi_1 \circ \exp (S) \in \mathrm{Cor}^+(n)$. \hfill (definition of $\phi^{-1}$)
\item[]{\textbf{Tangent diffeomorphism:}} \\$\mathrm d_C \mathrm{Log}^\bullet : X \in \mathrm{Hol}(n) \longmapsto \mathrm d_\Sigma \log (\Delta X \Delta + \frac{1}{2} (X^0 \Sigma + \Sigma X^0) ) \in \mathrm{Row}_0(n)$.
\item[]{\textbf{Inverse tangent diffeomorphism:}} $\mathrm d_S \mathrm{Exp}^\bullet : Y \in \mathrm{Row}_0(n) \longmapsto$\\
$ \Delta^{-1}\left( \mathrm d_S\exp (Y) -\frac{1}{2}(\Delta^{-2} \mathrm{Diag}(\mathrm d_S\exp (Y))\Sigma + \Sigma \mathrm{Diag}(\mathrm d_S\exp (Y)) \Delta^{-2}) \right) \Delta^{-1} \in \mathrm{Hol}(n)$.
\end{itemize}

\begin{definition}[\textit{log-scaling} Log-Euclidean structure on $\mathrm{Cor}^+(n)$]
Define $\phi : \mathrm{Cor}^+(n) \rightarrow \mathrm{Row}_0(n)$ to be the \emph{log-Euclidean Lie group on $\mathrm{Cor}^+(n)$} given via the diffeomorphism
\begin{align*}
    \mathrm{Log}^\bullet : C \in \mathrm{Cor}^+(n) \longmapsto \log (\mathcal{D}^*(C) C \mathcal{D}^*(C)) \in \mathrm{Row}_0(n)
\end{align*}
\end{definition}

\begin{definition}[\textit{log-scaling} Log-Euclidean metric on $\mathrm{Cor}^+(n)$, \cite{thanwerdas2024}]
The (family of) \emph{log-Euclidean metric} $g^{\mathrm{LS}}$ associated to the log-Euclidean Lie group $\mathrm{Log}^\bullet : C \in \mathrm{Cor}^+(n) \longmapsto \log (\mathcal{D}^*(C) C \mathcal{D}^*(C)) \in \mathrm{Row}_0(n)$ is the pullback metric of the Frobenius inner product $\langle \cdot, \cdot \rangle$ on $\mathrm{Row}_0(n)$ by $\mathrm{Log}^\bullet$:
\begin{align*}
    g^{\mathrm{LS}} := \left(\mathrm{Log}^{\bullet}\right)^* \langle \cdot, \cdot \rangle.  
\end{align*}
\end{definition}

Following these definitions, we can explicit the group operations, given $X,Y \in \mathrm{Cor}^+(n)$
\begin{align*}
    X \bullet Y &= \mathrm{Exp}^\bullet\left( \mathrm{Log}^\bullet(X) + \mathrm{Log}^\bullet(Y) \right)\\
    &= \pi_1 \circ \exp \left( \log \left( \mathcal{D}^*(X)X\mathcal{D}^*(X) \right) + \log \left( \mathcal{D}^*(Y)Y\mathcal{D}^*(Y) \right) \right)\\
    X_\bullet^{-1} &= \mathrm{Exp}^\bullet \left( - \mathrm{Log}^\bullet (X) \right)\\
    &= \pi_1 \left( X^{-1} \right)
\end{align*}
and given two tangent vectors at $X$, $\delta_X, \xi_X \in \mathcal{T}_X \mathrm{Cor}^+(n) \simeq \mathrm{Hol}(n)$, the metric writes
\begin{align*}
    g^{\mathrm{LS}}_X\left( \delta_X, \xi_X \right) := \left( \mathrm{Log}^\bullet \right)^* \langle \delta_X, \xi_X \rangle = \mathrm{tr}\Bigl( \left( \mathrm d_X \mathrm{Log}^\bullet (\delta_X) \right) \left(\mathrm d_X \mathrm{Log}^\bullet(\xi_X) \right) \Bigr).
\end{align*}
Observe that the group‐inverse map simply takes the usual matrix inverse, an SPD matrix which fails the unit-diagonal constraint, and then projects it back into $\mathrm{Cor}^+(n)$. Unlike the off-log structure on $\mathrm{Cor}^+(n)$, this operation admits a natural multiplicative inverse, which is especially convenient when working with correlation matrices.

\section{The Isomorphic Isometry Theorem}
\label{isom_iso_sec}

The main theorem of this section states that all log-Euclidean Lie groups of the same dimension are isomorphically isometric. It is classical that every simply connected, complete flat Riemannian manifold is globally isometric to a Euclidean space. In fact, every simply connected complete Riemannian manifold of constant sectional curvature is unique up to isometry (see for instance \cite{berger2007}). However, this fact alone does not provide an explicit isometry when considering specific Riemannian metrics, such as log-Euclidean metrics, which are extensively used in geometric statistics.
\par
We first make explicit the isomorphism from which we build the isometry and then apply this result to construct isomorphic isometries between log-Euclidean structures on SPD matrices and full-rank correlation matrices. Recall that for a log-Euclidean Lie group $\phi : G \rightarrow V$, the vector space $V$ is identified with its Lie algebra $\mathfrak{g}=\mathcal{T}_{e_G}G$. Hence, we shall write log-Euclidean Lie groups $\phi : G \rightarrow \mathfrak g$. Since $\phi^{-1}$ is precisely the Lie exponential $\exp_G$, we may use the two notations interchangeably. 

\begin{theorem}[Isomorphism Theorem]
\label{isomorphism_thm}
Let $\log_G : (G, \cdot) \rightarrow \mathfrak g$ and $\log_H : (H, \star) \rightarrow \mathfrak h$ be two log-Euclidean Lie groups of the same dimension and let $\psi : \mathfrak g \rightarrow \mathfrak h$ be any linear isomorphism (thus, invertible). Then the map $\Phi : (G, \cdot) \rightarrow (H, \star)$ defined as
\begin{align*}
    \Phi := \exp_H \circ \; \psi \circ \log_G
\end{align*}
is the unique Lie group isomorphism $G \rightarrow H$ such that $\psi = \mathrm{d}_{e_G}\Phi$.
\end{theorem}

\begin{proof}
Lie's second theorem (Theorem \ref{thm:Lie_second} in the simply connected case) states that given a Lie algebra morphism $\psi : \mathfrak g \rightarrow \mathfrak h$, there exists a unique Lie group morphism $\Phi : G \rightarrow H$ such that $\psi = \mathrm{d}_{e_G} \Phi$. We now show that
\begin{align*}
   \Phi := \exp_H \circ \; \psi \circ \log_G : G \rightarrow H 
\end{align*}
is the unique Lie group morphism satisfying $\mathrm{d}_{e_G} \Phi = \psi$. By Lemma~\ref{lem:LE-dlog-dexp}, one has
\begin{align*}
\mathrm d_{e_G} \log_G = \operatorname{Id}_{\mathfrak g}
\quad\text{and}\quad
\mathrm d_0 \exp_H = \operatorname{Id}_{\mathfrak h}.
\end{align*}
Therefore, by linearity of $\psi$ and the chain rule,
\begin{align*}
    \mathrm{d}_{e_G} \Phi 
    = \mathrm{d}_0 \exp_H \circ \; \psi \circ \mathrm{d}_{e_G} \log_G
    = \operatorname{Id}_{\mathfrak h} \circ \; \psi \circ \operatorname{Id}_{\mathfrak g}
    = \psi.
\end{align*}
Since the Lie algebras $\mathfrak g$ and $\mathfrak h$ are abelian, their Lie brackets are trivial. Hence, a Lie algebra morphism $\psi : \mathfrak g \rightarrow \mathfrak h$ is simply a linear map, and if it is bijective, then it is a Lie algebra isomorphism. Moreover, simple connectedness and commutativity imply that the Lie exponential maps are group homomorphisms. Thus, for all $g_1, g_2 \in G$, we compute: 
\begin{align*} 
\Phi(g_1 \cdot g_2) &= \exp_H \circ \; \psi \circ \log_G (g_1 \cdot g_2) \\ &= \exp_H \circ \; \psi \left( \log_G(g_1) + \log_G(g_2) \right) \\ 
&= \exp_H \left( \psi(\log_G(g_1)) + \psi(\log_G(g_2)) \right) = \Phi(g_1) \star \Phi(g_2), 
\end{align*} so $\Phi$ is a Lie group homomorphism. Since each component of $\Phi$ is bijective, the map itself is bijective and hence a Lie group isomorphism.
\end{proof}

Theorem~\ref{isomorphism_thm} in particular states that all log-Euclidean Lie groups of the same dimension are isomorphic. Moreover, given an isomorphism between their Lie algebras, one can make the Lie group isomorphism explicit. Since both Lie algebras are diffeomorphic to $\mathbb{R}^n$ for some $n \in \mathbb{N}$, the isomorphism between them is simply a linear change of coordinates. Additionally, fixing any linear isomorphism $A : \mathcal{T}_{e_G}G \rightarrow \mathcal{T}_{e_H}H$, one observes that the composition $\psi := \mathrm d_{e_H}\log_H \circ A \circ (\mathrm d_{e_G}\log_G)^{-1}$ defines a linear isomorphism $\mathfrak g \rightarrow \mathfrak h$ via the identification of the tangent spaces $\mathcal{T}_{e_G}G$ and $\mathcal{T}_{e_H}H$ with their respective Lie algebras. The following commutative diagram summarizes Theorem~\ref{isomorphism_thm} with this observation.

\begin{center}
\begin{tikzcd}[column sep=large, row sep=large]
G \arrow[rr, "\Phi"] \arrow[dd, "\log_G"'] & & H \\
                                           & &   \\
\mathfrak{g} \simeq \mathcal{T}_{e_G}G \arrow[rr, "\mathrm{d}_{e_H}\log_H \circ A \circ \left(\mathrm{d}_{e_G}\log_G\right)^{-1}"'] & & \mathfrak{h} \simeq \mathcal{T}_{e_H}H \arrow[uu, "\exp_H"']
\end{tikzcd}
\end{center}

\begin{theorem}[Isometric Isomorphism Theorem]
\label{iso_isom_thm}
Let 
\begin{align*}
    \log_G : (G,\cdot)\;\longrightarrow\;\mathfrak g
    \quad\text{and}\quad
    \log_H : (H,\star)\;\longrightarrow\;\mathfrak h 
\end{align*}
be two log‐Euclidean Lie groups of the same dimension, each equipped with its log‐Euclidean Riemannian metric derived from the Frobenius inner product on their Lie algebras. Then:
\begin{enumerate}[label=\emph{(\alph*)}]
  \item  $G$ and $H$ are isometrically isomorphic as Riemannian manifolds.
  \item  In fact, any linear isometry $\psi:\mathfrak g\to\mathfrak h$ (with respect to the Frobenius inner product) induces a unique Lie group isomorphism and Riemannian isometry
  \begin{align*}
    \Phi \;=\; \exp_H \circ \; \psi \circ \log_G
    \;:\; G \;\longrightarrow\; H \quad \text{ such that } \quad \mathrm{d}_{e_G}\Phi = \psi.  
  \end{align*}
\end{enumerate}
\end{theorem}

\begin{proof}
We prove statement (b), since it immediately implies (a). Any two finite-dimensional real inner product spaces of the same dimension are isometrically isomorphic. Now, by Theorem \ref{isomorphism_thm}, any linear isomorphism $\psi \;:\; \mathfrak g \;\rightarrow\; \mathfrak h$ integrates to a unique Lie‐group isomorphism
\begin{align*}
  \Phi \;=\; \exp_H \circ \; \psi \circ \log_G 
  \;:\; G \;\longrightarrow\; H.
\end{align*}
We claim that, if $\psi$ is furthermore a linear isometry of the Lie algebras
\begin{align*}
  \psi \;:\; (\mathfrak g,\langle\cdot,\cdot\rangle)
  \;\xrightarrow{\;\cong\;} 
  (\mathfrak h,\langle\cdot,\cdot\rangle),
\end{align*}
then $\Phi$ is also a Riemannian isometry. Let 
\begin{align*}
  g^{(1)} \;=\; (\log_G)^{*}\langle\cdot,\cdot\rangle
  \quad\text{and}\quad
  g^{(2)} \;=\; (\log_H)^{*}\langle\cdot,\cdot\rangle
\end{align*}
denote the log–Euclidean metrics on $G$ and $H$. Since $\log_H \circ \; \Phi \;=\; \psi \circ \log_G$, we have for every $X\in G$ the identity of differentials
\begin{align*}
  \mathrm d_{\Phi(X)}\!\log_H \;\circ\; \mathrm d_X\Phi
  \;=\; 
  \psi \;\circ\; \mathrm d_X\!\log_G.
\end{align*}
Now let $\delta_X,\xi_X\in \mathcal{T}_XG$ be tangent vectors at $X$.  Then
\begin{align*}
  g^{(2)}_{\Phi(X)}\bigl(\mathrm d_X\Phi(\delta_X),\,\mathrm d_X\Phi(\xi_X)\bigr)
  &= 
  \Bigl\langle \mathrm d_{\Phi(X)}\!\log_H\bigl(\mathrm d_X\Phi(\delta_X)\bigr),
  \;\mathrm d_{\Phi(X)}\!\log_H\bigl(\mathrm d_X\Phi(\xi_X)\bigr)\Bigr\rangle
  \\
  &=
  \Bigl\langle \psi\bigl(\mathrm d_X\!\log_G(\delta_X)\bigr),\;
    \psi\bigl(\mathrm d_X\!\log_G(\xi_X)\bigr)\Bigr\rangle
  \\
  &=
  \Bigl\langle \mathrm d_X\!\log_G(\delta_X),\;\mathrm d_X\!\log_G(\xi_X)\Bigr\rangle
  \;=\;
  g^{(1)}_X\bigl(\delta_X,\xi_X\bigr),
\end{align*}
where in the second line we used $\mathrm d_{\Phi(X)}\!\log_H\circ \; \mathrm d_X\Phi = \psi \circ \mathrm d_X\!\log_G$, and in the third line we used that $\psi$ is an isometry on the Lie algebras. Hence
\begin{align*}
  g^{(2)}_{\Phi(X)}\bigl(\mathrm d_X\Phi(\delta_X),\,\mathrm d_X\Phi(\xi_X)\bigr)
  \;=\;
  g^{(1)}_X(\delta_X,\xi_X),
  \quad
  \forall\,X\in G,\;\forall\,\delta_X,\xi_X\in \mathcal{T}_XG,
\end{align*}
showing that $\Phi$ is a Riemannian isometry. Finally, since $\psi$ is in particular a Lie‐algebra isomorphism, Theorem \ref{isomorphism_thm} guarantees that $\Phi$ is a Lie‐group isomorphism.  Therefore $\Phi$ is both a Lie‐group isomorphism and a Riemannian isometry, as claimed.
\end{proof}

\begin{remark}
The proof of Theorem \ref{iso_isom_thm} uses a linear isometry between the Lie algebras endowed with their chosen inner products. Concretely, given any linear isomorphism $A:\mathfrak g\to\mathfrak h$, one may take the orthogonal factor in the polar decomposition of $A$ (equivalently, map an orthonormal basis of $\mathfrak g$ to an orthonormal basis of $\mathfrak h$) to obtain a linear isometry $\psi:\mathfrak g\to\mathfrak h$. Thus any two log-Euclidean metrics on manifolds of the same dimension are (non-canonically) Riemannian isometric.
\end{remark}

\begin{remark}
One can always choose $\psi := \mathrm d_{e_H}\log_H \circ A \circ \left(\mathrm d_{e_G}\log_G\right)^{-1}$ and take its orthogonal part in its polar decomposition to obtain a linear isometry $\mathfrak g \rightarrow \mathfrak h$ which translates to a Lie group isomorphism and a Riemannian isometry $G \rightarrow H$. Alternatively, one can build isometries at the Lie algebras level by sending orthonormal basis to orthonormal basis. 
\end{remark}

Let us now apply Theorem \ref{iso_isom_thm} to the log‐Euclidean structures on the space of SPD matrices and on the space of full‐rank correlation matrices, which were introduced in the previous section. There are two main consequences:
\begin{enumerate}
  \item Any two log‐Euclidean Riemannian metrics on these spaces (of the same dimension) are actually isometric. In other words, despite arising from different bases or parametrizations, all of these metrics agree up to a global isometry.
  \item Moreover, one can write down these isometries explicitly, and each such map respects the underlying group structure. That is, the map realizing the isometry is also a Lie‐group isomorphism between the corresponding log‐Euclidean groups.
\end{enumerate}

\subsection{Isometry between $\left(\mathcal{S}^+(n-1), g^{\mathrm{LE}}\right)$ and $\left(\mathrm{Cor}^+(n), g^{\mathrm{OL}}\right)$}
\label{off_log_isometry}
Let us make explicit an isometry between SPD matrices endowed with the \textit{log-euclidean} Riemannian metric and full-rank correlation matrices endowed with the \textit{off-log} Riemannian metric. Definitions of these metrics and related mappings can be found in Sections \ref{exp_sec}, \ref{off-log_sec}.
First of all, observe that
\begin{align*}
  \dim\bigl(\mathcal{S}(n-1)\bigr)
  &= \frac{n(n-1)}{2}
  = \dim\bigl(\mathrm{Hol}(n)\bigr).
\end{align*}
Hence, by Theorem \ref{iso_isom_thm}, the two log‐Euclidean Lie groups $\bigl(\mathcal{S}^+(n-1),\,g^{\mathrm{LE}}\bigr)$ and $\bigl(\mathrm{Cor}^+(n),\,g^{\mathrm{OL}}\bigr)$ are isomorphically isometric. We now construct the explicit linear isometry between their Lie algebras,
\begin{align*}
  \psi_{\mathrm{OL}} : \bigl(\mathcal{S}(n-1),\,\langle\cdot,\cdot\rangle\bigr)
  \;\longrightarrow\;
  \bigl(\mathrm{Hol}(n),\,\langle\cdot,\cdot\rangle\bigr),
\end{align*}
where $\langle \cdot, \cdot \rangle$ denotes as usual the Frobenius inner product. First we provide orthonormal bases for $\mathcal{S}(n-1)$ and $\mathrm{Hol}(n)$. Then, we define $\psi_{\mathrm{OL}}$ to send orthonormal basis to orthonormal basis. Let $E_{ij} = \left( \delta_{ij} \right)_{ij}$ be the matrix defined with $1$ at coefficient $(i,j)$ and $0$ else.
\begin{itemize}
  \item \textbf{Orthonormal basis of $\mathcal{S}(n-1)$.} 
    \begin{itemize}
        \item Define the family $D_{kk} \;:=\; E_{kk}$ for $1 \le k \le n-1$ constituted of $(n-1)$ diagonal elements.
        \item Define the family $U_{ij} \;:=\; \frac{1}{\sqrt 2}\bigl(E_{ij} + E_{ji}\bigr)$ for $1 \le i < j \le n-1$ constituted of $\binom{\,n-1\,}{2} = \frac{(n-2)(n-1)}{2}$ off‐diagonal elements.
    \end{itemize}
    Altogether, 
    \begin{align*}
        \{\,D_{kk} : 1 \le k \le n-1\,\}
        \;\cup\;
        \{\,U_{ij} : 1 \le i < j \le n-1\,\}
    \end{align*}
    is an orthonormal basis of $\mathcal{S}(n-1)$.

  \item \textbf{Orthonormal basis of $\mathrm{Hol}(n)$.} \\
    For $1 \le i < j \le n$, define
    \[
      F_{ij} \;=\; \frac{1}{\sqrt 2}\bigl(E_{ij} + E_{ji}\bigr).
    \]
    These $\binom{\,n\,}{2} = \frac{n(n-1)}{2}$ matrices form an orthonormal basis of $\mathrm{Hol}(n)$.

  \item \textbf{Definition of $\psi_{\mathrm{OL}}$.} \\
    The map $\psi_{\mathrm{OL}}$ send each basis element of $\mathcal{S}(n-1)$ to the corresponding element of $\mathrm{Hol}(n)$ by
    \begin{align*}
    \psi_{\mathrm{OL}}=\left\{
        \begin{array}{ll}
             \psi_{\mathrm{OL}}(D_{kk}) &= F_{1, k+1},  \\
             \psi_{\mathrm{OL}}(U_{ij}) &= F_{i+1, j+1}. 
        \end{array} 
    \right.
    \end{align*}
    Since both $\{D_{kk},\,U_{ij}\}$ and $\{F_{1,k+1},\,F_{\,i+1,j+1}\}$ are orthonormal bases in their respective spaces, $\psi_{\mathrm{OL}}$ is a linear isometry. In block form, $\psi_{\mathrm{OL}}$ writes for any $S \in \mathcal{S}(n-1)$
    \begin{align*}
        \psi_{OL}(S)\;=\; \begin{bmatrix}
        0 & \dfrac{1}{\sqrt{2}}\,\mathrm{diag}(S)^{\!\top} \\[0.6em]
        \dfrac{1}{\sqrt{2}}\,\mathrm{diag}(S) & \mathrm{off}(S)
        \end{bmatrix}.
    \end{align*}
\end{itemize}
We recall that $\operatorname{diag}(S) \in \mathbb{R}^{\frac{n(n-1)}{2}}$ is a real vector and $\operatorname{off}(S) \in \operatorname{Hol}(n-1)$ is a real hollow matrix. Let us illustrate in an example the action of $\psi_{\mathrm{OL}}$ over a symmetric matrix.
\begin{example}
Fix $n = 5$. The action of $\psi_{\mathrm{OL}}$ over $X$ defined above can be seen as
\begin{align*}
X \;=\;
\begin{pmatrix}
\mathbf{1} & 5 & 6 & 7\\
5 & \mathbf{2} & 8 & 9\\
6 & 8 & \mathbf{3} & 10\\
7 & 9 & 10 & \mathbf{4}
\end{pmatrix}
\;\in\; \mathcal{S}(4)
\;\longmapsto\;
\psi_{\mathrm{OL}}(X)
\;=\;
\begin{pmatrix}
0 & \mathbf{\frac{1}{\sqrt{2}}}  & \mathbf{\frac{2}{\sqrt{2}}}  & \mathbf{\frac{3}{\sqrt{2}}}  & \mathbf{\frac{4}{\sqrt{2}}} \\
\mathbf{\frac{1}{\sqrt{2}}} & 0  & 5  & 6  & 7 \\
\mathbf{\frac{2}{\sqrt{2}}} & 5  & 0  & 8  & 9 \\
\mathbf{\frac{3}{\sqrt{2}}} & 6  & 8  & 0  & 10\\
\mathbf{\frac{4}{\sqrt{2}}} & 7  & 9  & 10 & 0
\end{pmatrix}
\;\in\; \mathrm{Hol}(5).
\end{align*}
\end{example}

Therefore, $\psi_{\mathrm{OL}} = \mathrm d \Phi_{\mathrm{OL}}$ integrates to the Lie group isomorphism and Riemannian isometry
\begin{align*}
    \Phi_{\mathrm{OL}} : \left(\mathcal{S}^+(n-1), g^{\mathrm{LE}}\right) &\longrightarrow \left(\mathrm{Cor}^+(n), g^{\mathrm{OL}}\right), \quad
    X \longmapsto \phi_2^{-1} \circ \; \psi_{\mathrm{OL}} \circ \phi_1 (X)
\end{align*}
where we have defined 
\begin{itemize}
    \item $\phi_1 : X \in \mathcal{S}^+(n-1) \mapsto \log(X) \in \mathcal{S}(n-1)$, \hfill (definition of $\log_G$)
    \item $\psi_{\mathrm{OL}} : X \in \mathcal{S}(n-1) \longmapsto \psi_{\mathrm{OL}}(X) \in \mathrm{Hol}(n)$,
    \item $\phi_2^{-1} \in \mathrm{Hol}(n) \longmapsto \mathrm{Exp}(X) = \exp(\mathcal{D}(X) + X) \in \mathrm{Cor}^+(n)$. \hfill (definition of $\exp_H$)
\end{itemize}
Thus, $\Phi_{\mathrm{OL}} : X \in \mathcal{S}^+(n-1) \longmapsto \exp\left( \mathcal{D}(\psi_{\mathrm{OL}}( \log (X))) + \psi_{\mathrm{OL}}( \log (X))\right) \in \mathrm{Cor}^+(n)$ is an isomorphic isometry and makes the following diagram commutative
\begin{center}
    \begin{tikzcd}
    \mathcal{S}^+(n-1) \arrow[rr, "\Phi_{\mathrm{OL}}"] \arrow[dd, "\phi_1 = \log"'] &  & \mathrm{Cor}^+(n)                                                                                 \\
                                                                       &  &                                                                                                   \\
    \mathcal{S}(n-1) \arrow[rr, "\psi_{\mathrm{OL}}= \mathrm d \Phi_{\mathrm{OL}}"']                &  & \mathrm{Hol}(n) \arrow[uu, "\phi_2^{-1}=\exp\left(\mathcal{D}(\cdot) + (\cdot) \right)"']
    \end{tikzcd}
\end{center}

\subsection{Isometry between $\left(\mathcal{S}^+(n-1), g^{\mathrm{LE}}\right)$ and $\left(\mathrm{Cor}^+(n), g^{\mathrm{LS}}\right)$}
\label{log_scaling_isometry}

Let us make explicit an isometry between SPD matrices endowed with the \textit{log-euclidean} Riemannian metric and full-rank correlation matrices endowed with the \textit{log-scaling} Riemannian metric. Definitions of these metrics and related mappings can be found in Sections \ref{exp_sec}, \ref{log-scaling_sec}. Similarly as in the previous section, observe that
\begin{align*}
  \dim\bigl(\mathcal{S}(n-1)\bigr)
  &= \frac{n(n-1)}{2}
  = \dim\bigl(\mathrm{Row}_0(n)\bigr).
\end{align*}

Hence, by Theorem \ref{iso_isom_thm}, the two log‐Euclidean Lie groups $\bigl(\mathcal{S}^+(n-1),\,g^{\mathrm{LE}}\bigr)$ and $\bigl(\mathrm{Cor}^+(n),\,g^{\mathrm{LS}}\bigr)$ are isomorphically isometric. We now construct the explicit linear isometry between their Lie algebras,
\begin{align*}
  \psi_{\mathrm{LS}} : \bigl(\mathcal{S}(n-1),\,\langle\cdot,\cdot\rangle\bigr)
  \;\longrightarrow\;
  \bigl(\mathrm{Row}_0(n),\,\langle\cdot,\cdot\rangle\bigr),
\end{align*}
where $\langle \cdot, \cdot \rangle$ denotes as usual the Frobenius inner product. Following \cite{clarke2008} we introduce the standard \textit{Helmert-contrast} matrix $B_n \in \mathbb{R}^{(n-1)\times n}$ defined as
\begin{align*}
B_n \;=\;
\begin{pmatrix}
\frac{1}{\sqrt{2}} & -\frac{1}{\sqrt{2}} & 0 & \cdots & 0 \\[6pt]
\frac{1}{\sqrt{6}} & \frac{1}{\sqrt{6}} & -\frac{2}{\sqrt{6}} & \cdots & 0 \\[6pt]
\vdots & \vdots & \vdots & \ddots & \vdots \\[6pt]
\frac{1}{\sqrt{(n-1)\,n}} & \frac{1}{\sqrt{(n-1)\,n}} & \frac{1}{\sqrt{(n-1)\,n}} & \cdots & -\frac{n-1}{\sqrt{(n-1)\,n}}
\end{pmatrix}.
\end{align*}
\begin{lemma}[\cite{clarke2008}]
\label{lemma_helmert_1}
Let $B_n \in \mathbb{R}^{(n-1)\times n}$ be the standard Helmert-constrast matrix,  
\begin{enumerate}[label=\emph{(\alph*)}]
    \item $B_n \mathds{1}_n = 0$,
    \item $B_n B_n^\top = \mathrm I_{n-1}$,
    \item $B_n^\top B_n = \mathrm I_n - \frac{1}{n}\mathds{1}_n\mathds{1}_n^\top$.
\end{enumerate}
\end{lemma}
Consequently, $B_n^\top$ is an $n \times (n-1)$ matrix whose columns are orthonormal and orthogonal to the unit vector $\mathds{1}_n$. The following lemma shows that $\psi_{\mathrm{LS}}$ conjugates a matrix by $B_n^\top$.

\begin{lemma}[\textbf{Definition of $\psi_{\mathrm{LS}}$}]
Let $B_{n}\in\R^{(n-1)\times n}$ be the standard Helmert‐contrast matrix. Define
\begin{align*}
\psi_{\mathrm{LS}}:\mathcal{S}(n-1)\;\longrightarrow\;\mathrm{Row}_{0}(n),
\qquad
X \;\longmapsto\; B_{n}^\top\,X\,B_{n},
\end{align*}
Then $\psi_{\mathrm{LS}}$ is a linear isometry (with respect to the Frobenius inner product).
\end{lemma}

\begin{proof}
Since $X$ is symmetric, then $\psi_{\mathrm{LS}}(X) = B_{n}^\top X B_{n}$ is also symmetric. From Lemma \ref{lemma_helmert_1} we see that $B_{n} \mathds{1}_n =0$ so $B_{n}^\top X B_{n} \mathds{1}_n = 0$ hence $\psi_{\mathrm{LS}}(X)$ is indeed a symmetric matrix with null row-sum. Let us show that it is a linear isometry of the Frobenius inner product. Since $B_n^\top$ has orthonormal columns (Lemma \ref{lemma_helmert_1}) conjugating any two matrices $X, Y \in \mathcal{S}(n-1)$ yields
\begin{align*}
    \langle B_n X B_n^\top , B_n Y B_n^\top \rangle = \mathrm{tr}\left( (B_n X B_n^\top)^\top (B_n Y B_n^\top) \right) = \mathrm{tr}\left( B_n^\top X (B_n B_n^\top) Y B_n \right) = \mathrm{tr}\left( X Y \right),
\end{align*}
because $B_n B_n^\top = \mathrm I_{n-1}$ and trace is cyclic. Therefore, $\Vert \psi_{\mathrm{LS}}(X) \Vert_E = \Vert X \Vert_E$.
\end{proof}

As before, $\psi_{\mathrm{LS}} = \mathrm d \Phi_{\mathrm{LS}}$ integrates to the Lie group isomorphism and Riemannian isometry
\begin{align*}
    \Phi_{\mathrm{LS}} : \left(\mathcal{S}^+(n-1), g^{\mathrm{LE}}\right) &\longrightarrow \left(\mathrm{Cor}^+(n), g^{\mathrm{LS}}\right), \quad
    X \longmapsto \phi_2^{-1} \circ \; \psi_{\mathrm{LS}} \circ \phi_1 (X)
\end{align*}
where we have defined 
\begin{itemize}
    \item $\phi_1 : X \in \mathcal{S}^+(n-1) \mapsto \log(X) \in \mathcal{S}(n-1)$, \hfill (definition of $\log_G$)
    \item $\psi_{\mathrm{LS}} : X \in \mathcal{S}(n-1) \longmapsto \psi(X)=B_{n}^\top\,X\,B_{n} \in \mathrm{Row}_0(n)$,
    \item $\phi_2^{-1} \in \mathrm{Row}_0(n) \longmapsto \mathrm{Exp}^\bullet(X) = \pi_1 \circ \exp(X) \in \mathrm{Cor}^+(n)$. \hfill (definition of $\exp_H$)
\end{itemize}
Thus, $\Phi_{\mathrm{LS}} : X \in \mathcal{S}^+(n-1) \longmapsto \pi_1 \circ \exp \left( B_{n}^\top\, \log(X) \,B_{n} \right) \in \mathrm{Cor}^+(n)$ is an isomorphic isometry and makes the following diagram commutative

\begin{center}
    \begin{tikzcd}
    \mathcal{S}^+(n-1) \arrow[rr, "\Phi_{\mathrm{LS}}"] \arrow[dd, "\phi_1 = \log"'] &  & \mathrm{Cor}^+(n)                                           \\
                                                                       &  &                                                             \\
    \mathcal{S}(n-1) \arrow[rr, "\psi_{\mathrm{LS}}= \mathrm d \Phi_{\mathrm{LS}}"']               &  & \mathrm{Hol}(n) \arrow[uu, "\phi_2^{-1}=\pi_1 \circ \exp"']
    \end{tikzcd}
\end{center}

\subsection{Isometry between $\left(\mathrm{Cor}^+(n), g^{\mathrm{OL}}\right)$ and $\left(\mathrm{Cor}^+(n), g^{\mathrm{LS}}\right)$}
\label{isom_log_subsec}

Let us make explicit an isometry between the two log-Euclidean (\textit{off-log, log-scaling}) Riemannian metrics on full-rank correlation matrices derived from \cite{thanwerdas2024} and defined in Sections \ref{log-scaling_sec}, \ref{off-log_sec}. In Section \ref{off_log_isometry} we have defined an isometry
\begin{align*}
    \Phi_{\mathrm{OL}} : \left(\mathcal{S}^+(n-1), g^{\mathrm{LE}}\right) &\longrightarrow \left(\mathrm{Cor}^+(n), g^{\mathrm{OL}}\right)\\
    X &\longmapsto \exp\left( \mathcal{D}(\psi_{\mathrm{OL}}(\log(X))) + \psi_{\mathrm{OL}}(\log(X))\right),
\end{align*}
whose inverse is given by
\begin{align*}
    \Phi_{\mathrm{OL}} :  \left(\mathrm{Cor}^+(n), g^{\mathrm{OL}}\right)&\longrightarrow \left(\mathcal{S}^+(n-1), g^{\mathrm{LE}}\right)\\
    X &\longmapsto \exp\left( \psi_{\mathrm{OL}}^{-1}(\mathrm{off} \circ (\log(X))) \right).
\end{align*}
Likewise, in Section \ref{log_scaling_isometry} we have defined an isometry
\begin{align*}
    \Phi_{\mathrm{LS}} : \left(\mathcal{S}^+(n-1), g^{\mathrm{LE}}\right) &\longrightarrow \left(\mathrm{Cor}^+(n), g^{\mathrm{LS}}\right)\\
    X &\longmapsto \pi_1 \circ \exp\left( B_n^\top \log(X) B_n \right).
\end{align*}
Composing $\Phi_{\mathrm{LS}}$ with $\Phi_{\mathrm{OL}}^{-1}$ yields an isometry
\begin{align*}
    \Phi_{\mathrm{LS}} \circ \Phi_{\mathrm{OL}}^{-1} : \left(\mathrm{Cor}^+(n), g^{\mathrm{OL}}\right)&\longrightarrow \left(\mathrm{Cor}^+(n), g^{\mathrm{LS}}\right) \\
    X &\longmapsto \pi_1 \circ \exp \left( B_n^\top \psi_{\mathrm{OL}}^{-1} (\mathrm{off} \circ (\log(X))) B_n \right)
\end{align*}
making the following diagram commutative
\begin{center}
    \begin{tikzcd}
    {\left(\mathrm{Cor}^+(n), g^{\mathrm{OL}}\right)} \arrow[rr, "\Phi_{\mathrm{LS}} \circ \Phi_{\mathrm{OL}}^{-1}"] \arrow[rd, "\Phi_{\mathrm{OL}}^{-1}"'] &                                                                                                          & {\left(\mathrm{Cor}^+(n), g^{\mathrm{LS}}\right)} \\
                                                                                                                                                                                    & {\left(\mathcal{S}^+(n-1), g^{\mathrm{LE}}\right)} \arrow[ru, "\Phi_{\mathrm{LS}}"'] &                                                           
    \end{tikzcd}
\end{center}
Hence, we showed that the two Riemannian metrics defined in \cite{thanwerdas2024} are actually isometric, and the isometry is furthermore an isomorphism (i.e. respecting the group operations associated to the two log-Euclidean Lie groups). 

\section{Quotient of Log-Euclidean Lie groups}
\label{quotient_sec}

\subsection{Isomorphic sections}
\label{isomorphic_section}

In this section, we first show that 
\begin{enumerate}
    \item The quotient $G/H$ of a log-Euclidean Lie group $G$ by any closed connected subgroup $H \subset G$ is a log-Euclidean Lie group as well.
    \item The quotient map $\pi : G \rightarrow G/H$ is a principal $H$-bundle which admits global sections. In particular, Proposition \ref{trivial_section} yields the triviality of this bundle:
        \begin{align*}
            G \simeq G/H \times H.
        \end{align*}
\end{enumerate}
For the proof of the latter point, we made explicit a global section $s$ of $\pi$
\begin{align*}
    s : G/H \longrightarrow G, \quad [g] \longmapsto \exp_G\left( s_{\mathrm{lin}}(\log_G(g) + \mathfrak h)\right)
\end{align*}
with $s_{\mathrm{lin}} := (\mathrm d_{e_G}\pi_{\vert \mathfrak m})^{-1}$ being a smooth linear choice of representatives in any supplementary space $\mathfrak m$ to $\mathfrak h$ in $\mathfrak g/ \mathfrak h$. In a second step, we show that $s$ is a smooth Lie group homomorphism whose kernel is trivial, and thus we deduce from the first isomorphism theorem that the image of $s$, denoted $K$ is a subgroup $K \subset G$ diffeomorphic to $G/H$. In particular, any such section is a log-Euclidean Lie group as well. Naturally, we then answer which sections are log-Euclidean Lie groups ---that is, which sections are Lie group isomorphisms--- diffeomorphic to $G/H$, that inherits the log-Euclidean group structure of $G/H$, we call such sections $\emph{log-Euclidean sections}$. It turns out that a section $\sigma : G/H \rightarrow G$ is a log-Euclidean section if and only if $G$ splits as an internal direct product
\begin{align*}
    G = K \times H, \quad K := \mathrm{im}(\sigma).
\end{align*}
This results is almost satisfactory on its own, but because log-Euclidean Lie groups are globally diffeomorphic to their Lie algebras, we show in a third step that log-Euclidean sections are given exactly by a linear choice of a supplementary space $\mathfrak m$ to $\mathfrak h$ in $\mathfrak g/ \mathfrak h$ and thus, have the same form as the section $s$ defined above. For proving this, we show that any log-Euclidean section splits at the level of the Lie algebras the short exact sequence
\begin{align*}
    0 \longrightarrow \mathfrak h \xhookrightarrow{i} \mathfrak g
    \;\overset{\mathrm d_{e_G}\pi}{\underset{s_{\mathrm{lin}}}{\Longrightleftarrows}}\;
    \mathfrak g / \mathfrak h \longrightarrow 0 .  
\end{align*}
First, we recall essential properties of homogeneous manifolds that we specify to log-Euclidean Lie groups. Since any log-Euclidean Lie group is connected, any open subgroup in $G$ must be $G$ itself, hence in all that follows we shall only consider \textit{closed connected} subgroups of log-Euclidean Lie groups.

\begin{theorem}[\cite{warner2013}]
\label{hom_mfd}
Let $H$ be a closed subgroup of a Lie group $G$, and let $G/H$ be the set $\left\{ \sigma H: \sigma \in G \right\}$ of left cosets modulo $H$. Let $\pi : G \rightarrow G/H$ denote the natural projection $\pi(\sigma) = \sigma H$. Then $G/H$ has a unique manifold structure such that
\begin{enumerate}[label=\emph{(\alph*)}]
    \item $\pi$ is $\mathcal{C}^\infty$,
    \item there exists smooth local sections of $G/H$ in $G$, that is, if $\sigma H \in G/H$, there is a neighborhood $W$ of $\sigma H$ and a $\mathcal{C}^\infty$ map $\tau : W \rightarrow G$ such that $\pi \circ \tau = \mathrm{id}_W$.
\end{enumerate}
\end{theorem}

% \begin{definition}[Homogeneous manifolds, \cite{warner2013}]
% \label{def:homogeneous_mfd}
% Manifolds of the form $G/H$, where $G$ is a Lie group, $H$ is a closed subgroup of $G$, and the manifold structure is the unique one satisfying (a) and (b) of Theorem \ref{hom_mfd} are called \emph{homogeneous manifolds}.
% \end{definition}

\begin{theorem}[\cite{warner2013}]
Let $G$ be a Lie group and $H$ a closed subgroup of $G$, then the homogeneous manifold $G/H$ with its natural group structure is a Lie group.
\end{theorem}

Hence, the quotient $G/H$ of a log-Euclidean Lie group $G$ by any closed connected normal subgroup $H \subseteq G$ is a homogeneous manifold and Lie group. Moreover, it is a standard fact in Lie groups and fiber bundles theory that the quotient of a Lie group $G/H$ is a principal $H$-bundle for right-multiplication.

\begin{theorem}[\cite{kobayashi1963}]
Let $G$ be a Lie group and let $H$ be any closed connected normal subgroup of $G$. Let $\pi : G \rightarrow G/H$ denote the natural projection, so that $G/H$ is a homogeneous manifold. Then $\pi : G \rightarrow G/H$ is a principal $H$-bundle over $G/H$ given by the right action
\begin{align*}
    G \times H \longrightarrow G, \quad (g, h) \longmapsto gh.
\end{align*}
\end{theorem}

\begin{theorem}[\cite{lee2003}]
Suppose $G$ is a Lie group and $H \subseteq G$ is a closed normal subgroup. Then the quotient map $\pi : G \rightarrow G/H$ is smooth Lie group homomorphism whose kernel is $H$.
\end{theorem}

Let us now show that any principal $H$-bundle $\pi : G \rightarrow G/H$ where $G$ is a log-Euclidean Lie group admits global sections. In particular the following proposition states that this is equivalent to the triviality of $\pi$.

\begin{proposition}[\cite{steenrod1999}]
\label{trivial_section}
A principal bundle is trivial if and only if it admits a (global) section.
\end{proposition}

\begin{theorem}
\label{trivial_thm}
Let $\pi : G \rightarrow G/H$ be a $H$-principal bundle where $\log_G : G \rightarrow \mathfrak g$ is a log-Euclidean Lie group and $H$ is a closed connected subgroup of $G$ (hence, normal since the group is abelian). Then, $\pi : G \rightarrow G/H$ is trivial.
\end{theorem}

\begin{proof}
By Proposition \ref{trivial_section}, $\pi : G \rightarrow G/H$ is trivial if and only if it admits a global section. The Lie exponential $\exp_G$ identifies $G \cong \mathfrak g$. Likewise, the closed subgroup $H \subseteq G$ identifies to a linear subspace $\mathfrak h = \log_G(H) \subseteq \mathfrak g$. The quotient manifold $G/H$ is diffeomorphic to $\mathfrak{g}/\mathfrak{h}$ via the quotient Lie logarithm map
\begin{align*}
    \overline{\log_G} : G/H \longrightarrow \mathfrak{g}/\mathfrak{h}, \quad [g] = gH \longmapsto \log_G(g) + \mathfrak h,
\end{align*}
where $g$ is any representative of $[g]$, and one can check that the definition of $\overline{\log_G}$ is independant of this choice. Now consider the smooth submersion $\pi : G \rightarrow G/H$. Its differential at the identity $e_G \in G$ is denoted
\begin{align*}
    \mathrm d_{e_G} \pi : \mathcal{T}_{e_G}G \simeq \mathfrak g \longrightarrow \mathcal{T}_{[e_G]}\left( G/H \right) \cong \mathfrak{g}/\mathfrak{h},
\end{align*}
and we have $\mathrm{ker}(\mathrm{d}_{e_G}\pi) = \mathfrak h$ and $\mathrm{im}(\mathrm{d}_{e_G}\pi) = \mathfrak{g}/\mathfrak{h}$. Thus, in coordinates, we write for any $X \in \mathfrak g$:
\begin{align*}
    \mathrm{d}_{e_G}\pi(X) = X + \mathfrak h.
\end{align*}
Let us pick any linear decomposition of $\mathfrak g$:
\begin{align*}
    \mathfrak g = \mathfrak h \oplus \mathfrak m,
\end{align*}
then, the restriction $\mathrm d_{e_G}\pi_{\vert \mathfrak m} : \mathfrak m \longrightarrow \mathfrak g / \mathfrak h$ is a linear vector space isomorphism (since its kernel is trivial), whose inverse we denote $s_{\mathrm{lin}}: \mathfrak g / \mathfrak h \rightarrow \mathfrak m \subset \mathfrak g$, i.e.,
\begin{align*}
    s_{\mathrm{lin}} \circ \mathrm d_{e_G}\pi_{\vert \mathfrak m} = \mathrm{id}_\mathfrak{m} \quad \text{and} \quad \mathrm d_{e_G}\pi_{\vert \mathfrak m} \circ \; s_{\mathrm{lin}} = \mathrm{id}_{\mathfrak g / \mathfrak h} \quad \text{hence} \quad \mathrm d_{e_G}\pi\circ \; s_{\mathrm{lin}} = \mathrm{id}_{\mathfrak g/\mathfrak h}.
\end{align*}
Define
\begin{align*}
    s : G/H \longrightarrow G, \quad [g] \longmapsto \exp_G\left( s_{\mathrm{lin}}(\log_G(g) + \mathfrak h)\right)
\end{align*}
then, $s$ is a smooth global section of $\pi$. Indeed, we have the following commutative diagram of smooth maps
    \begin{center}
    \begin{tikzcd}
    G \arrow[rr, "\pi"] \arrow[dd, "\log_G"']    &  & G/H \arrow[dd, "\overline{\log_G}"] \\
                                                 &  &                                     \\
    \mathfrak g \arrow[rr, "\mathrm d_{e_G}\pi"] &  & \mathfrak g / \mathfrak h          
    \end{tikzcd}
\end{center}
since for all $g \in G$, $\overline{\log_G}\circ \pi(g) = \overline{\log_G}([g]) = \log_G(g) + \mathfrak h= \mathrm{d}_{e_G}\pi \circ \log_G (g)$. By construction we have for any coset $[g] \in G/H$,
\begin{align*}
    & & \pi\bigl(s([g])\bigr)
    &= \pi\Bigl(\exp_G\bigl(s_{\mathrm{lin}}(\log_G(g) + \mathfrak h)\bigr)\Bigr)
    &&\quad\text{(definition of $s$)}\\
    & & 
    &= (\overline{\log_G})^{-1}\Bigl(\mathrm d_{e_G}\pi\bigl(\log_G\bigl(\exp_G(s_{\mathrm{lin}}(\log_G(g)+\mathfrak h))\bigr)\bigr)\Bigr)
    &&\quad\text{(commutativity of the diagram)}\\
    & &
    &= (\overline{\log_G})^{-1}\Bigl(\mathrm d_{e_G}\pi\bigl(s_{\mathrm{lin}}(\log_G(g) + \mathfrak h)\bigr)\Bigr)
    &&\quad(\log_G\circ\exp_G = \mathrm{id}_{\mathfrak g})\\
    & &
    &= (\overline{\log_G})^{-1}\bigl(\log_G(g) + \mathfrak h\bigr)
       = [g]
    &&\quad\text{\parbox[t]{5cm}{%
    ($\mathrm d_{e_G}\pi\circ \; s_{\mathrm{lin}}=\mathrm{id}_{\mathfrak g/\mathfrak h}$\\
    and def.\ of $\overline{\log_G}$)}}
\end{align*}
Thus, $\pi \circ s = \mathrm{id}_{G/H}$. It remains to show that $s$ is smooth, but it is clear as it is a composition of smooth maps,
\begin{align*}
    G/H \xrightarrow{\log_G + \mathfrak{h}} \mathfrak g / \mathfrak h \xrightarrow{s_{\mathrm{lin}}} \mathfrak m \xrightarrow{\exp_G} G.
\end{align*}
Therefore, $s$ is a  global section of $\pi : G \rightarrow G/H$ and this principal $H$-bundle is thus trivial, i.e., $G \cong G/H \times H$.
\end{proof}

\begin{remark}
In the definition of the section $s$, we chose its component $s_{\mathrm{lin}}$ to be linear, so that it is a morphism of Lie algebras, and so that the whole section $s$ is a morphism of Lie groups $G/H \rightarrow G$, that is, a log-Euclidean section.
\end{remark}

\begin{proposition}
\label{homomorphism_s_prop}
A global section defined as in the proof of Theorem \ref{trivial_thm} is furthermore a Lie group homomorphism.
\end{proposition}

\begin{proof}
Let $\left(G, \star \right)$ be a log-Euclidean Lie group and let $[g], [g'] \in \left(G/H, \cdot \right)$, then,
\begin{align*}
    s\left([g] \cdot [g']\right) = s\left( [g \star g']\right) &= \exp_G\left( s_{\mathrm{lin}}(\overline{\log_G}([g \star g']))\right)\\
    &= \exp_G\left( s_{\mathrm{lin}}(\log_G(g \star g') + \mathfrak h)\right)\\
    &= \exp_G\left( s_{\mathrm{lin}}(\log_G(g) + \log_G(g') + \mathfrak h)\right)\\
    &= \exp_G \left( s_{\mathrm{lin}}(\log_G(g) + \mathfrak h) + s_{\mathrm{lin}}(\log_G(g') + \mathfrak h)\right)\\
    &= \exp_G \left( s_{\mathrm{lin}}(\log_G(g) + \mathfrak h) \right) \star \exp_G \left( s_{\mathrm{lin}}(\log_G(g') + \mathfrak h)\right)\\
    &= s([g])\star s([g'])
\end{align*}
where we used that $\log_G(g\star g') = \log_G(g) + \log_G(g')$ since the group is commutative, that $s_{\mathrm{lin}}$ is linear and that $\exp_G$ is a group homomorphism $\left( \mathfrak g, + \right) \rightarrow \left(G, \star \right)$. 
\end{proof}

Since the kernel of $s$ is trivial, we can deduce from the first isomorphism theorem the following corollary.

\begin{corollary}
\label{homomorphism_s_corollary}
The image of the global section $s$ is a subgroup of $G$ isomorphic to $G/H$.
\end{corollary}

We now have two subgroups of $(G, \star)$:
\begin{align*}
    H \quad \text{and} \quad K := s(G/H) \cong G/H.
\end{align*}
They satisfy the internal direct product theorem conditions
\begin{enumerate}
    \item \emph{Trivial intersection}: $H \cap K = \{ e_G \}$. Indeed, if $x \in H \cap K$, then $\pi(x) = [e_G]$ so $x = s ([e_G]) = e_G $.
    \item \emph{Generation}: every $g \in G$ can be written uniquely as $g =  h \star k$ for some $h \in H, k \in K$. Indeed, we have $G=HK$. For any $g \in G$, write $k = s \circ \pi (g)$ and $h = k^{-1} \star g$, then $\pi(g) = \pi(k)$ so $k^{-1} \star g \in \mathrm{ker}(\pi) = H$, hence, $h \in H$, therefore $g = k \star h$ showing that $G = KH = HK$ (commutativity). Since the intersection is trivial, uniqueness follows.
    \item \emph{Normality}: since $\star$ is abelian, every closed subgroup is normal. 
\end{enumerate}
Therefore, the map
\begin{align*}
    \Phi : H \times K \longrightarrow G, \quad \Phi(h, k) = h \star k
\end{align*}
is a group isomorphism and we write $G \simeq K \times H$ as the internal direct product of $K$ and $H$. In fact, the following proposition generalizes the above result to any global section that is a Lie group homomorphism.

\begin{proposition}
A smooth section $\sigma : G/H \rightarrow G$ is a Lie group homomorphism if and only if $G$ splits as an internal direct product $G = K \times H$, $K:= \sigma(G/H)$.
\end{proposition}

\begin{proof}
Assume $\sigma : G/H \rightarrow G$ is a Lie group homomorphism with $\pi \circ \sigma = \mathrm{id}_{G/H}$. By the first isomorphism theorem its image $K = \sigma(G/H)$ is a Lie subgroup isomorphic to $G/H$. For any $g \in G$, denote $\pi(g) = [g] \in G/H$. Then
\[
\pi\bigl(g \star \sigma([g])^{-1}\bigr)= [g]\star [g]^{-1} = [e],
\]
so $g \star \sigma([g])^{-1} \in \ker(\pi)=H$. Thus any $g\in G$ can be written
\[
g=\sigma([g]) \star h
\]
with $h \in H$. Hence $G = KH$. Moreover, if $x\in K\cap H$, then $x=\sigma([g])$ for some $[g]\in G/H$, and
\[
[e]=\pi(x)=\pi(\sigma([g]))=[g],
\]
so $x=\sigma([e])=e$. Thus $K\cap H=\{e\}$. Since $H$ is normal in $G$ and $G$ is abelian, we obtain the internal direct product $G = K \times H$.

Conversely, if $G = K \times H$, every coset in $G/H$ has a unique representative in $K$. Define $\sigma([g]) = k$, where $k\in K$ is the unique element such that $[g]=[k]$. Then $\pi\circ \sigma=\mathrm{id}_{G/H}$. For $[g_1],[g_2]\in G/H$, writing $\sigma([g_1])=k_1$ and $\sigma([g_2])=k_2$, we have
\[
[g_1]\star [g_2]=[k_1]\star [k_2]=[k_1\star k_2],
\]
hence
\[
\sigma([g_1]\star [g_2]) = k_1\star k_2 = \sigma([g_1]) \star \sigma([g_2]).
\]
Therefore $\sigma$ is a Lie group homomorphism.
\end{proof}

Every \emph{log-Euclidean} section (i.e. group-homomorphic sections) inherits its log-Euclidean Lie group structure from the one of $G$. Assume $H\subset G$ is closed, normal and connected; then $H=\exp_G(\mathfrak h)$ for a linear subspace $\mathfrak h\subset\mathfrak g$. Because $G\cong(\mathbb R^{n},+)$ is connected, simply connected and abelian, the quotient
\begin{align*}
    G/H \;\cong\; \mathfrak g/\mathfrak h \;\cong\; \mathbb R^{\,n-\dim\mathfrak h}
\end{align*}
is also connected, simply connected and abelian. Let $\sigma:G/H\rightarrow G$ be a log-Euclidean section and let $K$ denote the image of $\sigma$. Since $\sigma$ is a diffeomorphism onto its image,
$K$ is connected, simply connected and abelian as well. Moreover the restrictions
\begin{align*}
\exp_G\bigl|_{\mathfrak m}:\mathfrak m\;\longrightarrow\;K, \qquad \log_G\bigl|_{K}:K\;\longrightarrow\;\mathfrak m, \quad \mathfrak m:=\mathrm{im}\bigl(\mathrm d_{e_G}\sigma\bigr),
\end{align*}
are global diffeomorphisms. Hence $K$ is diffeomorphic to some finite-dimensional vector space and is thus a \textit{log-Euclidean Lie group}.
Let us now characterize the space of all group-homomorphic sections at the Lie algebras level. 

\begin{theorem}
Consider the principal $H$-bundle $\pi : G \rightarrow G/H$ where $G$ is a log-Euclidean Lie group and let $H$ be a closed connected subgroup (hence, normal) of $G$. Then there is a one-to-one correspondence between group homomorphic sections of $\pi$ and $\mathbb{R}$-linear splittings of the short exact sequence, 
\begin{align*}
    0 \longrightarrow \mathfrak h \xhookrightarrow{i} \mathfrak g \xrightarrow{\mathrm d_{e_G}\pi} \mathfrak{g} / \mathfrak{h} \longrightarrow 0.
\end{align*}
\end{theorem}

\begin{proof}
Let us first prove the necessary direction. Assume that $s_\mathrm{lin}$ is a $\mathbb{R}$-linear splitting of the above short exact sequence. Then, $\mathrm{d}_{e_G}\pi \circ \; s_\mathrm{lin} = \mathrm{id}_{\mathfrak g /\mathfrak h}$. One can descend $s_\mathrm{lin}$ to a smooth section $s$ of $\pi$ defined as in the proof of Theorem \ref{trivial_thm}
\begin{align*}
    s = \exp_G \circ \; s_\mathrm{lin} \circ \overline{\log_G} : G/H \longrightarrow G
\end{align*}
and by writing $\pi = \overline{\log_G}^{-1} \circ \mathrm d_{e_G}\pi \circ \log_G$ (commutativity of the diagram in the proof of Theorem \ref{trivial_thm}) one can show in a completely similar manner that this a well-defined global smooth section of $\pi$ that is furthermore a group-homomorphism (Proposition \ref{homomorphism_s_prop}). Conversely, assume that $s : G/H \rightarrow G$ is a group-homomorphic section. The short exact sequence of abelian groups
\begin{align*}
    0 \longrightarrow H \xhookrightarrow{i} G \xrightarrow{\pi} G/H \longrightarrow 0
\end{align*}
splits if and only if $\pi$ admits a right-inverse, in which case $G = G/H \oplus H$. Splitting the sequence is equivalent to choose a group-homomorphic section, i.e., a smooth homomorphism $s : G/H \rightarrow G$ such that $\pi \circ \; s = \mathrm{id}_{G/H}$. One can lift the section $s$ to a $\mathbb{R}$-linear splitting $s_\mathrm{lin}$ of the short exact sequence 
\begin{align*}
    0 \longrightarrow \mathfrak h \xhookrightarrow{i} \mathfrak g \xrightarrow{\mathrm d_{e_G}\pi} \mathfrak{g} / \mathfrak{h} \longrightarrow 0
\end{align*}
by writing 
\begin{align*}
    s_{\mathrm{lin}} := \log_G \circ \; s \circ (\overline{\log_G})^{-1} : \mathfrak g / \mathfrak h \rightarrow \mathfrak g.
\end{align*}
Clearly, $s_\mathrm{lin}$ is smooth as a composition of smooth maps, and the equality $\mathrm d_{e_G}\pi \circ \; s_\mathrm{lin} = \mathrm{id}_{\mathfrak g / \mathfrak h}$ falls directly from rewriting $\mathrm d_{e_G}\pi$ as $\overline{\log_G}\circ \pi(g) \circ \exp_G$.
\end{proof}

Moreover, identifying group-homomorphic sections with $\mathbb{R}$-linear splittings of the short exact sequence
\begin{align*}
    0 \longrightarrow \mathfrak h \xhookrightarrow{i} \mathfrak g \xrightarrow{\mathrm d_{e_G} \pi} \mathfrak{g} / \mathfrak{h} \longrightarrow 0
\end{align*}
is equivalent to choosing a subspace $\mathfrak m \subset \mathfrak g$ such that $\mathfrak g = \mathfrak h \oplus \mathfrak m$ (because $\mathrm{im}(s_\mathrm{lin}) \cap \mathfrak h = \{ 0 \}$ and one sets $\mathrm{im}(s_\mathrm{lin}) =: \mathfrak m$). Thus, there is a one-to-one correspondence
\begin{align*}
\scalebox{0.9}{$
\begin{aligned}
    \left\{
    \begin{array}{l}
        \text{log-Euclidean sections}\\
        s : G/H \rightarrow G
    \end{array}
    \right\}
    \;\longleftrightarrow\;
    \left\{
    \begin{array}{l}
        \text{splittings of $G$ as an}\\
        \text{internal direct product}\\
        G = K \times H, \quad K := \mathrm{im}(s)
    \end{array}
    \right\}
    \;\longleftrightarrow\;
    \left\{
    \begin{array}{l}
        \text{choice of supplementary}\\
        \text{space } \mathfrak g= \mathfrak h \oplus \mathfrak m
    \end{array}
    \right\}.
\end{aligned}
$}
\end{align*}
\begin{remark}
    If we further endow $G$ with a bi-invariant metric (i.e. a log-Euclidean metric), then any immersion $s: G/H \rightarrow G$ that is also a group homomorphism is furthermore a totally geodesic immersion. For the general statement of this result see for instance \cite{doCarmo1992}. Observe that no property on the curvature of $G$ is required here. 
\end{remark}

\subsection{Riemannian submersion and isometric embeddings}

In this section we study quotients of log-Euclidean Lie groups endowed with a log-Euclidean metric, for which the quotient map becomes a Riemannian submersion. In particular, we are interested in the characterization of the global sections that are isometric embeddings. We make the link with the previous section by showing that there exists a unique \textit{canonical section} which is both \textit{log-Euclidean} (i.e., a group-isomorphism) and an \textit{isometric embedding} of the quotient space in the total space. In which case, the quotient metric coincides with the induced metric on the image of the section, which furthermore is a subgroup of the total space. Namely, when $G$ is endowed with a log-Euclidean metric $g$, the canonical section $s : (G/H, g^Q) \hookrightarrow (K \subset G, g)$ is an isometric embedding:
\begin{align*}
    g^Q = s^*g_{\vert K},
\end{align*}
where $K$ denote the image of $s$. Explicitly, for any coset $[p] \in G/H$ and quotient tangent vectors $v, w \in \mathcal{T}_{[p]}(G/H)$,
\begin{align*}
    g_{s([p])}\bigl( \mathrm{d}_{[p]} s(v), \mathrm{d}_{[p]} s(w) \bigr) = g^Q_{[p]} \bigl(v, w \bigr). 
\end{align*}
In particular, $g_{\vert K} = \pi^*g^Q$. Finally, we derive the horizontal lift associated to the Riemannian submersion $\pi$, which allows us to compute the quotient metric and its geodesics. 

\medskip

Let us equip $(G, \star)$ with a log-Euclidean metric $g$. As before, let $H$ denote a closed connected subgroup (hence, normal) of $G$, then we have shown in the previous section that $\pi : G \rightarrow G/H$ is a trivial principal $H$-bundle. Any log-Euclidean metric is by definition bi-invariant with respect to the log-Euclidean Lie group action "$\star$", therefore, there exists a unique Riemannian metric $g^{\mathrm Q}$ on $G/H$ that turns $\pi$ into a Riemannian submersion \cite{oneill1966}, that is, 
\begin{align*}
    \mathrm d_p{\pi} :\mathrm{ker}(\mathrm d_p \pi)^\bot \rightarrow \mathcal{T}_{\pi(p)}\left( G/H \right)
\end{align*}
is a linear isometry for all $p \in G$ and we denote the \textit{horizontal space} at $p \in G$ with $\mathcal{H}_p := \mathrm{ker}(\mathrm d_p \pi)^\bot$, and the \textit{vertical space} at $p \in G$ as $\mathcal{V}_p := \mathrm{ker}(\mathrm d_p \pi)$ so that, clearly, $\mathcal{T}_pG = \mathcal{H}_p \oplus \mathcal{V}_p$. In the case of log-Euclidean Lie groups, this splitting translates directly in the Lie algebra, at each $p \in G$, $\mathcal{T}_pG \cong (L_p)_* \mathfrak{g} = (L_p)_* \mathfrak{h} \oplus (L_p)_* \mathfrak{m}$ where $\mathfrak h = \mathcal{V}_{e_G}$ and $\mathfrak m = \mathcal{H}_{e_G}$. Recall that a smooth section of a Riemannian submersion is an isometric embedding if and only if its differential is horizontal. 

\begin{proposition}
\label{isometric_prop}
A smooth section $\sigma : G/H \rightarrow G$ is an isometric embedding if and only if for all $[p] \in G/H$ and for all $v \in \mathcal{T}_{[p]}\left( G/H \right)$, $\mathrm d_{[p]}\sigma(v) \in \mathcal{H}_{\sigma([p])}$. 
\end{proposition}

\begin{proof}
Let $[p] \in G/H$ and let $v \in \mathcal{T}_{[p]}\left( G/H \right)$. Assume $\Vert \mathrm d_{[p]}\sigma(v) \Vert_g = \Vert v \Vert_{g^Q}$, by definition of the quotient metric, 
\begin{align*}
    g_{[p]}^Q(v, v) = g_{\sigma([p])}(w, w)
\end{align*}
whenever $\mathrm d_{\sigma([p])}\pi(w) = v$ and $w \in \mathcal{H}_{\sigma([p])}$. Because $\pi \circ \sigma = \mathrm{id}_{G/H}$, we have $\mathrm{d} \pi \circ \mathrm{d} \sigma = \mathrm{id}_{\mathfrak g / \mathfrak h}$ and let us write $w = \mathrm d_{[p]}\sigma(v)$. We show that $w$ is the horizontal lift along $\sigma$ of $v$. Since $\mathcal{T}_{\sigma([p])}G = \mathcal{H}_{\sigma([p])} \oplus \mathcal{V}_{\sigma([p])}$ we write $w = w_\mathcal{V} + w_\mathcal{H}$ the sum of the vertical and horizontal parts of $w$, respectively. By canceling the cross terms in the computation of the squared norm, we get
\begin{align*}
    \Vert w \Vert_g^2 = \Vert w_\mathcal{V} \Vert_g^2 + \Vert w_\mathcal{H} \Vert_g^2 \geq \Vert w_\mathcal{H} \Vert_g^2 = \Vert \mathrm{d}_{\sigma([p])}\pi(w_\mathcal{H}) \Vert_{g^Q}^2 = \Vert v \Vert_{g^Q}^2.
\end{align*}
The equality is achieved precisely when $w$ has no vertical components, hence is horizontal. By uniqueness it must be the horizontal lift along $\sigma$ of $v$. Conversely, assume that $\sigma$ is horizontal and let $v \in \mathcal{T}_{[p]}\left( G/H \right)$, then since
\begin{align*}
    \mathrm{d}_{\sigma([p])}\pi : \mathcal{H}_{\sigma([p])} \rightarrow \mathcal{T}_{[p]}\left( G/H \right)
\end{align*}
is a linear isometry, we have 
\begin{align*}
    \Vert \mathrm{d}_{[p]}\sigma(v) \Vert_g = \Vert \mathrm{d}_{\sigma([p])}\pi(\mathrm{d}_{[p]}\sigma(v)) \Vert_{g^Q} = \Vert v \Vert_{g^Q}
\end{align*}
that is, $\sigma$ is an isometric embedding.
\end{proof}

Define a splitting of the short exact sequence
\begin{align*}
    0 \longrightarrow \mathfrak h \xhookrightarrow{i} \mathfrak g
    \;\overset{\mathrm d_{e_G}\pi}{\underset{s_{\mathrm{lin}}}{\Longrightleftarrows}}\;
    \mathfrak g / \mathfrak h \longrightarrow 0 .  
\end{align*}
given by the map
\begin{align*}
    s_{\mathrm{lin}} : \mathfrak g / \mathfrak h \rightarrow \mathfrak m := \mathfrak h^\bot = \mathcal{H}_{e_G} \quad \text{ with } \quad s_{\mathrm{lin}} = \left(\mathrm d_{e_G} \pi_{\vert \mathfrak m}\right)^{-1}.
\end{align*}
In other words, $s_{\mathrm{lin}}$ is a smooth choice of representatives in $\mathfrak m := \mathfrak{h}^\bot$ of cosets in $\mathfrak g / \mathfrak h$, more precisely, $s_{\mathrm{lin}}$ is exactly the horizontal lift (equivalently, the horizontal projection) to $\mathcal{H}_{e_G}$. Because $\mathrm d_{e_G} \pi_{\vert \mathfrak m}$ is an isometry, so is its right-inverse, $s_{\mathrm{lin}}$. Since $s_{\mathrm{lin}}$ (and therefore its differential) lands in the horizontal space $\mathcal H_{e_G}$, and $\mathrm d\exp_G$ sends horizontals to horizontals, $\mathrm d_{[p]}s$ is horizontal for all $[p] \in G/H$, and by Proposition \ref{isometric_prop}, $s$ is therefore an isometric section. Let us show that any other isometric section is a vertical translation of the section $s$ by a fixed element of $H$.

\begin{proposition}
\label{isometric_s_prop}
Let $\sigma : G/H \rightarrow (G, \star)$ be a smooth section, then $\sigma$ is an isometric embedding if and only if there exists $h_0 \in H$ such that $\sigma([p]) = s([p]) \star h_0$ for all $[p] \in G/H$.
\end{proposition}

\begin{proof}
Let $\sigma : (G/H, \cdot) \rightarrow (G, \star)$ be an isometric embedding. Then by Proposition \ref{isometric_prop} its differential at any point is horizontal, that is, for all $[p]\in G/H$ and all $v \in \mathcal{T}_{[p]}(G/H)$,
\begin{align*}
    \mathrm d_{[p]}\sigma(v) \in \mathcal{H}_{\sigma([p])}.
\end{align*}
Define the map $G/H \rightarrow H$ by
\begin{align*}
    [p]\mapsto s([p])^{-1} \star \sigma([p]),
\end{align*}
because $\pi(s([p])^{-1} \star \sigma([p])) = \pi(s([p])^{-1}) \cdot \pi(\sigma([p])) = [p]^{-1} \cdot [p] = [e_G]$ so indeed, $[p]\mapsto s([p])^{-1} \star \sigma([p])$ lies in $\mathrm{ker}(\pi) =H$. We will show this map is constant. Take its differential at any coset $[p]\in G/H$, because $G$ is log-Euclidean, it is simpler to work directly in the log charts. Set
\begin{align*}
    u([p]) := \log_G(s([p])) \quad \text{and} \quad v([p]) := \log_G(\sigma([p])).
\end{align*}
Then because $G$ is abelian,
\begin{align*}
    \log_G(s([p])^{-1} \star \sigma([p])) = - u([p]) + v([p]) \qquad \hfill \text{(additivity of the log)}.
\end{align*}
Taking differentials yields
\begin{align*}
    \mathrm d_{[p]}u &= ( L_{s([p])^{-1}} )^* \mathrm d_{[p]}s, 
    \\
    \mathrm d_{[p]}v &= ( L_{\sigma([p])^{-1}} )^* \mathrm d_{[p]}\sigma.
\end{align*}
Hence, writing for all $p \in G$, $\mathrm d_p \log_G = (L_{p^{-1}})^*$ in a log-Euclidean Lie group (Lemma ~\ref{lem:LE-dlog-dexp}), we have
\begin{align*}
    \mathrm  d_{[p]} \log_G (s^{-1}\star \sigma)
    = ( L_{\sigma([p])^{-1}} )^* \mathrm d_{[p]}\sigma - ( L_{s([p])^{-1}} )^* \mathrm d_{[p]}s.
\end{align*}
Since both $\sigma$ and $s$ are horizontal isometric embeddings, we have
\begin{align*}
    \mathrm d_{[p]}\sigma \in \mathcal{H}_{\sigma([p])} \quad\text{and}\quad \mathrm d_{[p]}s \in \mathcal{H}_{s([p])},
\end{align*}
so that,
\begin{align*}
    ( L_{\sigma([p])^{-1}} )^* \mathrm d_{[p]}\sigma \in \mathcal{H}_{e_G} \quad\text{and}\quad ( L_{s([p])^{-1}} )^* \mathrm d_{[p]}s \in \mathcal{H}_{e_G},
\end{align*}
and since $\mathcal{H}_{e_G}$ is a linear subspace of $\mathfrak g$, the difference $( L_{\sigma([p])^{-1}} )^* \mathrm d_{[p]}\sigma - ( L_{s([p])^{-1}} )^* \mathrm d_{[p]}s$ still belongs to $\mathcal{H}_{e_G}$. Thus, at each coset $[p] \in G/H$, the differential $\mathrm  d_{[p]} \log_G (s^{-1}\star \sigma)$ lies entirely in the horizontal space. However, since $\mathrm  d_{[p]} \log_G (s^{-1}\star \sigma)$ is also tangent to $H$ (as $[p]\mapsto s([p])^{-1} \star \sigma([p])$ takes values in $H$), it must lie in the vertical space. The only vector in both horizontal and vertical spaces is the zero vector, hence
\begin{align*}
    \mathrm d_{[p]}(s^{-1} \star \sigma)(v)=0\quad\text{for all } v\in \mathcal{T}_{[p]}\left(G/H\right).
\end{align*}
Since the differential vanishes at every point, the map $[p]\mapsto s([p])^{-1}\star \sigma([p])\in H$ is locally constant, and by connexity of $G/H$, it must be globally constant. Call this constant $h_0$. Thus, for all $[p]\in G/H$,
\begin{align*}
    \sigma([p])=s([p]) \star h_0.
\end{align*}
This proves the claim. The converse follows directly from horizontality of $s$ and the fact that left and right translations are isometries in log-Euclidean Lie groups.
\end{proof}

Let us finally show that the section $s$ we defined is the only isometric section that is also a Lie group homomorphism whenever $s_{\mathrm{lin}}$ is chosen to be linear. In that sense, we call this section \textit{canonical}.

\begin{theorem}
\label{canonical_section_thm}
Let $(G, \star, g)$ be a log-Euclidean Lie group and $H$ a connected closed subgroup of $G$. Let $s$ be a section of the principal $H$-bundle $\pi : G \rightarrow G/H$ defined by
\begin{align*}
    s : G/H \longrightarrow G, \quad [g] \longmapsto \exp_G\left( s_{\mathrm{lin}}(\log_G(g) + \mathfrak h)\right)
\end{align*}
where $s_\mathrm{lin}$ is the linear splitting defined as
\begin{align*}
    s_{\mathrm{lin}} : \mathfrak g / \mathfrak h \rightarrow \mathfrak m := \mathfrak h^\bot = \mathcal{H}_{e_G} \quad \text{ with } \quad s_{\mathrm{lin}} = \left(\mathrm d_{e_G} \pi_{\vert \mathfrak m}\right)^{-1}.
\end{align*}
Then $s$ is the unique section of $\pi$ that is an isometric embedding $(G/H, g^Q) \xhookrightarrow{s} (G, g)$ isomorphic (as log-Euclidean Lie groups) to $G/H$.
\end{theorem}

\begin{proof}
We have shown that $s : (G/H, \cdot) \rightarrow (G, \star)$ is an isometric embedding (Proposition \ref{isometric_prop} with the fact that for all $[p]\in G/H$, $\mathrm{d}_{[p]} s \in \mathcal{H}_{s([p])}$) and that it is furthermore a log-Euclidean Lie group (Proposition \ref{homomorphism_s_prop}) isomorphic to $G/H$ (Corollary \ref{homomorphism_s_corollary}). In virtue of Proposition \ref{isometric_s_prop}, any other such isometric section $\sigma : G/H \rightarrow G$ must be obtained as a vertical translation of $s$, i.e., there exists $h_0 \in H$ such that $\sigma([p]) = s([p]) \star h_0$ for all $[p]\in G/H$. Now let $[p], [p'] \in (G/H, \cdot)$, then 
\begin{align*}
    \sigma([p] \cdot [p']) = \sigma([p]) \star \sigma([p']) \Longleftrightarrow s([p] \cdot [p']) \star h_0 = s([p]) \star h_0 \star s([p']) \star h_0,
\end{align*}
Because G is abelian, $s([p]) \star h_0 \star s([p']) \star h_0 = s([p] \cdot [p']) \star h_0 \star h_0$ which implies $h_0 \star h_0 = h_0$ which forces $h_0=e_G$. Hence $\sigma = s \star e_G = s$ is the only group-isomorphism section which is also an isometric embedding.
\end{proof}

In other words, there exists a unique log-Euclidean section that isometrically embeds $G/H$ in $G$. As a consequence one has the following theorem asserting that not only $G$ is isomorphic to the internal direct product $s(G/H) \times H$ but this isomorphism is furthermore a Riemannian isometry, in other words, $G$ metrically splits orthogonally onto its factors.

\begin{theorem}
\label{thm:can_split}
Let $(G,\star,g)$ be a log-Euclidean Lie group and $H$ a connected closed subgroup of $G$. Let $s : G/H \rightarrow G$ be the canonical section of $\pi : G \rightarrow G/H$, and denote its image by $K = s(G/H)$. Then the map
\begin{align*}
    \Phi : (H,g|_H) \times (K,g|_K) \longrightarrow (G,g), \quad (h,k)\longmapsto h\star k 
\end{align*}
is a Riemannian isometry. Thus, $(G,g)$ splits as an orthogonal internal direct product of $(H,g_{\vert H})$ and $(K,g_{\vert K})$, that is, $g = g_{\vert H} + g_{\vert K}$.
\end{theorem}

\begin{proof}
The fact that $\Phi$ is a Lie group isomorphism was already established in Section \ref{isomorphic_section}. Let us show that $\Phi$ is furthermore a Riemannian isometry. Let $(h,k)\in H\times K$ and consider arbitrary tangent vectors $(u,v)\in \mathcal{T}_hH\times \mathcal{T}_kK$. Since $g$ is log-Euclidean, $g$ is bi-invariant, left and right translations in $G$ are isometries. Since
\begin{align*}
    \mathrm d_{(h,k)}\Phi(u,v) = \mathrm d_k L_h(v) + \mathrm d_h R_k(u),
\end{align*}
using the orthogonality $\mathfrak g = \mathfrak h\oplus \mathfrak m$ at the identity, we obtain
\begin{align*}
    g_{h\star k}(\mathrm d_{(h,k)}\Phi(u,v), \mathrm d_{(h,k)}\Phi(u,v))
    &= g_{h\star k}(\mathrm d_kL_h(v) + \mathrm d_hR_k(u), \mathrm d_kL_h(v) + \mathrm d_hR_k(u)) \\
    &= g_{h\star k}(\mathrm d_kL_h(v), \mathrm d_kL_h(v)) + g_{h\star k}(\mathrm d_hR_k(u), \mathrm d_hR_k(u)) \\
    & \quad + 2 g_{h\star k}(\mathrm d_kL_h(v),\mathrm d_hR_k(u)) \\
    &= g_{h\star k}(\mathrm d_kL_h(v), \mathrm d_kL_h(v)) + g_{h\star k}(\mathrm d_hR_k(u), \mathrm d_hR_k(u)) \\
    &= g_h(u,u) + g_k(v,v)
\end{align*}
with no cross terms, as $\mathrm d_kL_h(v) $ and $\mathrm d_hR_k(u)$ remain orthogonal with respect to $g$ by bi-invariance. Hence, $\Phi$ is a Riemannian isometry.
\end{proof}

\begin{remark}
The preceding argument gives the abstract orthogonal splitting theorem for arbitrary log-Euclidean Lie groups. The decomposition of Theorem~\ref{metric_split_LE_thm} is the special case $G=\mathcal S^+(n)$ and $H=\mathrm{Diag}^+(n)$, with complementary factor $K=\exp(\mathrm{Hol}(n))$.
\end{remark}

\begin{proposition}
\label{horizontal_lift_prop}
Let $\pi : (G,g) \rightarrow (G/H,g^Q)$ be the Riemannian submersion described above, and let $s : G/H \rightarrow G$ be the canonical section defined by the linear orthogonal splitting $s_{\mathrm{lin}} : \mathfrak g/\mathfrak h \rightarrow \mathfrak m := \mathfrak h^\bot$. Then the horizontal lift of a tangent vector $v \in \mathcal{T}_{[p]}(G/H)$ at an arbitrary point $p \in G$ with $\pi(p)=[p]$ is given explicitly by
\begin{align*}
    \mathrm{Hor}_p(v) = (L_p)^*\left(s_{\mathrm{lin}}(\tilde v)\right).    
\end{align*}
where we set $\tilde v \;:=\;\bigl(\mathrm d_{[e_G]}L_{[p]}\bigr)^{-1}(v) + \mathfrak h \;\in\;\mathfrak g/\mathfrak h$. Moreover, if one picks any lift $v^\mathcal{H} \in \mathcal{T}_pG$ of $v$, the formula for the horizontal lift becomes
\begin{align*}
    \mathrm{Hor}_p(v) = (L_p)^*\left(s_{\mathrm{lin}}\bigl( \mathrm{d}_p{\log_G(v^\mathcal{H})} + \mathfrak h \bigr)\right).
\end{align*}
In particular, for any smooth section $\sigma : G/H \rightarrow G$, the horizontal lift along $\sigma$ at the point $[p] \in G/H$ is
\begin{align*}
    \mathrm{Hor}_{\sigma([p])}(v) = (L_{\sigma([p])})^*\left(s_{\mathrm{lin}}(\tilde v)\right). 
\end{align*}
\end{proposition}

\begin{proof}
By definition of a Riemannian submersion, the differential $\mathrm d_p\pi : \mathcal{H}_p \rightarrow \mathcal{T}_{[p]}(G/H)$ is a smooth isomorphism and a linear isometry onto its image for each $p\in G$. Thus, the horizontal lift at $p$ of a tangent vector $v\in \mathcal{T}_{[p]}(G/H)$ is the unique vector in $\mathcal{H}_p$ projecting onto $v$. Because $\pi$ is a group morphism, $\pi \circ L_p = L_{[p]} \circ \pi$, and differentiating this equality yields
\begin{align*}
    \mathrm d_{e_G} \bigl( \pi \circ L_p \bigr)
    \;=\;
    \mathrm d_{e_G} \bigl( L_{[p]}\circ \pi \bigr).
\end{align*}
Since $\mathrm d_{e_G} \pi\bigl(s_{\mathrm{lin}}(\tilde v)\bigr)=\tilde v$ by definition of $s_{\mathrm{lin}}$, we get
\begin{align*}
    \mathrm d_p\pi\left((L_p)^* s_{\mathrm{lin}}(\tilde v)\right) 
    \;=\; \mathrm d_{[e_G]}L_{[p]}(\tilde v)
    \;=\; v.
\end{align*}
It remains to compute $\tilde v$ in terms of Lie logarithms,
\begin{align*}
    \tilde v &:= \bigl(\mathrm d_{[e_G]}L_{[p]}\bigr)^{-1}(v) + \mathfrak h \\
    &= \mathrm d_{[p]} L_{[p]^{-1}} (v) + \mathfrak h.
\end{align*}
Let $v^\mathcal{H} \in \mathcal{T}_pG$ be any lift of $v$ at $p \in G$, i.e. $\mathrm{d}_p \pi (v^\mathcal{H}) = v$, where $\pi(p) = [p] \in G/H$. Then 
\begin{align*}
    \mathrm d_{[p]} L_{[p]^{-1}} (v) &= \mathrm{d}_p\bigl( \pi \circ L_{p^{-1}} \bigr)(v) \\
    &= \mathrm d_{e_G} \pi \bigl( \mathrm d_p L_{p^{-1}} (v^\mathcal{H}) \bigr)
\end{align*}
Now since for all $x \in G$, $L_{p^{-1}}(x) = \exp_G \bigl( - \log_G(p) + \log_G(x) \bigr)$ we have by the chain rule
\begin{align*}
    \mathrm d_p L_{p^{-1}} (v^\mathcal{H}) &= \mathrm{d}_{(- \log_G(p) + \log_G(p)) } \exp_G \bigl( \mathrm d_p \log_G (v^\mathcal{H}) \bigr) \\
    &= \mathrm d_0 \exp_G \bigl(\mathrm d_p \log_G (v^\mathcal{H}) \bigr) \\
    &= \mathrm d_p \log_G (v^\mathcal{H})
\end{align*}
where we used the fact that $\mathrm d_x \bigl( x \mapsto -\log_G(p) + \log_G(x) \bigr) = \mathrm d_x \log_G $ and we set $x=p$. Thus, we have
\begin{align*}
     \tilde v = \mathrm d_{e_G} \pi \bigl( \mathrm d_p \log_G (v^\mathcal{H}) \bigr) = \mathrm d_p \log_G (v^\mathcal{H}) + \mathfrak h.
\end{align*}
Note that any two lifts $v^\mathcal{H}$ differ by an element of $\mathrm{ker}( \mathrm d_p \pi) = (L_p)_* \mathfrak h$, so the coset $\mathrm d_p \log_G (v^\mathcal{H}) + \mathfrak h$ is well-defined. Moreover $s_{\mathrm{lin}}(\tilde v)\in\mathfrak m=\mathcal{H}_{e_G}$ implies $(L_p)^*s_{\mathrm{lin}}(\tilde v)\in\mathcal H_p$.  Hence $(L_p)^*s_{\mathrm{lin}}(\tilde v)$ is the unique horizontal vector projecting to $v$, proving the first claim. The second follows immediately by taking $p=\sigma([p])$
\end{proof}

The characterization of totally geodesic submanifolds in a log-Euclidean Lie group $G$ is straightforward: since the Lie exponential is a Riemannian isometry, these submanifolds are precisely the images under the Lie exponential map of linear subspaces of the Lie algebra $\mathfrak g$.

\subsection{Quotient log-Euclidean metrics on SPD and correlation matrices}
\label{sec:quotient_appli}

Let us apply the results derived in the previous sections to the smooth manifolds of SPD matrices and full-rank correlation matrices. Let $\log: \left( \mathcal{S}^+(n), \star \right) \rightarrow \mathcal{S}(n)$ denote the log-Euclidean Lie group of SPD matrices whose group action is defined by
\begin{align*}
    \Sigma_1 \star \Sigma_2 := \exp \left( \log(\Sigma_1) + \log(\Sigma_2) \right), \quad \text{for all} \quad \Sigma_1, \Sigma_2 \in \mathcal{S}^+(n),
\end{align*}
where $\exp$ and $\log$ are the usual matrix exponential and matrix logarithm, respectively. Let $\left(\mathrm{Diag}^+(n), \star \right)$ denote the closed connected (normal) Lie subgroup of $\left( \mathcal{S}^+(n), \star \right)$ consisting of positive diagonal matrices. The Lie algebra of $\mathrm{Diag}^+(n)$ is given by the vector space $\mathrm{Diag}(n)$ of diagonal matrices, which is a subspace of the Lie algebra of $\mathcal{S}^+(n)$, namely, a subspace of symmetric matrices $\mathcal{S}(n)$. Positive diagonal matrices act smoothly, properly and freely on the right (for the \textit{log-Euclidean group action}) on SPD matrices so that the quotient map 
\begin{align*}
  \pi : \mathcal{S}^+(n) &\longrightarrow \mathcal{S}^+(n) / \mathrm{Diag}^+(n) \\
  \Sigma &\longmapsto [\Sigma] = \Sigma \star \mathrm{Diag}^+(n)
\end{align*}
is a principal $\mathrm{Diag}^+(n)$-bundle which is furthermore a Riemannian submersion when $\mathcal{S}^+(n)$ is equipped with its log-Euclidean metric $g^\mathrm{LE}$ defined over the Frobenius inner product on $\mathcal{S}(n)$, and $\mathcal{S}^+(n) / \mathrm{Diag}^+(n)$ is endowed with the quotient metric $g^Q$. We remind that the considered quotient space here is given by the log-Euclidean group action of $\mathrm{Diag}^+(n)$ on $\mathcal{S}^+(n)$ instead of the usual congruence action. In the preceding section, we defined to any Riemannian submersion of log-Euclidean Lie groups the canonical section to be the unique global section that is both a group isomorphism (i.e. inherits \textit{log-Euclideanity}) and an isometric embedding. Let us first provide the canonical section of $\pi$.

\begin{theorem}
\label{thm:can_section_appli}
    Let $\pi : \left(\mathcal{S}^+(n), g^{\mathrm{LE}}, \star \right) \rightarrow \left(\mathcal{S}^+(n) / \mathrm{Diag}^+(n), g^Q, \star \right)$ be the Riemannian submersion of log-Euclidean Lie groups described above. Its canonical section is
    \begin{align*}
        s : \mathcal{S}^+(n) / \mathrm{Diag}^+(n) \longrightarrow \mathcal{S}^+(n), \quad [\Sigma] \longmapsto \exp \circ \;\mathrm{off} \circ \log (\Sigma),
    \end{align*}
    where $\Sigma$ is any choice of representative of $[\Sigma]$, i.e., $\pi(\Sigma) = [\Sigma]$.
\end{theorem}

\begin{proof}
    First, we need to show that $s$ is a well-defined section. It is rather clear due to additivity of the logarithm and since the $\mathrm{off}$-map is linear and removes the diagonal. Let $\Sigma$ and $\Sigma \star D$ be two representatives of $[\Sigma]$ for some $D \in \mathrm{Diag}^+(n)$. Then, 
    \begin{align*}
        s(\Sigma \star D) = \exp \circ \;\mathrm{off} \circ \log (\Sigma \star D) = \exp \circ \;\mathrm{off} (\log(\Sigma) + \log(D)) = \exp \circ \;\mathrm{off} \circ \log (\Sigma) = s(\Sigma).
    \end{align*}
    It remains to show that this section is the canonical one. In Theorem \ref{canonical_section_thm} we showed that the canonical section of a Riemannian submersion of log-Euclidean Lie groups is given by 
    \begin{align*}
        s : G/H \longrightarrow G, \quad [g] \longmapsto \exp_G\left( s_{\mathrm{lin}}(\log_G(g) + \mathfrak h)\right).
    \end{align*}
    In our case, the Lie exponential map $\exp_G$ is the matrix exponential map and its inverse is the matrix logarithm. The application $[g] \longmapsto \log_G(g) + \mathfrak h$ is written $[\Sigma] \longmapsto \log(\Sigma) + \mathrm{Diag}(n)$, with $\Sigma$ any element of the fiber of $[\Sigma]$. It is well-defined again because of the additivity of the logarithm. Since, $\mathrm{Diag}(n)$ and $\mathrm{Hol}(n)$ are orthogonal in $\mathcal{S}(n)$ with respect to the Frobenius inner product, the application $s_{\mathrm{lin}}$ sends any coset $\log(\Sigma) + \mathrm{Diag}(n) \in \mathcal{S}(n) / \mathrm{Diag}(n)$ to a matrix in $\mathrm{Hol}(n)$, that is, a symmetric matrix with null-diagonal. Therefore, one can write $s_{\mathrm{lin}}$ exactly as the $\mathrm{off}$-map, and composing these three maps yields exactly the canonical section of $\pi$.
\end{proof}

Let us now show in the next proposition that the image set of the canonical section $K := \mathrm{im}(s)$ is also a global section for the congruence bundle. This will conclude the proof of uniqueness in Theorem~\ref{thm:canonical_congruence_normalization}. 

\begin{proposition}
\label{prop:K_global_section}
Let
\[
\pi_{\mathrm{cong}}:\mathcal{S}^+(n)\longrightarrow \mathcal{S}^+(n)/\mathrm{Diag}^+(n),
\qquad
\Sigma\longmapsto [\Sigma]_{\mathrm{cong}},
\]
denote the quotient map for the usual congruence action
\[
(\Delta,\Sigma)\longmapsto \Delta\Sigma\Delta.
\]
Denote by $K$ the image of $s : \mathcal{S}^+(n) / \mathrm{Diag}^+(n) \longrightarrow \mathcal{S}^+(n)$, $[\Sigma] \longmapsto \exp \circ \;\mathrm{off} \circ \log (\Sigma)$. Then $K$ is a global section of $\pi_{\mathrm{cong}}$.
\end{proposition}

\begin{proof}
Let
\[
\pi_{\mathrm{cong}}:\mathcal{S}^+(n)\longrightarrow \mathcal{S}^+(n)/\mathrm{Diag}^+(n),
\qquad
\Sigma\longmapsto [\Sigma]_{\mathrm{cong}},
\]
denote the quotient map for the usual congruence action of $\mathrm{Diag}^+(n)$ on $\mathcal{S}^+(n)$,
\[
(\Delta,\Sigma)\longmapsto \Delta\Sigma\Delta.
\]
By \cite{david2019}, this action is smooth, proper and free, hence $\pi_{\mathrm{cong}}$ is a principal $\mathrm{Diag}^+(n)$-bundle. Fix $\Sigma\in \mathcal{S}^+(n)$ and define the functional
\[
F_\Sigma:\mathrm{Diag}(n)\longrightarrow \mathrm{Diag}(n),
\qquad
F_\Sigma(D):=\mathrm{Diag}\!\bigl(\log(e^D\Sigma e^D)\bigr).
\]
Then we have
\[
e^D\Sigma e^D\in K
\quad\Longleftrightarrow\quad
F_\Sigma(D)=0.
\]

We first show that $F_\Sigma$ is injective. Let $D\in \mathrm{Diag}(n)$ and
$H\in \mathrm{Diag}(n)$. Set
\[
A:=e^D\Sigma e^D\in \mathcal{S}^+(n).
\]
Since $D$ and $H$ are diagonal, they commute and likewise for $e^D$ and $H$, hence, by product rule
\[
\mathrm{d}_D\!\bigl(D \mapsto e^D\Sigma e^D\bigr)[H]=HA+AH.
\]
By the chain rule and since the diagonal operator $\mathrm{Diag}$ is linear,
\[
\mathrm{d}_D F_\Sigma[H]
=
\mathrm{Diag}\!\bigl(\mathrm{d}_A\log[HA+AH]\bigr).
\]
Hence
\[
\langle H,\mathrm{d}_D F_\Sigma[H]\rangle_F
=
\mathrm{tr}\!\Bigl(H\,\mathrm{d}_A\log[HA+AH]\Bigr),
\]
because $H$ is diagonal. Now diagonalize $A=Q\Lambda Q^\top$ with $Q$ orthogonal and
\[
\Lambda=\mathrm{diag}(\lambda_1,\dots,\lambda_n),\qquad \lambda_i>0 \; \text{ for all } i \in \{1, \ldots, n \},
\]
and set $\widetilde H:=Q^\top H Q$. The Fr\'echet derivative at $A$ in direction $U := HA + AH$ of the matrix logarithm yields
\[
Q^\top\bigl(\mathrm{d}_A\log[U]\bigr)Q
=
M(\Lambda)\odot (Q^\top UQ),
\]
where $\odot$ denotes the Hadamard product and
\[
M(\Lambda)_{ij}
=
\begin{cases}
\dfrac{\log\lambda_i-\log\lambda_j}{\lambda_i-\lambda_j}, & i\neq j,\\[1ex]
\dfrac1{\lambda_i}, & i=j.
\end{cases}
\]
All coefficients $M(\Lambda)_{ij}$ are strictly positive. Since
\[
Q^\top(HA+AH)Q=\widetilde H\,\Lambda+\Lambda\,\widetilde H, \quad \text{hence,} \quad Q^\top\bigl(\mathrm{d}_A\log[U]\bigr)Q = M(\Lambda)\odot (\widetilde H\,\Lambda+\Lambda\,\widetilde H)
\]
and since $\Lambda$ is diagonal, $( \widetilde H \Delta + \Delta \widetilde H)_{ij} = (\lambda_i + \lambda_j) \widetilde H_{ij}$ and we obtain
\[
\langle H,\mathrm{d}_D F_\Sigma[H]\rangle_F
= \mathrm{tr}\!\Bigl(\widetilde H \left[ M(\Lambda) \odot \left( \widetilde H \Lambda + \Lambda \widetilde H \right) \right] \Bigr) =
\sum_{i,j=1}^n
M(\Lambda)_{ij}\,(\lambda_i+\lambda_j)\,\widetilde H_{ij}^2.
\]
Because $\lambda_i+\lambda_j>0$ and $M(\Lambda)_{ij}>0$, this shows that
\[
\langle H,\mathrm{d}_D F_\Sigma[H]\rangle_F>0
\qquad\text{for all }H\neq 0.
\]
Hence the quadratic form associated with $\mathrm{d}_D F_\Sigma$ is positive definite on $\mathrm{Diag}(n)$. In particular, $\mathrm{d}_D F_\Sigma$ is injective, and therefore invertible for every $D\in \mathrm{Diag}(n)$. Now let $D_1,D_2\in \mathrm{Diag}(n)$, set $H:=D_1-D_2$ and define $\gamma(t):=D_2+tH$. Then $\gamma(0)=D_2$ and $\gamma(1)=D_1$. By the fundamental theorem of calculus,
\[
F_\Sigma(D_1)-F_\Sigma(D_2)
=
\int_0^1 \frac{d}{dt}F_\Sigma(\gamma(t))\,dt.
\]
By the chain rule,
\[
\frac{d}{dt}F_\Sigma(\gamma(t))
=
\mathrm d_{\gamma(t)}F_\Sigma[\gamma'(t)]
=
\mathrm d_{D_2+tH}F_\Sigma[H].
\]
Taking the Frobenius inner product with $H$ yields
\[
\langle H,F_\Sigma(D_1)-F_\Sigma(D_2)\rangle_F
=
\int_0^1
\langle H,\mathrm d_{D_2+tH}F_\Sigma[H]\rangle_F\,dt.
\]
If $D_1\neq D_2$, then $H\neq 0$, so the integrand is strictly positive for all $t$, and therefore
\[
\langle H,F_\Sigma(D_1)-F_\Sigma(D_2)\rangle_F>0.
\]
Hence $F_\Sigma(D_1)\neq F_\Sigma(D_2)$, proving that $F_\Sigma$ is injective.

Existence in Theorem~\ref{thm:canonical_congruence_normalization} states exactly that
$F_\Sigma$ admits a zero for every $\Sigma\in\mathcal{S}^+(n)$. Since $F_\Sigma$ is injective,
this zero is unique. We denote it by $D_K(\Sigma)\in \mathrm{Diag}(n)$, so that
\[
\log\!\bigl(e^{D_K(\Sigma)}\Sigma e^{D_K(\Sigma)}\bigr)\in \mathrm{Hol}(n).
\]
Equivalently,
\[
e^{D_K(\Sigma)}\Sigma e^{D_K(\Sigma)}\in K.
\]

Consider the smooth map
\[
\mathcal F:\mathcal S^+(n)\times \mathrm{Diag}(n)\to \mathrm{Diag}(n),
\qquad
\mathcal F(\Sigma,D):=\mathrm{Diag}\!\bigl(\log(e^D\Sigma e^D)\bigr).
\]
For each fixed $\Sigma$, one has $\mathcal F(\Sigma,\cdot)=F_\Sigma$. Let
$\Sigma_0\in \mathcal S^+(n)$ and let $D_0=D_K(\Sigma_0)$ be the unique solution of
\[
\mathcal F(\Sigma_0,D_0)=0.
\]
The partial differential of $\mathcal F$ with respect to the second variable at
$(\Sigma_0,D_0)$ is precisely
\[
\partial_D\mathcal F(\Sigma_0,D_0)=\mathrm d_{D_0}F_{\Sigma_0},
\]
which is invertible by the previous step. Hence, by the implicit function theorem,
there exist neighborhoods $U$ of $\Sigma_0$ and $V$ of $D_0$, and a unique smooth map
$\varphi:U\to V$ such that
\[
\mathcal F(\Sigma,\varphi(\Sigma))=0
\qquad\text{for all }\Sigma\in U.
\]
Thus the zero of $F_\Sigma$ depends smoothly on $\Sigma$ locally. Since this zero is unique for every $\Sigma$, the local solutions given by the implicit function theorem coincide on overlaps, and therefore define a global smooth map
\[
D_K:\mathcal S^+(n)\to \mathrm{Diag}(n).
\]

We now define
\[
s_K:\mathcal{S}^+(n)/\mathrm{Diag}^+(n)\longrightarrow K,
\qquad
s_K([\Sigma]_{\mathrm{cong}})
:=
e^{D_K(\Sigma)}\Sigma e^{D_K(\Sigma)}.
\]
This is well-defined: let $\Sigma'=\Lambda\Sigma\Lambda$ with
$\Lambda\in\mathrm{Diag}^+(n)$. We claim that
\[
D_K(\Sigma')=D_K(\Sigma)-\log\Lambda.
\]
Indeed, since diagonal matrices commute,
\[
e^{D_K(\Sigma)-\log\Lambda}\,\Sigma'\,e^{D_K(\Sigma)-\log\Lambda}
=
e^{D_K(\Sigma)}\Sigma e^{D_K(\Sigma)}\in K.
\]
Thus $D_K(\Sigma)-\log\Lambda$ sends $\Sigma'$ into $K$, this implies
\[
D_K(\Sigma')=D_K(\Sigma)-\log\Lambda.
\]
Consequently,
\[
e^{D_K(\Sigma')}\Sigma' e^{D_K(\Sigma')}
=
e^{D_K(\Sigma)}\Sigma e^{D_K(\Sigma)},
\]
so the definition of $s_K([\Sigma]_{\mathrm{cong}})$ is independent of the chosen representative.

Its image is exactly $K$: if $X\in K$, then $\log X\in \mathrm{Hol}(n)$, so $F_X(0)=0$.
By uniqueness, $D_K(X)=0$, and hence
\[
s_K([X]_{\mathrm{cong}})= e^0Xe^0 = X.
\]
Therefore, $K$ intersects every congruence orbit in exactly one point, i.e. $K$ is a global
section for the congruence bundle.
\end{proof}

A consequence of Theorems~\ref{thm:can_split} and \ref{thm:can_section_appli} is that the usual log-Euclidean metric on $\mathcal{S}^+(n)$ splits orthogonally onto the image of the canonical section, $K := \mathrm{im}(s) = \{ \exp \circ \;\mathrm{off} \circ \log (\Sigma) \mid \Sigma \in \mathcal{S}^+(n)\}$ and the vertical subspace $\mathrm{Diag}^+(n)$. By observing $K = \exp(\mathrm{Hol}(n))$ and $\mathrm{Diag}^+(n) = \exp(\mathrm{Diag}(n))$, we recover exactly the statement of Theorem \ref{metric_split_LE_thm}. In other words, we can define a global trivialization and Riemannian isometry
\begin{align*}
    \Phi : K \times \mathrm{Diag}^+(n) \longrightarrow \mathcal{S}^+(n), \quad (X,D) \longmapsto X \star D.
\end{align*}
Then at each $(X, D)$ the differential 
\begin{align*}
    \mathrm d_{(X, D)} \Phi : \mathcal{T}_XK \times \mathcal{T}_D \mathrm{Diag}^+(n) \longrightarrow \mathcal{T}_{X \star D}\mathcal{S}^+(n)
\end{align*}
identifies the two factors as $g^{\mathrm{LE}}$-orthogonal subspaces. In particular, for $\delta_X, \delta_X' \in \mathcal{T}_XK$ and $\xi_D, \xi_D' \in \mathcal{T}_D\mathrm{Diag}^+(n)$, one has
\begin{align*}
    g^\mathrm{LE}\Bigl(\mathrm d_{(X, D)} \Phi(\delta_X, \xi_D),  \mathrm d_{(X, D)} \Phi(\delta_X', \xi_D')\Bigr) = g^\mathrm{LE}_{\vert K}\Bigl(\delta_X, \delta_X'\Bigr) + g^\mathrm{LE}_{\vert \mathrm{Diag}^+(n)}\Bigl(\xi_D, \xi_D'\Bigr),
\end{align*}
and the induced metric on $K \subset \mathcal{S}^+(n)$ is precisely $\pi^*g^Q$, i.e., 
\begin{align*}
    g^\mathrm{LE}_{\vert K}(\delta_X, \delta_X') = g^Q(\mathrm d_X \pi (\delta_X), \mathrm d_X \pi (\delta_X')).
\end{align*}
The canonical section inherits the log-Euclidean Lie group structure of $\left(\mathcal{S}^+(n), \star \right)$, and because it is furthermore an isometric embedding $\mathcal{S}^+(n) / \mathrm{Diag}^+(n) \hookrightarrow \mathcal{S}^+(n)$, the quotient metric and the induced metric on $K$ coincide.
\begin{theorem}
    Let $[X] \in \mathcal{S}^+(n) / \mathrm{Diag}^+(n)$ and let $v, w \in \mathcal{T}_{[X]}\mathcal{S}^+(n) / \mathrm{Diag}^+(n)$ be two quotient tangent vectors at $[X]$. Then 
    \begin{align*}
        g^Q_{[X]} \Bigl( v, w \Bigr) = g^{\overline{\mathrm{OL}}}_{[X]}(v, w)
    \end{align*}
    with $g^{\overline{\mathrm{OL}}}$ the Riemannian metric on $\mathcal{S}^+(n) / \mathrm{Diag}^+(n)$ defined by pullback by the map
    \begin{align*}
        \mathrm{off} \circ \log : \mathcal{S}^+(n) / \mathrm{Diag}^+(n) \rightarrow \mathrm{Hol}(n).
    \end{align*}
\end{theorem}

\begin{proof}
    We have shown in the proof of Theorem \ref{thm:can_section_appli} that the $\mathrm{off} \circ \log$ map is well-defined (i.e., is $\mathrm{Diag}^+(n)$-equivariant). The rest of the proof uses the isometric parametrization given by the canonical section $s$, and is only computational. Let $[X] \in \mathcal{S}^+(n) / \mathrm{Diag}^+(n)$ and let $v, w \in \mathcal{T}_{[X]}\mathcal{S}^+(n) / \mathrm{Diag}^+(n)$, pick $X$ a representative of $[X]$, then, 
        \begin{align*}
          g^{Q}(v,w)
          &= g^{\mathrm{LE}}_{\vert K}\Bigl(\mathrm d_{[X]}s(v),\mathrm d_{[X]}s(w)\Bigr)
          &&\quad\text{\parbox[t]{3.5cm}{($s$ isometric\\embedding)}}\\
          &=
          g^{\mathrm{LE}}_{\vert K}\!\Bigl(
              \mathrm d_{\mathrm{off}\circ\log X}\exp\!\bigl(\mathrm d_{[X]}\mathrm{off}\circ\log(v)\bigr),
              \mathrm d_{\mathrm{off}\circ\log X}\exp\!\bigl(\mathrm d_{[X]}\mathrm{off}\circ\log(w)\bigr)
            \Bigr)
          &&\quad\text{(def.\ of $s$)}\\[2pt]
          &=
          \Bigl\langle
            \mathrm d_{s([X])}\log\!\bigl(\mathrm d_{\mathrm{off}\circ\log X}\exp(\overline{v})\bigr),
            \mathrm d_{s([X])}\log\!\bigl(\mathrm d_{\mathrm{off}\circ\log X}\exp(\overline{w})\bigr)
          \Bigr\rangle
          &&\quad\text{(def.\ of $g^{\mathrm{LE}}$)}\\[2pt]
          &=
          \Bigl\langle
            \mathrm d_{[X]}\mathrm{off}\circ\log(v),
            \mathrm d_{[X]}\mathrm{off}\circ\log(w)
          \Bigr\rangle
          &&\quad\text{(chain rule)}\\[2pt]
          &=
          g^{\overline{\mathrm{OL}}}(v,w)
          &&\quad\text{(def.\ of $g^{\overline{\mathrm{OL}}}$)}
        \end{align*}
    where $\overline{v} = \mathrm d_{[X]} \mathrm{off} \circ \log (v)$ and similarly for $\overline{w}$.
\end{proof}

\begin{remark}
    The $\mathrm{off}$-$\log$ map here has domain the quotient manifold $\mathcal{S}^+(n) / \mathrm{Diag}^+(n)$ but has the same image as the map $\mathrm{Log} := \mathrm{off} \circ \log : \mathcal{S}^+(n) \rightarrow \mathrm{Hol}(n)$ (cf: Section \ref{off-log_sec}). Both maps have the exact same image in $\mathrm{Hol}(n)$ and we write $g^{\overline{\mathrm{OL}}} = g^Q$ interchangeably.
\end{remark}
We proved that the $\mathrm{off}$-$\log$ metric $g^{\overline{\mathrm{OL}}}$ is in fact the quotient metric $g^Q$ of the log-Euclidean metric $g^\mathrm{LE}$ on $\mathcal{S}^+(n)$. Namely,
\begin{align*}
    g^Q = g^{\overline{\mathrm{OL}}} = s^*g_{\vert K}^\mathrm{LE}.
\end{align*}
Next, we provide a parametrization of the full-rank correlation matrices in this principal $\mathrm{Diag}^+(n)$-bundle, which we recall is not the usual one, in the sense that the action of positive diagonal matrices on $\mathcal{S}^+(n)$ is given by a log-Euclidean group action here, instead of the usual congruence action. We show that there exists a section of $\pi$ whose image is $\mathrm{Cor}^+(n)$.

\begin{theorem}
\label{thm:corr_section_appli}
    Let $\pi : \left(\mathcal{S}^+(n), g^{\mathrm{LE}}, \star \right) \rightarrow \left(\mathcal{S}^+(n) / \mathrm{Diag}^+(n), g^Q, \star \right)$ be the Riemannian submersion of log-Euclidean Lie groups described above. Define its correlation section 
    \begin{align*}
        \sigma_\mathrm{corr} : \mathcal{S}^+(n) / \mathrm{Diag}^+(n) \longrightarrow \mathrm{Cor}^+(n) \subset \mathcal{S}^+(n), \quad [\Sigma] \longmapsto \Sigma \star D_\Sigma,
    \end{align*}
    where $D_\Sigma := \exp(\mathcal{D}(\log(\Sigma))) \in \mathrm{Diag}^+(n)$ with $\mathcal{D} : S \in \mathrm{Hol}(n) \mapsto \mathcal{D}(S) \in \mathrm{Diag}(n)$ yields the unique diagonal matrix $\mathcal{D}(S)$ such that $\exp(\mathcal{D}(S) + S) \in \mathrm{Cor}^+(n)$, (cf: Section \ref{off-log_sec}).
\end{theorem}

\begin{proof}
    The only thing we really need to show is that this section is well-defined. Recall from Lemma \ref{equiv_diag} the \textit{surjective} map $S \in \mathcal{S}(n) \mapsto \exp(S + \mathcal{D}(S)) \in \mathrm{Cor}^+(n)$ is equivariant under the additive group action $\mathcal{S}(n) \times \mathrm{Diag}(n) \rightarrow \mathcal{S}(n)$. Therefore, the map
    \begin{align*}
        \sigma_\mathrm{corr} = \exp \Bigl( \log (\cdot) + \mathcal{D}(\log (\cdot)) \Bigr) = \exp \Bigl( \mathrm{off} \circ \log (\cdot) + \mathcal{D}(\mathrm{off} \circ \log (\cdot)) \Bigr)
    \end{align*}
    is $\mathrm{Diag}(n)$-equivariant for the usual addition in log charts, which is exactly the $\mathrm{Diag}^+(n)$-equivariance for the log-Euclidean group action. Hence, the map $\sigma_\mathrm{corr}$ is a well-defined global section of $\pi$ whose image spans the space of full-rank correlation matrices.
\end{proof}

Like the canonical section, the full-rank correlation matrices section, $\sigma_\mathrm{corr}$, provides a parametrization of the quotient manifold $\mathcal{S}^+(n) / \mathrm{Diag}^+(n)$ which also allows to compute the quotient metric. In fact, Proposition~\ref{horizontal_lift_prop} allows to compute the quotient metric $g^Q$ by horizontally lifting along any global section of $\pi$. We now show how we recover the $\mathrm{off}$-$\log$ metric on $\mathrm{Cor}^+(n)$ using the global section $\sigma_\mathrm{corr}$, i.e., how we recover the equation 
\begin{align*}
    g^Q = g^{\overline{\mathrm{OL}}} = \sigma_\mathrm{corr}^*(g^\mathrm{OL}).
\end{align*}

\begin{corollary}
\label{quotient_metric_log_thm}
The \emph{off-log} metric $g^{\mathrm{OL}}$ defined on $\mathrm{Cor}^+(n)$ in Section~\ref{off-log_sec} is isometric to the quotient metric $g^Q$ of the Riemannian submersion
\begin{align*}
    \pi : \bigl( \mathcal{S}^+(n), \star, g^{\mathrm{LE}} \bigr)
    \longrightarrow \bigl(\mathcal{S}^+(n) / \mathrm{Diag}^+(n), g^Q\bigr),
    \qquad
    \pi(\Sigma) = \Sigma \star \mathrm{Diag}^+(n),
\end{align*}
via the full-rank correlation matrices section
\begin{align*}
    \sigma_{\mathrm{corr}} : 
    \mathcal{S}^+(n) / \mathrm{Diag}^+(n) &\longrightarrow \mathrm{Cor}^+(n)
    \subset \mathcal{S}^+(n), \\
    [\Sigma] &\longmapsto \Sigma \star D_\Sigma
    = \exp\bigl( \mathcal{D}( \log (\Sigma)) + \log (\Sigma) \bigr).
\end{align*}
\end{corollary}

\begin{proof}
Let us recall that the log-Euclidean metric $g^{\mathrm{LE}}$ on $\mathcal{S}^+(n)$ is given, at the identity $\mathrm{I}_n$, by the Frobenius inner product
\begin{align*}
    g_{\mathrm{I}_n}^{\mathrm{LE}}(U, V) = \mathrm{tr}(UV), \quad U,V \in \mathcal{T}_{\mathrm{I}_n}\mathcal{S}^+(n)=\mathcal{S}(n),
\end{align*}
since $\mathrm{d}_{\mathrm{I}_n}\log=\mathrm{I}_n$. Consider the subgroup $\mathrm{Diag}^+(n) \subset \mathcal{S}^+(n)$ and the associated Riemannian submersion
\begin{align*}
    \pi:\mathcal{S}^+(n)\longrightarrow \mathcal{S}^+(n)/\mathrm{Diag}^+(n).
\end{align*}
The tangent space decomposition at the identity is
\begin{align*}
    \mathcal{T}_{\mathrm{I}_n}\mathcal{S}^+(n)=\mathrm{Diag}(n)\oplus \mathrm{Hol}(n),
\end{align*}
where $\mathrm{Diag}(n)$ and $\mathrm{Hol}(n)$ are the vertical and horizontal spaces at the identity, respectively. Moreover, this decomposition is orthogonal with respect to the Frobenius inner product, hence orthogonal with respect to the log-Euclidean metric at the identity. Let $\Sigma \in \mathcal{S}^+(n)$, the quotient metric $g^Q$ at any $[\Sigma]\in\mathcal{S}^+(n)/\mathrm{Diag}^+(n)$ is defined by taking horizontal lifts $v^H,w^H\in\mathcal{H}_\Sigma$ of tangent vectors $v,w\in \mathcal{T}_{[\Sigma]}(\mathcal{S}^+(n)/\mathrm{Diag}^+(n))$, namely,
\begin{align*}
    g_{[\Sigma]}^Q(v,w)=g_{\Sigma}^{\mathrm{LE}}(v^H,w^H).
\end{align*}
Now, the correlation section of $\pi$ defined as
\begin{align*}
    \sigma_{\mathrm{corr}} : \mathcal{S}^+(n) / \mathrm{Diag}^+(n) \longrightarrow \mathcal{S}^+(n), \quad [\Sigma] \longmapsto \Sigma \star D_\Sigma = \exp \left( \mathcal{D}( \log (\Sigma)) +\log (\Sigma) \right)
\end{align*}
embeds $\mathcal{S}^+(n) / \mathrm{Diag}^+(n)$ into $\mathcal{S}^+(n)$. Let $[C]\in \mathcal{S}^+(n) / \mathrm{Diag}^+(n)$, and define tangent vectors $\delta_{[C]}, \xi_{[C]} \in\mathcal{T}_{[C]} (\mathcal{S}^+(n)/\mathrm{Diag}^+(n))$. Write $\sigma_\mathrm{corr}([C]) = C \in \mathrm{Cor}^+(n)$, $\delta_C = \mathrm d_{[C]}\sigma_\mathrm{corr}(\delta_{[C]})$ and $\xi_C = \mathrm d_{[C]}\sigma_\mathrm{corr}(\xi_{[C]})$. According to Proposition~\ref{horizontal_lift_prop}, the horizontal lifts of the quotient tangent vectors along $\sigma_\mathrm{corr}$ are given by
\begin{align*}
    \mathrm{Hor}_{C}(\delta_{[C]})&=(L_{C})^*(s_\mathrm{lin}(\mathrm{d}_{C}\log_G(\delta_C) + \mathrm{Diag}(n))) \\
    &= (L_{C})^*(\mathrm{off}(\mathrm{d}_{C}\log(\delta_C) + \mathrm{Diag}(n))) \\
    &= (L_{C})^*(\mathrm{off}(\mathrm{d}_{C}\log(\delta_C))),
\end{align*}
and likewise $\mathrm{Hor}_{C}(\delta_{[C]}) = (L_{C})^*(\mathrm{off}(\mathrm{d}_{C}\log(\xi_C)))$ with $\mathrm d_C \pi(\delta_C) = \delta_{[C]}$ and $\mathrm d_C \pi(\xi_C) = \xi_{[C]}$. Therefore, using left-invariance of the log-Euclidean metric we have 
\begin{align*}
    g_{[C]}^Q(\delta_{[C]},\xi_{[C]}) &= g_C^{\mathrm{LE}}\left(\mathrm{Hor}_{C}(\delta_{[C]}),\mathrm{Hor}_{C}(\xi_{[C]})\right)\\
    &=g_C^{\mathrm{LE}}\left((L_C)^*(\mathrm{off}(\mathrm{d}_C\log(\delta_C))),(L_C)^*(\mathrm{off}(\mathrm{d}_C\log(\xi_C)))\right)\\
    &=g_{\mathrm{I}_n}^{\mathrm{LE}}\left(\mathrm{off}(\mathrm{d}_C\log(\delta_C)),\mathrm{off}(\mathrm{d}_C\log(\xi_C))\right)\\
    &=\mathrm{tr}\left(\mathrm{off}(\mathrm{d}_C\log(\delta_C)) \mathrm{off}(\mathrm{d}_C\log(\xi_C))\right) = (\sigma_\mathrm{corr}^*g^{\mathrm{OL}})_{[C]}\left(\delta_{[C]}, \xi_{[C]} \right).
\end{align*}
\end{proof}

\paragraph{On the induced vs.\ the quotient metric for the correlation section.} However, unlike the canonical section $s$, the correlation section $\sigma_\mathrm{corr}$ is not horizontal. We make explicit the vertical term in the expression of $\sigma_\mathrm{corr}$.

\begin{proposition}
    The correlation section differs from the canonical section by a non-constant vertical term.
\end{proposition}

\begin{proof}
    Recall the correlation section is defined at each $[\Sigma] \in \mathcal{S}^+(n) / \operatorname{Diag}^+(n)$ by the log-Euclidean action of a positive diagonal term $D_\Sigma \in \mathrm{Diag}^+(n)$ on any representative $\Sigma$ of $[\Sigma]$. Since the Lie exponential (the usual matrix exponential here) is a group morphism, we have 
    \begin{align*}
        \sigma_\mathrm{corr}([\Sigma]) &:= \Sigma \star D_\Sigma \\
        &= \exp(\mathrm{off} (\log (\Sigma))) \star \exp (\mathcal{D}(\mathrm{off}(\log (\Sigma)))) = s \star \exp(D_\Sigma),
    \end{align*}
    where $D_\Sigma$ is a vertical element \emph{smoothly varying} with $[\Sigma]$. 
\end{proof}

Although the quotient metric and the induced log-Euclidean metric on $K := \mathrm{im}(s)$ coincide:
\begin{align*}
    g^Q = s^*(g^\mathrm{LE}_{\vert K}) \text{ on } \mathcal{S}^+(n) / \operatorname{Diag}^+(n),
\end{align*}
the quotient metric along $\sigma_\mathrm{corr}$ and the induced metric on $\mathrm{Cor}^+(n) = \mathrm{im}(\sigma_\mathrm{corr})$ do not coincide: 
\begin{align*}
    \sigma_\mathrm{corr}^*\bigl(g^\mathrm{LE}\big|_{\mathrm{Cor}^+(n)}\bigr)
    = \underbrace{\sigma_\mathrm{corr}^*(g^\mathrm{OL})}_{g^Q}
       \;+\;
       \underbrace{\sigma_\mathrm{corr}^*(g^\mathrm{DL})}_{\text{vertical term}}.
\end{align*}

In section \ref{isom_log_subsec} we have explicitly shown that the off-log metric and the log-scaling Riemannian metrics on full-rank correlation matrices are isometric, call this isometry $\Phi^\mathrm{LS-OL}$. Since by Theorem \ref{thm:corr_section_appli} (and Corollary \ref{quotient_metric_log_thm}) the quotient metric of the Riemannian submersion 
\begin{align*}
    \pi : \left( \mathcal{S}^+(n), \star, g^{\mathrm{LE}} \right) \longrightarrow \left(\mathcal{S}^+(n) / \mathrm{Diag}^+(n), \star, g^Q\right), \quad \pi(\Sigma) = \Sigma \star \mathrm{Diag}^+(n),
\end{align*}
coincides with the off-log metric on full-rank correlation matrices, namely, one deduces immediately the following corollary.

\begin{corollary}
The \emph{log-scaling} metric on full-rank correlation matrices can be expressed as a Riemannian isometry of the quotient metric of $\pi$,
\begin{align*}
    g^{\mathrm{LS}} = (\pi \circ \Phi^\mathrm{LS-OL})^* (g^Q).
\end{align*}
\end{corollary}

In fact, any log--Euclidean Riemannian metric on $\mathrm{Cor}^+(n)$ is
isometric to the quotient metric on $\mathcal{S}^+(n)/\mathrm{Diag}^+(n)$. Via the identification $\mathcal{S}^+(n)/\mathrm{Diag}^+(n)\simeq \mathrm{Cor}^+(n)$, it coincides with the off-log metric.

\medskip

In Section \ref{isom_iso_sec}, we constructed a Riemannian isometry between log-Euclidean metrics on $n\times n$ full-rank correlation matrices, using isometries with the usual log-Euclidean metric for SPD matrices of size $(n-1) \times (n-1) $. However, as observed in \cite{bisson2025}, experiments show that rescaling SPD matrices to satisfy the unit-diagonal constraint of full-rank correlation matrices can alter the correlation coefficients. Accordingly, the standard correlation section
\begin{align*}
    \sigma_{\mathrm{corr}} : \mathcal{S}^+(n) / \mathrm{Diag}^+(n) \longrightarrow \mathcal{S}^+(n), \quad [\Sigma] \longmapsto \Sigma \star D_\Sigma = \exp \left( \mathcal{D}( \log (\Sigma)) +\log (\Sigma) \right),
\end{align*}
although a smooth embedding, always picks up a non‑zero “\emph{vertical}” component when one pulls back the log‑Euclidean metric. In this section, we have shown that for the standard log-Euclidean metric $g^{\mathrm{LE}}$ on $\mathcal{S}^+(n)$ the correlation section $\sigma_{\mathrm{corr}}$ is not horizontal, so the induced metric on $\mathrm{Cor}^+(n)$ differs from the intrinsic off-log metric by an explicit vertical term (the diagonal part $g^{\mathrm{DL}}$). This obstruction is specific to the \emph{ambient} choice of log-Euclidean metric. Indeed, we now show that one can construct an explicit log-Euclidean metric on $\mathcal{S}^+(n)$ adapted to the off-log metric, that is, for which the inclusion of full-rank correlation matrices into SPD matrices is an isometric embedding.

\begin{theorem}[Ambient log-Euclidean extension of the off-log metric]\label{thm:ambient_extension}
Define
\[
\Phi:\mathcal{S}^+(n)\longrightarrow \mathrm{Hol}(n)\oplus \mathrm{Diag}(n),\qquad
\Phi(\Sigma):=\Big(\mathrm{off}(\log\Sigma),\ \mathrm{Diag}(\log\Sigma)-\mathcal{D}(\mathrm{off}(\log\Sigma))\Big),
\]
where $\mathcal{D}$ is the diagonal correction map of Section~\ref{off-log_sec}. Then $\Phi$ is a smooth diffeomorphism with inverse
\[
\Phi^{-1}(S,\Delta)=\exp\big(S+\mathcal{D}(S)+\Delta\big).
\]
Let $g^{\mathrm{ext}}:=\Phi^{*}\langle\cdot,\cdot\rangle$ be the corresponding log-Euclidean metric on $\mathcal{S}^+(n)$, where $\langle\cdot,\cdot\rangle$ is the Frobenius inner product on $\mathrm{Hol}(n)\oplus \mathrm{Diag}(n)$. Then the inclusion
\[
i:\big(\mathrm{Cor}^+(n),g^{\mathrm{OL}}\big)\hookrightarrow \big(\mathcal{S}^+(n),g^{\mathrm{ext}}\big)
\]
is a Riemannian isometric embedding. Moreover, $i\big(\mathrm{Cor}^+(n)\big)$ is totally geodesic and corresponds, in $\Phi$-coordinates, to the linear subspace $\mathrm{Hol}(n)\oplus\{0\}$.
\end{theorem}

\begin{proof}
Write $\Phi(\Sigma)=(S,\Delta)$ with $S=\mathrm{off}(\log\Sigma)\in \mathrm{Hol}(n)$ and $\Delta=\mathrm{Diag}(\log\Sigma)-\mathcal{D}(S)\in \mathrm{Diag}(n)$. Since $\log\Sigma=S+\mathrm{Diag}(\log\Sigma)$ and $\mathcal{D}(S)$ is diagonal, we have
\[
\log\Sigma=S+\mathcal{D}(S)+\Delta,
\]
which proves $\Sigma=\exp\big(S+\mathcal{D}(S)+\Delta\big)$ and hence the claimed inverse formula. Therefore $\Phi$ is bijective. Moreover, $\Phi$ is smooth since $\log:\mathcal S^+(n)\to\mathcal S(n)$ and $\mathcal D$ are smooth, and $\mathrm{off}$, $\mathrm{Diag}$ are linear maps. Its inverse
\[
(S,\Delta)\longmapsto \exp\bigl(S+\mathcal D(S)+\Delta\bigr)
\]
is also smooth, since $(S,\Delta)\mapsto S+\mathcal D(S)+\Delta$ is linear and $\exp:\mathcal S(n)\to\mathcal S^+(n)$ is smooth. Hence $\Phi$ is a diffeomorphism. If $C\in \mathrm{Cor}^+(n)$, then by Lemma~\ref{diag_log} one has $\mathcal{D}(\mathrm{Log}(C))=\mathrm{Diag}(\log C)$, and since $\mathrm{Log}(C)=\mathrm{off}(\log C)$ this implies $\Delta=0$ and
\[
\Phi(C)=\big(\mathrm{Log}(C),0\big).
\]
Therefore $\Phi$ restricts on $\mathrm{Cor}^+(n)$ to the off-log chart (up to the trivial inclusion $S\mapsto (S,0)$), and consequently $g^{\mathrm{ext}}|_{\mathrm{Cor}^+(n)}=g^{\mathrm{OL}}$, proving that $i$ is an isometric embedding. Finally, since $\Phi$ is a global isometry from $(\mathcal{S}^+(n),g^{\mathrm{ext}})$ to the Euclidean space $\mathrm{Hol}(n)\oplus \mathrm{Diag}(n)$ endowed with the product Euclidean metric
induced by the Frobenius inner product on $\mathcal S(n)$, the vector subspace $\mathrm{Hol}(n)\oplus\{0\}$ is totally geodesic; pulling back yields the last claim.
\end{proof}

The following corollary is then immediate from Theorem~\ref{thm:ambient_extension}.

\begin{corollary}[Distance and projection]\label{cor:projection}
For $\Sigma\in \mathcal{S}^+(n)$ write $\Phi(\Sigma)=(S(\Sigma),\Delta(\Sigma))$. Then the $g^{\mathrm{ext}}$-distance from $\Sigma$ to $\mathrm{Cor}^+(n)$ is $\|\Delta(\Sigma)\|_F$, and a $g^{\mathrm{ext}}$-closest point is given by
\[
P_{\mathrm{Cor}}^{\mathrm{ext}}(\Sigma)=\Phi^{-1}\big(S(\Sigma),0\big)=\exp\big(S(\Sigma)+\mathcal{D}(S(\Sigma))\big)\in \mathrm{Cor}^+(n).
\]
\end{corollary}

\begin{remark}

It is important to distinguish between the different embeddings of the quotient manifold $\mathcal{S}^+(n)/\operatorname{Diag}^+(n)$ into $\mathcal{S}^+(n)$ that appear in this work. We have shown that the Riemannian submersion $\pi : (\mathcal{S}^+(n),g^{\mathrm{LE}}) \longrightarrow (\mathcal{S}^+(n) / \operatorname{Diag}^+(n),g^Q)$ admits a (canonical) horizontal section $s : \mathcal{S}^+(n)/\operatorname{Diag}^+(n) \longrightarrow \mathcal{S}^+(n)$ whose image $K := \mathrm{im}(s) = \exp \left( \operatorname{Hol}(n) \right)$ satisfies $K \times \operatorname{Diag}^+(n) = \mathcal{S}^+(n)$ and on which the induced and quotient metric coincide: 
\begin{align*}
    s^*(g^{\mathrm{LE}}_{\vert_K}) = g^Q.
\end{align*}
On the other hand, the correlation section $\sigma_{\mathrm{corr}} : \mathcal{S}^+(n)/\operatorname{Diag}^+(n) \longrightarrow \operatorname{Cor}^+(n) \subset \mathcal{S}^+(n)$ is \emph{not} horizontal with respect to $g^{\mathrm{LE}}$, and in general $\sigma_{\mathrm{corr}}^* g^{\mathrm{LE}} \neq g^Q$ because of the additional vertical contribution $g^{\mathrm{DL}}$. The following identity
\begin{align*}
    g^Q = \sigma_{\mathrm{corr}}^* g^{\mathrm{OL}}
\end{align*}
should therefore be read as follows: the intrinsic off-log metric $g^{\mathrm{OL}}$ on $\operatorname{Cor}^+(n)$ is \emph{isometric} to the quotient metric $g^Q$ on $\mathcal{S}^+(n)/\operatorname{Diag}^+(n)$ via $\sigma_{\mathrm{corr}}$. While the canonical section $s$ gives another isometric model $(K,g^{\mathrm{LE}}_{\vert_K})$ of the same base, these are two different isometric embeddings of the same abstract Riemannian manifold $(\mathcal{S}^+(n)/\operatorname{Diag}^+(n),g^Q)$ into $(\mathcal{S}^+(n),g^{\mathrm{LE}})$: one horizontal (the image $K$ of $s$), and one given by correlation matrices (the image of $\sigma_{\mathrm{corr}}$), which is typically not horizontal. Thus the equalities $s^*(g^{\mathrm{LE}}|_K) = g^Q = \sigma_{\mathrm{corr}}^* g^{\mathrm{OL}}$ do not imply that $\operatorname{Cor}^+(n)$ is a horizontal submanifold of $(\mathcal{S}^+(n),g^{\mathrm{LE}})$. In fact, we can summarize this last remark and section as follows:
\begin{enumerate}
    \item $\sigma_{\mathrm{corr}}$ is \emph{not} an isometric embedding $\left(\mathcal{S}^+(n) / \operatorname{Diag}^+(n), g^Q\right) \hookrightarrow \left(\mathcal{S}^+(n), g^\mathrm{LE}\right)$ because it is not horizontal with respect to $g^\mathrm{LE}$;
    \item $\sigma_{\mathrm{corr}}$ \emph{is} an isometry $\left(\mathcal{S}^+(n) / \operatorname{Diag}^+(n), g^Q\right) \longrightarrow \left(\operatorname{Cor}^+(n), g^\mathrm{OL}\right)$;
    \item by Theorem~\ref{thm:ambient_extension}, the standard inclusion $i:\left(\operatorname{Cor}^+(n), g^\mathrm{OL}\right)\hookrightarrow\left(\mathcal{S}^+(n), g^\mathrm{ext}\right)$ is an isometric (totally geodesic) embedding for a suitable ambient log-Euclidean metric $g^\mathrm{ext}$.
\end{enumerate}
\end{remark}

\subsubsection{Further consequences of the adapted extension metric}\label{sec:ext_consequences}

The adapted global chart $\Phi$ introduced above yields closed-form expressions for orthogonal projection, geodesics and distance
in $\bigl(\mathcal{S}^+(n),g^{\mathrm{ext}}\bigr)$, as well as an explicit global isometry with the standard log-Euclidean metric on $\mathcal{S}^+(n)$.

\begin{proposition}
\label{prop:projection_vs_rescaling}
For every real number $\lambda>0$ and every $C\in\mathrm{Cor}^+(n)$ one has
\[
  P_{\mathrm{Cor}}^{\mathrm{ext}}(\lambda C)=R(\lambda C)=C.
\]
In general, $P_{\mathrm{Cor}}^{\mathrm{ext}}(\Sigma)\neq R(\Sigma)$.
More precisely, for every $n\ge 2$, there exist smooth one-parameter families
$\Sigma_s\in\mathcal{S}^+(n)$ such that
$P_{\mathrm{Cor}}^{\mathrm{ext}}(\Sigma_s)$ is constant in $s$ while $R(\Sigma_s)$ varies with $s$.
\end{proposition}

\begin{proof}
Let $\Sigma=\lambda C$ with $\lambda>0$ and $C\in\mathrm{Cor}^+(n)$. Then
$\log\Sigma=(\log\lambda)I+\log C$ and hence $S(\Sigma)=\mathrm{off}(\log C)=\mathrm{Log}(C)$.
By Lemma~\ref{diag_log} one obtains
\[
P_{\mathrm{Cor}}^{\mathrm{ext}}(\Sigma)
=\exp\bigl(\mathrm{Log}(C)+\mathcal{D}(\mathrm{Log}(C))\bigr)
=\exp(\log C)=C.
\]
Moreover, since $\mathrm{Diag}(\Sigma)=\lambda I$, rescaling gives
$R(\Sigma)=C$ as well. To see that the two projections differ in general, fix $h\neq 0$ and consider, for $s\in\mathbb{R}$, the matrices $A_s$ and $\Sigma_s^{(2)}$ defined by
\[
  A_s=\begin{pmatrix}s & h\\ h & 0\end{pmatrix},
  \qquad
  \Sigma_s^{(2)}=\exp(A_s)\in\mathcal{S}^+(2).
\]
For any $n\ge2$, define the block diagonal embedding
\[
  \Sigma_s:=\Sigma_s^{(2)}\oplus I_{n-2}
  \;=\;
  \begin{pmatrix}
    \Sigma_s^{(2)} & 0\\
    0 & I_{n-2}
  \end{pmatrix}
  \in\mathcal{S}^+(n).
\]
Then
\[
  \log\Sigma_s=A_s\oplus 0_{n-2}
  \;=\;
  \begin{pmatrix}
    A_s & 0\\
    0 & 0_{n-2}
  \end{pmatrix}
  \in \mathcal{S}(n).
\]
so
\[
  S(\Sigma_s)=\mathrm{off}(\log\Sigma_s)
  =
  \begin{pmatrix}0&h\\ h&0\end{pmatrix}\oplus 0_{n-2}
  \in \mathrm{Hol}(n).
\]
which is independent of $s$. Hence $P_{\mathrm{Cor}}^{\mathrm{ext}}(\Sigma_s)$ is independent of $s$.
Set
\[
B:= S(\Sigma_s) = \begin{pmatrix}0&h\\ h&0\end{pmatrix}\oplus 0_{n-2}\in\mathrm{Hol}(n).
\]
By definition, $\mathcal D(B)\in\mathrm{Diag}(n)$ is the unique diagonal matrix such that
\[
\exp\bigl(B+\mathcal D(B)\bigr)\in\mathrm{Cor}^+(n).
\]
Now
\[
B_2:=\begin{pmatrix}0&h\\ h&0\end{pmatrix}
\qquad\text{satisfies}\qquad
B_2^2=h^2I_2,
\]
hence the power series of the matrix exponential gives
\[
\exp(B_2)
=
\begin{pmatrix}
\cosh h & \sinh h\\
\sinh h & \cosh h
\end{pmatrix}.
\]
Therefore, with
\[
D:=\bigl(-\log(\cosh h)\,I_2\bigr)\oplus 0_{n-2},
\]
one has
\[
\exp(B+D)
=
\exp\bigl(B_2-\log(\cosh h)\,I_2\bigr)\oplus I_{n-2}
=
\frac{1}{\cosh h}\exp(B_2)\oplus I_{n-2},
\]
that is,
\[
\exp(B+D)
=
\begin{pmatrix}
1 & \tanh h\\
\tanh h & 1
\end{pmatrix}\oplus I_{n-2}\in\mathrm{Cor}^+(n).
\]
By uniqueness of the diagonal correction, it follows that
\[
\mathcal D(B)
=
\bigl(-\log(\cosh h)\,I_2\bigr)\oplus 0_{n-2}.
\]
Consequently,
\[
P_{\mathrm{Cor}}^{\mathrm{ext}}(\Sigma_s)
=
\exp\bigl(B+\mathcal D(B)\bigr)
=
\begin{pmatrix}
1 & \tanh h\\
\tanh h & 1
\end{pmatrix}\oplus I_{n-2}.
\]

On the other hand,
\[
  R(\Sigma_s)=R(\Sigma_s^{(2)})\oplus I_{n-2},
\]
and $R(\Sigma_s^{(2)})$ varies with $s$. Hence $R(\Sigma_s)$ also varies with $s$.
Therefore $P_{\mathrm{Cor}}^{\mathrm{ext}}(\Sigma_s)\neq R(\Sigma_s)$ for generic $s$.
\end{proof}

\paragraph{Differential calculus for $\Phi$ and geodesics of $g^{\mathrm{ext}}$.}
We recall that $\Phi(\Sigma)=(S(\Sigma),\Delta(\Sigma))$ with
$S(\Sigma)=\mathrm{off}(\log\Sigma)$ and $\Delta(\Sigma)=\mathrm{Diag}(\log\Sigma)-\mathcal{D}(S(\Sigma))$.
The inverse diffeomorphism is
\begin{equation}\label{eq:Phi_inverse}
  \Phi^{-1}(S,\Delta)=\exp\bigl(S+\mathcal{D}(S)+\Delta\bigr).
\end{equation}

\begin{proposition}[Differentials of $\Phi$ and $\Phi^{-1}$]\label{prop:Phi_differentials}
Let $\Sigma\in\mathcal{S}^+(n)$ and set $A=\log\Sigma$, $S=\mathrm{off}(A)$ and $\Delta=\mathrm{Diag}(A)-\mathcal{D}(S)$.
For $U\in T_\Sigma\mathcal{S}^+(n)\simeq\mathcal{S}(n)$,
\begin{equation}\label{eq:dPhi}
  \mathrm{d}_\Sigma\Phi[U]
  =
  \Bigl(
    \mathrm{off}\bigl(\mathrm{d}_\Sigma\log[U]\bigr),\ 
    \mathrm{Diag}\bigl(\mathrm{d}_\Sigma\log[U]\bigr)-\mathrm{d}_S\mathcal{D}\Bigl[\mathrm{off}\bigl(\mathrm{d}_\Sigma\log[U]\bigr)\Bigr]
  \Bigr).
\end{equation}
Moreover, for $(\dot S,\dot\Delta)\in \mathrm{Hol}(n)\oplus \mathrm{Diag}(n)$,
\begin{equation}\label{eq:dPhi_inverse}
  \mathrm{d}_{(S,\Delta)}\Phi^{-1}[\dot S,\dot\Delta]
  =
  \mathrm{d}_{A}\exp\bigl[\dot S+\mathrm{d}_S\mathcal{D}[\dot S]+\dot\Delta\bigr],
  \qquad A=S+\mathcal{D}(S)+\Delta.
\end{equation}
In particular, using the Fr\'echet derivative formulas for the matrix exponential and logarithm
\[
  \mathrm{d}_{A}\exp[B]=\int_0^1 e^{(1-t)A}Be^{tA}\,\mathrm{d}t,
  \qquad
  \mathrm{d}_{\Sigma}\log[U]=\int_0^\infty (\Sigma+tI)^{-1}U(\Sigma+tI)^{-1}\,\mathrm{d}t,
\]
the differentials \eqref{eq:dPhi}--\eqref{eq:dPhi_inverse} are explicit.
\end{proposition}

\begin{proof}
Both identities follow from the chain rule, the definitions of $S(\Sigma)$ and $\Delta(\Sigma)$, and~\eqref{eq:Phi_inverse}.
\end{proof}

Since $\Phi$ is a global Riemannian isometry from $\bigl(\mathcal{S}^+(n),g^{\mathrm{ext}}\bigr)$
to the Euclidean space $\mathrm{Hol}(n)\oplus\mathrm{Diag}(n)$, the geodesics are straight lines in
$(S,\Delta)$-coordinates.

\begin{corollary}[Geodesics and distance for $g^{\mathrm{ext}}$]\label{cor:geodesics_ext}
Let $\Sigma_0,\Sigma_1\in\mathcal{S}^+(n)$ and write $\Phi(\Sigma_i)=(S_i,\Delta_i)$ for $i=1, \;2$. The unique minimizing $g^{\mathrm{ext}}$-geodesic joining $\Sigma_0$ to $\Sigma_1$ is
\[
  \gamma(t)=\Phi^{-1}\bigl((1-t)(S_0,\Delta_0)+t(S_1,\Delta_1)\bigr)
  =\exp\bigl(S_t+\mathcal{D}(S_t)+\Delta_t\bigr),
\]
where $S_t=(1-t)S_0+tS_1$ and $\Delta_t=(1-t)\Delta_0+t\Delta_1$.
Moreover,
\[
  d_{g^{\mathrm{ext}}}(\Sigma_0,\Sigma_1)^2=\|S_1-S_0\|_F^2+\|\Delta_1-\Delta_0\|_F^2.
\]
If $\Sigma_0,\Sigma_1\in\mathrm{Cor}^+(n)$ (so $\Delta_0=\Delta_1=0$), then $\gamma(t)\in\mathrm{Cor}^+(n)$ for all $t$:
the inclusion $\mathrm{Cor}^+(n)\hookrightarrow (\mathcal{S}^+(n),g^{\mathrm{ext}})$ is totally geodesic.
\end{corollary}

\begin{proof}
This is the pullback, through the global isometry $\Phi$, of the Euclidean geodesic segment in
$\mathrm{Hol}(n)\oplus\mathrm{Diag}(n)$.
\end{proof}

\paragraph{A global isometry with the standard log-Euclidean geometry.}
Both $g^{\mathrm{ext}}$ and $g^{\mathrm{LE}}$ are log-Euclidean metrics on $\mathcal{S}^+(n)$, hence both are
globally flat and globally isometric to Euclidean space via their defining diffeomorphisms. The orthogonal decomposition $\mathcal{S}(n)=\mathrm{Hol}(n)\oplus\mathrm{Diag}(n)$ suggests an explicit identification between these two charts.

\begin{proposition}[Explicit isometry between $\bigl(\mathcal{S}^+(n),g^{\mathrm{ext}}\bigr)$ and $\bigl(\mathcal{S}^+(n),g^{\mathrm{LE}}\bigr)$]\label{prop:isometry_ext_to_LE}
Define the linear isometry $L:\mathrm{Hol}(n)\oplus\mathrm{Diag}(n)\to\mathcal{S}(n)$ by $L(S,\Delta)=S+\Delta$.
Then the map
\begin{equation}\label{eq:J_def}
  J:=\exp\circ L\circ \Phi:\bigl(\mathcal{S}^+(n),g^{\mathrm{ext}}\bigr)\longrightarrow \bigl(\mathcal{S}^+(n),g^{\mathrm{LE}}\bigr)
\end{equation}
is a Riemannian isometry, with explicit formulas
\begin{align}
  J(\Sigma)
  &=\exp\bigl(\log\Sigma-\mathcal{D}(\mathrm{off}(\log\Sigma))\bigr),\label{eq:J_formula}\\
  J^{-1}(Y)
  &=\exp\bigl(\log Y+\mathcal{D}(\mathrm{off}(\log Y))\bigr).\label{eq:Jinv_formula}
\end{align}
Moreover,
\[
  J\bigl(\mathrm{Cor}^+(n)\bigr)=\exp\bigl(\mathrm{Hol}(n)\bigr)=:K,
\]
the canonical horizontal section introduced in Theorem~\ref{thm:can_section_appli}, and $J$ restricts to an isometry
$(\mathrm{Cor}^+(n),g^{\mathrm{OL}})\cong (K,g^{\mathrm{LE}}|_K)$.
\end{proposition}

\begin{proof}
The map $\Phi$ is a global isometry from $(\mathcal{S}^+(n),g^{\mathrm{ext}})$ to $\mathrm{Hol}(n)\oplus\mathrm{Diag}(n)$
by definition of $g^{\mathrm{ext}}$, and $\log$ is a global isometry from $(\mathcal{S}^+(n),g^{\mathrm{LE}})$ to $\mathcal{S}(n)$
by definition of the usual log-Euclidean metric. Since $L$ is a linear isometry, the composition~\eqref{eq:J_def} is a Riemannian isometry.

Using $\Phi(\Sigma)=(S,\Delta)$ with $S=\mathrm{off}(\log\Sigma)$ and $\Delta=\mathrm{Diag}(\log\Sigma)-\mathcal{D}(S)$ yields
$L\circ\Phi(\Sigma)=S+\Delta=\log\Sigma-\mathcal{D}(\mathrm{off}(\log\Sigma))$, proving~\eqref{eq:J_formula}.
For~\eqref{eq:Jinv_formula}, let $Y\in\mathcal{S}^+(n)$ and set $A:=\log Y\in\mathcal{S}(n)$.
Since
\[
J^{-1}=\Phi^{-1}\circ L^{-1}\circ \log,
\]
we have
\[
J^{-1}(Y)=\Phi^{-1}\bigl(L^{-1}(A)\bigr).
\]
Now
\[
L^{-1}(A)=\bigl(\mathrm{off}(A),\mathrm{Diag}(A)\bigr),
\]
and
\[
\Phi^{-1}(S,\Delta)=\exp\bigl(S+\mathcal{D}(S)+\Delta\bigr).
\]
Therefore
\begin{align*}
J^{-1}(Y)
&=\Phi^{-1}\bigl(\mathrm{off}(A),\mathrm{Diag}(A)\bigr)\\
&=\exp\bigl(\mathrm{off}(A)+\mathcal{D}(\mathrm{off}(A))+\mathrm{Diag}(A)\bigr)\\
&=\exp\bigl(A+\mathcal{D}(\mathrm{off}(A))\bigr)\\
&=\exp\bigl(\log Y+\mathcal{D}(\mathrm{off}(\log Y))\bigr),
\end{align*}
which proves~\eqref{eq:Jinv_formula}.
Finally, if $C\in\mathrm{Cor}^+(n)$ then $\Phi(C)=(\mathrm{Log}(C),0)$, hence
$J(C)=\exp(\mathrm{Log}(C))\in \exp(\mathrm{Hol}(n))=K$. Conversely any element of $K$ is of the form $\exp(S)$
with $S\in\mathrm{Hol}(n)$, and $J^{-1}(\exp(S))=\exp(S+\mathcal{D}(S))\in\mathrm{Cor}^+(n)$.
\end{proof}

\paragraph{Nested sequences of dimension-raising isometric embeddings.}
The isometries constructed in Section~\ref{sec:quotient_appli} combine with Proposition~\ref{prop:isometry_ext_to_LE} to produce an explicit sequence of totally geodesic isometric embeddings between log-Euclidean SPD manifolds. This yields, in turn, a parallel sequence for full-rank correlation manifolds endowed with the off-log metric.

\begin{theorem}[Nested sequence of log-Euclidean SPD matrices]\label{thm:SPD_hierarchy}
For each $n\ge 2$, let $\Phi_{\mathrm{OL}}^{(n)}:(\mathcal{S}^+(n-1),g^{\mathrm{LE}})\to (\mathrm{Cor}^+(n),g^{\mathrm{OL}})$
denote the isometric isomorphism constructed in Section~\ref{off_log_isometry}, built from the linear isometry $\psi_{\mathrm{OL}}^{(n)}:\bigl(\mathcal{S}(n\!-\!1),\langle\cdot,\cdot\rangle\bigr)\to\bigl(\mathrm{Hol}(n),\langle\cdot,\cdot\rangle\bigr)$, and let
$i^{(n)}:\mathrm{Cor}^+(n)\hookrightarrow \mathcal{S}^+(n)$ be the standard inclusion. Define the map
\begin{equation}\label{eq:Fn_def}
  F_n\;:=\;J^{(n)}\circ i^{(n)}\circ \Phi_{\mathrm{OL}}^{(n)}
  :\bigl(\mathcal{S}^+(n-1),g^{\mathrm{LE}}\bigr)\longrightarrow \bigl(\mathcal{S}^+(n),g^{\mathrm{LE}}\bigr),
\end{equation}
where $J^{(n)}$ is the isometry~\eqref{eq:J_def} on $\mathcal{S}^+(n)$.
Then $F_n$ is a totally geodesic isometric embedding and admits the closed-form expression
\begin{equation}\label{eq:Fn_formula}
  F_n(X)=\exp\bigl(\psi_{\mathrm{OL}}^{(n)}(\log X)\bigr)\in K^{(n)}:=\exp\bigl(\mathrm{Hol}(n)\bigr)\subset\mathcal{S}^+(n).
\end{equation}
Consequently, for every $m<n$ the composition $F_n\circ\cdots\circ F_{m+1}$ yields a totally geodesic isometric embedding
\[
  \bigl(\mathcal{S}^+(m),g^{\mathrm{LE}}\bigr)\hookrightarrow \bigl(\mathcal{S}^+(n),g^{\mathrm{LE}}\bigr),
\]
and the family $\bigl(\mathcal{S}^+(n),g^{\mathrm{LE}}\bigr)_{n\ge 1}$ forms a nested sequence of totally geodesic isometric embeddings.
\end{theorem}

\begin{proof}
Each map in~\eqref{eq:Fn_def} is an isometry onto its image: $\Phi_{\mathrm{OL}}^{(n)}$ is a Riemannian isometry by construction,
$i^{(n)}$ is an isometric (totally geodesic) embedding for $g^{\mathrm{ext}}$ by Theorem~\ref{thm:ambient_extension},
and $J^{(n)}$ is an isometry by Proposition~\ref{prop:isometry_ext_to_LE}. Hence $F_n$ is an isometric embedding; since
$i^{(n)}$ is totally geodesic and $J^{(n)}$ is an isometry, $F_n$ is totally geodesic as well.

To obtain~\eqref{eq:Fn_formula}, recall that
$\Phi_{\mathrm{OL}}^{(n)}(X)=\exp\bigl(\mathcal{D}(\psi_{\mathrm{OL}}^{(n)}(\log X))+\psi_{\mathrm{OL}}^{(n)}(\log X)\bigr)$.
For such a point in $\mathrm{Cor}^+(n)$ one has $\Phi\bigl(\Phi_{\mathrm{OL}}^{(n)}(X)\bigr)=\bigl(\psi_{\mathrm{OL}}^{(n)}(\log X),0\bigr)$,
hence $J^{(n)}\circ i^{(n)}\circ \Phi_{\mathrm{OL}}^{(n)}(X)=\exp(\psi_{\mathrm{OL}}^{(n)}(\log X))$ by applying each component of~\eqref{eq:J_def}.
\end{proof}

\begin{theorem}[Nested sequence of off-log full-rank correlation matrices]\label{thm:Cor_hierarchy}
For each $n\ge 2$, define the map
\begin{equation}\label{eq:En_def}
  E_n\;:=\;\Phi_{\mathrm{OL}}^{(n+1)}\circ F_n\circ \bigl(\Phi_{\mathrm{OL}}^{(n)}\bigr)^{-1}
  :\bigl(\mathrm{Cor}^+(n),g^{\mathrm{OL}}\bigr)\longrightarrow \bigl(\mathrm{Cor}^+(n+1),g^{\mathrm{OL}}\bigr),
\end{equation}
whose composites are defined in Section~\ref{off_log_isometry} and Theorem~\ref{thm:SPD_hierarchy}. Then $E_n$ is a totally geodesic isometric embedding.
In off-log coordinates, if $C\in\mathrm{Cor}^+(n)$ and $S=\mathrm{Log}(C)\in\mathrm{Hol}(n)$, then
\begin{equation}\label{eq:En_log_formula}
  \mathrm{Log}\bigl(E_n(C)\bigr)=\iota_n(S),
  \qquad
  \iota_n:=\psi_{\mathrm{OL}}^{(n+1)}\circ\bigl(\psi_{\mathrm{OL}}^{(n)}\bigr)^{-1}:\mathrm{Hol}(n)\hookrightarrow \mathrm{Hol}(n+1),
\end{equation}
and thus
\[
  E_n(C)=\mathrm{Exp}\bigl(\iota_n(\mathrm{Log}(C))\bigr)=\exp\bigl(\iota_n(\mathrm{Log}(C))+\mathcal{D}(\iota_n(\mathrm{Log}(C)))\bigr).
\]
Consequently, for every $m<n$ the compositions $E_n\circ\cdots\circ E_{m}$ provide a totally geodesic isometric embedding
\[
  \bigl(\mathrm{Cor}^+(m),g^{\mathrm{OL}}\bigr)\hookrightarrow \bigl(\mathrm{Cor}^+(n+1),g^{\mathrm{OL}}\bigr).
\]
\end{theorem}

\begin{proof}
The map $E_n$ is a composition of Riemannian isometries (onto their images):
$\bigl(\Phi_{\mathrm{OL}}^{(n)}\bigr)^{-1}$ is an isometry,
$F_n$ is a totally geodesic isometric embedding by Theorem~\ref{thm:SPD_hierarchy},
and $\Phi_{\mathrm{OL}}^{(n+1)}$ is an isometry.
Hence $E_n$ is a totally geodesic isometric embedding.
The expression~\eqref{eq:En_log_formula} follows by writing the three maps in log and off-log coordinates and using
$F_n(X)=\exp(\psi_{\mathrm{OL}}^{(n)}(\log X))$.
\end{proof}

\begin{corollary}[Cross-dimensional log-Euclidean and off-log distances]\label{cor:cross_dim_distances}
For integers $1\le m\le n$, define the iterated embeddings
\[
F_{m\to n}:=
\begin{cases}
\mathrm{id}_{\mathcal S^+(n)}, & m=n,\\
F_n\circ F_{n-1}\circ\cdots\circ F_{m+1}, & m<n,
\end{cases}
\]
and, for integers $2\le m\le n$, define similarly
\[
E_{m\to n}:=
\begin{cases}
\mathrm{id}_{\mathrm{Cor}^+(n)}, & m=n,\\
E_{n-1}\circ E_{n-2}\circ\cdots\circ E_m, & m<n.
\end{cases}
\]
Then each $F_{m\to n}$ and $E_{m\to n}$ is a totally geodesic isometric embedding.
If $d_n^{\mathrm{LE}}$ denotes the geodesic distance of $g^{\mathrm{LE}}$ on $\mathcal S^+(n)$, and $d_n^{\mathrm{OL}}$ the geodesic distance of $g^{\mathrm{OL}}$ on $\mathrm{Cor}^+(n)$, then for any
$X\in \mathcal S^+(m)$, $Y\in \mathcal S^+(n)$ and any $N\ge \max\{m,n\}$, the quantity
\[
\delta^{\mathrm{LE}}(X,Y):=d_N^{\mathrm{LE}}\bigl(F_{m\to N}(X),F_{n\to N}(Y)\bigr)
\]
is independent of $N$.
Likewise, for any $C\in \mathrm{Cor}^+(m)$, $D\in \mathrm{Cor}^+(n)$ and any $N\ge \max\{m,n\}$, the quantity
\[
\delta^{\mathrm{OL}}(C,D):=d_N^{\mathrm{OL}}\bigl(E_{m\to N}(C),E_{n\to N}(D)\bigr)
\]
is independent of $N$.
In particular, if $m\le n$, then
\[
\delta^{\mathrm{LE}}(X,Y)=d_n^{\mathrm{LE}}\bigl(F_{m\to n}(X),Y\bigr),
\qquad
\delta^{\mathrm{OL}}(C,D)=d_n^{\mathrm{OL}}\bigl(E_{m\to n}(C),D\bigr).
\]
Thus the nested hierarchies induce natural distances on the disjoint unions
\[
\bigsqcup_{n\ge 1}\mathcal S^+(n)
\qquad\text{and}\qquad
\bigsqcup_{n\ge 2}\mathrm{Cor}^+(n),
\]
thereby providing a geometrically consistent way to compare SPD matrices and full-rank correlation matrices of different sizes.
\end{corollary}

\begin{proof}
By Theorems~\ref{thm:SPD_hierarchy} and~\ref{thm:Cor_hierarchy}, every map $F_k$ and $E_k$ is a totally geodesic isometric embedding, hence so are their compositions.
Now let $N'\ge N\ge \max\{m,n\}$. Using the composition rule
$F_{m\to N'}=F_{N\to N'}\circ F_{m\to N}$ and the fact that $F_{N\to N'}$ is an isometric embedding, we obtain
\[
\begin{aligned}
 d_{N'}^{\mathrm{LE}}\bigl(F_{m\to N'}(X),F_{n\to N'}(Y)\bigr)
 &=d_{N'}^{\mathrm{LE}}\bigl(F_{N\to N'}(F_{m\to N}(X)),F_{N\to N'}(F_{n\to N}(Y))\bigr)\\
 &=d_N^{\mathrm{LE}}\bigl(F_{m\to N}(X),F_{n\to N}(Y)\bigr).
\end{aligned}
\]
This proves that $\delta^{\mathrm{LE}}(X,Y)$ does not depend on the chosen common dimension $N$.
The proof for $\delta^{\mathrm{OL}}$ is identical, replacing $F$ by $E$ and $d^{\mathrm{LE}}$ by $d^{\mathrm{OL}}$.
The simplified formulas for $m\le n$ are obtained by taking $N=n$.
\end{proof}

\paragraph{Practical interpretation.}
Corollary~\ref{cor:cross_dim_distances} provides a natural procedure for comparing matrices of different dimensions. One first embeds the lower-dimensional object isometrically into a common higher-dimensional space, and then computes the usual geodesic distance there. Because the embeddings are isometric and totally geodesic, this comparison is compatible with the intrinsic log-Euclidean and off-log metrics, and therefore preserves the corresponding geodesics, Fr\'echet means, and first-order statistics.

\begin{table}[ht]
\small
\centering
\resizebox{\textwidth}{!}{%
\begin{tabular}{@{}c c P{5.1cm} P{5.7cm}@{}}
\toprule
\textbf{Symbol} & \textbf{Manifold} & \textbf{Definition / Construction} & \textbf{Relations} \\ \midrule
$g^{\mathrm{LE}}$ & $S^{+}(n)$ &
Pullback of Frobenius inner product by matrix $\log$ &
Orthogonal split $g^{\mathrm{LE}} = g^{\mathrm{OL}}\oplus g^{\mathrm{DL}}$ \\[4pt]

$g^{\mathrm{ext}}$ & $S^{+}(n)$ &
Pullback of Frobenius inner product by the chart
$\Phi(\Sigma)=\bigl(\mathrm{off}(\log\Sigma),\,\mathrm{Diag}(\log\Sigma)-\mathcal{D}(\mathrm{off}(\log\Sigma))\bigr)$ &
Restricts to $g^{\mathrm{OL}}$ on $\operatorname{Cor}^{+}(n)$; moreover
$J:(S^{+}(n),g^{\mathrm{ext}})\to(S^{+}(n),g^{\mathrm{LE}})$ is an isometry \\[4pt]

$g^{\mathrm{OL}}$ & $S^{+}(n)$ (off-diag) or $\operatorname{Cor}^{+}(n)$ &
“Off-log’’ part of $g^{\mathrm{LE}}$ (kills diagonals) &
$g^{\mathrm{LE}}(X^\mathcal{H},Y^\mathcal{H}) = g^Q(\mathrm d \pi\, X^\mathcal{H}, \mathrm d \pi\, Y^\mathcal{H})$ \\[4pt]

$g^{\mathrm{DL}}$ & $S^{+}(n)$ (vertical) &
“Diag-log’’ part of $g^{\mathrm{LE}}$ (purely diagonal) &
Appears in $\sigma_{\text{corr}}^{*}g^{\mathrm{LE}}$ as vertical term \\[4pt]

$g^{Q}=g^{\overline{\mathrm{OL}}}$ &
$S^{+}(n)/\operatorname{Diag}^{+}(n)$ &
Quotient metric for the submersion $\pi$ &
Equals off-log metric on the quotient \\[4pt]

$g^{\mathrm{LE}}_{\!\vert\,\operatorname{Cor}^{+}(n)}$ &
$\operatorname{Cor}^{+}(n)$ &
Metric induced by inclusion &
$g^{\mathrm{LE}}_{\!\vert\,\operatorname{Cor}^{+}(n)}=g^{\mathrm{OL}}+g^{\mathrm{DL}}$ \\[4pt]

$g^{\mathrm{LE}}_{\!\vert\,K}=s^{*}g^{\mathrm{LE}}$ &
$K=\operatorname{im}s$ (canonical section) &
Metric induced on the canonical horizontal section &
$g^{\mathrm{LE}}_{\!\vert\,K}=s^{*}g^{\mathrm{LE}}=g^{Q}=g^{\overline{\mathrm{OL}}}$ \\[4pt]

$\sigma_{\text{corr}}^{*}g^{\mathrm{OL}}$ &
$S^{+}(n)/\!\operatorname{Diag}^{+}(n)$ &
Pullback of $g^{\mathrm{OL}}$ along correlation section &
$\sigma_{\text{corr}}^{*}g^{\mathrm{OL}}=g^{Q}$ \\[4pt]

$\sigma_{\text{corr}}^{*}g^{\mathrm{LE}}$ &
$S^{+}(n)/\!\operatorname{Diag}^{+}(n)$ &
Pullback of $g^{\mathrm{LE}}$ along correlation section &
$\sigma_{\text{corr}}^{*}g^{\mathrm{LE}}=g^{Q}+\sigma_{\text{corr}}^{*}g^{\mathrm{DL}}$ \\[4pt]

$g^{\mathrm{LS}}$ &
$\operatorname{Cor}^{+}(n)$ &
Log-scaling metric &
$(\pi\circ\Phi^\mathrm{LS-OL})^{*}g^{Q}=g^{\mathrm{LS}}$ \\ 
\bottomrule
\end{tabular}}
\caption{Log-Euclidean–related metrics appearing in Section \ref{sec:quotient_appli}}
\end{table}

% Bibliography

\clearpage
\bibliographystyle{alpha}
\bibliography{references}

\end{document}